\documentclass[leqno]{amsart}
\usepackage{amsmath, amstext, amssymb, amsthm, a4wide, esint}
\usepackage{epic,eepic,epsf,graphicx,color,enumitem,microtype}
\usepackage{hyperref,pdfsync}
\usepackage{pgf,tikz}
\usepackage{mathrsfs,amsmath}
\usetikzlibrary{arrows}

\usepackage{geometry}
\newtheorem{theorem}{Theorem}[section]

\newtheorem{lemma}[theorem]{Lemma}
\newtheorem{proposition}[theorem]{Proposition}
\newtheorem{corollary}[theorem]{Corollary}
\theoremstyle{definition}
\newtheorem{definition}[theorem]{Definition}
\newtheorem{assumption}{Assumption}

\theoremstyle{remark}
\newtheorem{remark}{Remark}

\newcommand{\Kappa}{\mathcal K}
\newcommand{\card}{{\mathbf{\#}}}
\newcommand{\R}{\mathbb R}
\newcommand{\N}{\mathbb N}
\renewcommand{\d}{\mathrm{d}}

\numberwithin{equation}{section}

\title[On the derivation of a Stokes-Brinkman problem]{On the derivation of a Stokes-Brinkman problem from Stokes equations around a random array of moving spheres}
\author{Kleber Carrapatoso \& Matthieu Hillairet}
\date{\today}

\address{Institut Montpelli\'erain Alexander Grothendieck, CNRS, Univ. Montpellier}
\email{kleber.carrapatoso@umontpellier.fr}
\email{matthieu.hillairet@umontpellier.fr}

\begin{document}

\definecolor{ffqqqq}{rgb}{1,0,0}
\definecolor{dtsfsf}{rgb}{0.8274509803921568,0.1843137254901961,0.1843137254901961}
\definecolor{ttqqqq}{rgb}{0.2,0,0}
\definecolor{rvwvcq}{rgb}{0.08235294117647059,0.396078431372549,0.7529411764705882}
\definecolor{wrwrwr}{rgb}{0.3803921568627451,0.3803921568627451,0.3803921568627451}
\definecolor{dtsfsf}{rgb}{0.8274509803921568,0.1843137254901961,0.1843137254901961}
\definecolor{cqcqcq}{rgb}{0.7529411764705882,0.7529411764705882,0.7529411764705882}

\begin{abstract}	
We consider the Stokes system in $\mathbb R^3,$ deprived of $N$ spheres of radius $1/N,$  completed by constant boundary conditions on the spheres. {This problem models the instantaneous response of a viscous fluid to an immersed cloud of moving solid spheres}. We assume that the centers of the spheres and the boundary conditions are given randomly and we compute the asymptotic behavior of solutions when the parameter $N$ diverges. Under the assumption that the distribution of spheres/centers is chaotic, we prove convergence in mean to the solution of a Stokes-Brinkman problem. 
\end{abstract}

\maketitle

\tableofcontents

\section{Introduction}

This paper is a contribution to a rigorous justification of mesoscopic models for the motion of a cloud of solid particles in a
viscous fluid. As explained in \cite{Desvillettes10}, the modeling of particle suspensions can borrow to different areas of partial differential equations. If the cloud contains few particles, the behavior of particles can be modeled by a finite dimensional system
and the coupling with the fluid equations yields a fluid/solid problem similar to the ones studied in \cite{Desjardins&Esteban99,Desjardins&Esteban00,GlassSueur,Takahashi03} for example. If the number of particle increases, a description of the particle phase {\em via} its individuals seems irrelevant. Depending on the volume fraction of the particle phase it is then necessary to turn 
to a kinetic/fluid description (as in \cite{BarangerDesvillettes} or \cite{BDGM}) or a multiphase description (see \cite{Ishii}).

In the case of a kinetic/fluid description, a system -- that we can find in references -- is the following Vlasov--Navier-Stokes system:
\begin{eqnarray*}
\partial_t f + v \cdot \nabla_x f +  6\pi \text{ div}_v [(u-v) f]  &=& 0 \,, \\
(\partial_t u + u \cdot \nabla_x u)  &=& \Delta_x u - \nabla_x p - 6\pi \int_{\mathbb R^3} f (u-v)\text{ d}v \,,	\\
\text{ div}_x u &=&0\,.
\end{eqnarray*}
Here  we introduce $f:(t,x,v) \in [0,\infty) \times \mathbb R^3 \times \mathbb R^3 \to [0,\infty)$ the particle distribution function which counts the proportion of particles at time $t$ which are in position $x \in \mathbb R^3$ and have velocity $v \in \mathbb R^3.$ This unknown encodes the cloud behavior. We emphasize that $v$ is a parameter of $f,$ hence the notations with indices to express with respect to which variable we differentiate. The two other unknowns $(u,p)$ represent respectively the fluid velocity-field and pressure. One recognizes in the two last equations
Navier-Stokes like equations. For simplicity, we do not include physical parameters such as the fluid density and viscosity. A particular feature of this model is the supplementary term
\begin{equation} \label{eq_brinkman}
6\pi \int_{\mathbb R^3} f (u-v)\text{ d}v\,,
\end{equation}
that appears on the right-hand side of the momentum equation.  It is supposed to model the exchange of momentum between the solid phase and the fluid. { As emphasized in \cite{Allaire} this supplementary term occurs when the product "number of particles" times "radius of particles" is of order $1$. The explicit form above} can be justified with the following formal reasoning. Assume that the
cloud is made of $N$ identical spheres of radius $1/N.$ If the particles are sufficiently spaced, they interact with the fluid as if they where alone: at its own scale, the particle $i$ moves with its velocity $v_i$ in  a viscous fluid whose velocity at infinity is $u(h_i).$ Stokes' law entails that fluid viscosity is responsible of the drag force:
$$
F_i  =  \dfrac{6\pi}{N} (v_i - u(h_i)) . 
$$ 
This term corresponds to the forcing term in the Vlasov equation and the corresponding term  \eqref{eq_brinkman} in the Navier-Stokes system is obtained by assuming that the forces induced by the $N$ particles can be superposed. 

\medskip

We are interested here in a rigorous approach to the above formal reasoning. This supposes to start from the
fluid/solid problem, where the particle dynamics equations are solved individually, and let the number of
particles diverge with their radius and density given by a suitable scaling. This question mixes large particle system problems (justification of the Vlasov equations starting from a system of ODEs) with fluid homogenization issues
(computing a macroscopic equation for the fluid unknowns).
The full problem seeming still out of reach now, we focus here on the fluid homogenization part.
Namely, one assumes that the particle behavior is given and wants to compute the new term in the fluid equation which takes
into account the influence of the particles. Since this term is due to fluid viscosity, we restrict to the Stokes system ({\em i.e.} 
the system obtained by neglecting the full time derivative on the left hand side of the momentum fluid equation).
Then, the problem reduces to homogenizing the Stokes problem in a perforated domain with non-zero boundary conditions
(mimicking the particle translation). This particular homogenization problem has been the subject of recent publications
(see \cite{DGR,Hil17,HMS,MH}). Therein, the limit Stokes system including the Brinkman term \eqref{eq_brinkman} is obtained
under specific dilution assumption of the particle phase. One further step toward tackling the time-dependent 
problem is then to discuss whether the set of favorable configurations -- {\em i.e.} such that the Brinkman term \eqref{eq_brinkman} appears in the limit -- is sufficiently large.  
To this end, we propose here to derive the Stokes-Brinkman problem via a Liouville approach in the spirit of \cite{Rubinstein}. More precisely, we first pick at random $N$ identical spherical particles/obstacles of radius $1/N$, each of them being characterized by its center of mass and its velocity, under the constraint that particles do not intersect each other. {We assume that the cloud of particles lies within a bounded open subset $\Omega_0$ of $\mathbb R^3.$}
We then consider a fluid occupying the whole space $\R^3$ deprived of these particles and satisfying a stationary Stokes equation with Dirichlet boundary condition at the boundary of each particle given by its velocity. Our aim is to rigorously derive the Stokes-Brinkman equation as an effective equation of the above problem in the limit $N \to \infty$.

Let us describe the problem in details. 
To begin with, fix $N \in \mathbb N^*$ arbitrary large and consider the experiment of dropping  randomly $N$ spheres of radius $1/N$ in the whole space $\mathbb R^3.$ Since the radius of the spheres is very small in comparison with their number (note that the volume fraction occupied by the spheres is typically of size $1/N^{2}$), we adapt a model that is classical for large point-particle systems.  We denote  
$$
\mathcal O^{N} := \Bigl\{ ((X_1^N,V_1^N),\ldots,(X_N^N,V_N^N))  \in [\mathbb R^{3} \times \mathbb R^{3}]^{N} \text{ s.t. }  |X_i^N - X_j^N| > \dfrac{2}{N} \quad \forall \, i \neq j \Bigr\}.
$$   
This represents the set of admissible configurations for the centers of mass $\mathbf X^N = (X_1^N,\ldots,X_N^N) $ and  velocities $\mathbf V^N = (V_1^N,\ldots,V_N^N).$ In what follows, we also denote $Z_i = (X_i,V_i)$
the state variable for the particle $i$ and keep bold symbols for $N$-component entities. For instance, we denote $\mathbf Z^N =((X^N_1,V_1^N),\ldots,(X_N^N,V_N^N)) \in \mathcal O^N$ a configuration.

The configuration of particles $\mathbf Z^N$ will be chosen at random under some law 
$F^N \in  \mathcal P (\mathcal O^N)$, where we denote by $\mathcal P (E)$ the space of probability measures on $E$. We assume that this probability measure is
absolutely continuous w.r.t.\ the Lebesgue measure and also denote by $F^N$ its density. Moreover, since the particles are indistinguishable, we shall assume that $\mathbf Z^N$ is an exchangeable random variable, which means that its law $F^N$ is symmetric, that is, for any permutation $\sigma \in \mathfrak{S}_N$ there holds
$$
F^N(Z_1^N,\ldots,Z_N^N) = F^N(Z_{\sigma(1)}^N,\ldots,Z_{\sigma(N)}^N), 
\quad \forall \, \mathbf{Z}^{N} \in \mathcal O^N.
$$

Given a configuration $ \mathbf Z^N = ((X_1^N,V_1^N),\ldots,(X_N^N,V_N^N)) \in \mathcal O^N$ we introduce the perforated domain:
$$
\mathcal F^N=  \R^3  \setminus \bigcup_{i=1}^{N} \overline{B_{i}^{N}}\,, \quad \text{ where } B_i^N = B(X_i^N, \tfrac1N) \quad \forall \, i=1,\ldots,N ,
$$
and consider the following Stokes problem: 
\begin{equation} \label{eq_stokes}
\left\{
\begin{array}{rcl}
- \Delta {u} + \nabla {p} &=& 0  \\
{\rm div}\, {u} &= & 0 
\end{array}
\right.
\quad \text{ in $\mathcal F^N$},
\end{equation}
with boundary conditions
\begin{equation} \label{cab_stokes}
\begin{array}{rcll}
{u}(x) &=& V_i^N  &  \text{on } \partial B_i^{N} \text{ for }  i=1,\ldots,N, \\
\displaystyle \lim_{|x| \to \infty} |{u}(x)| &=& 0 .
\end{array} 
\end{equation} 
We obtain a stationary exterior problem in 3 dimensions. Such systems are
extensively studied in \cite[Section V]{Galdi} where it is proven for instance that 
there exists a unique solution $({u},{p})$ to \eqref{eq_stokes}-\eqref{cab_stokes}.
We may then construct:
$$
u[\mathbf Z^N](x)  = 
\left\{
\begin{array}{rl}
{u}(x), & \text{ if $x \in \mathcal F^N$}\, \\[4pt]
V_i^N, & \text{ if $x \in B_i^N$ for $i=1,\ldots,N$} .
\end{array}
\right.
$$
The above reference on the exterior problem entails that $u[\mathbf Z^N] \in \dot{H}^1(\mathbb R^3)$ (where we denote $\dot{H}^1(\mathbb R^3)$  the closure of $C^{\infty}_c(\mathbb R^3)$ for the $L^2$-norm of the gradient). Therefore, we construct the mapping 
\begin{equation}\label{def:UN}
\begin{array}{rrcl}
U_N : & \mathcal O^N &  \longrightarrow & \dot{H}^1(\mathbb R^3) \\[4pt]
	& \mathbf Z^N  & \longmapsto & u[\mathbf Z^N]\\
	\end{array}
\end{equation}
as a random variable on $\mathcal O^N$ endowed with the probability measure 
$F^N$. 

\medskip

At first in \cite{DGR}, it is shown that, for a given sequence $\mathbf Z^N$ satisfying some conditions and with prescribed asymptotic behavior when $N \to \infty,$ the associated solutions to \eqref{eq_stokes}-\eqref{cab_stokes} converge to 
a solution to the Stokes-Brinkman problem:
\begin{equation} \label{intro:SB}
\left\{
\begin{array}{rcl}
-\Delta \tilde{u} + \nabla \tilde{p} + 6\pi \rho \tilde{u} &=& 6\pi j  \\
{\rm div} \, \tilde{u} &=& 0  
\end{array}
\right.
\quad \text{ in $\mathbb R^3$},
\end{equation}
 with vanishing condition at infinity
\begin{equation} \label{intro:cab_SB}
\lim_{|x| \to \infty} |\tilde{u}(x)| = 0.
\end{equation}
In this system the flux $j$ and density $\rho$ are related to the asymptotic behavior of the $\mathbf Z^N.$
In this paper, we compute the flux $j$ and density $\rho$ depending on the asymptotic behavior of the law $F^N$ in order that the expectation of $U_N$ converges in a suitable sense to the same Stokes-Brinkman problem.
As we recall in the beginning of {Section~\ref{sec:bonneconfig}}, this system  is well-posed for positive $\rho \in L^{3/2}(\mathbb R^3)$ and $j \in L^{6/5}(\mathbb R^3).$
%

\subsection{Main result}
Our main result requires some conditions on the sequence of symmetric probability measures $(F^N)_{N \in \N^*}$ on $\mathcal O^N$. 
To state our conditions, we introduce the family of sets $\mathcal O^m[R]$ for an integer $m \geq 2$ and $R >0$ as defined by:
\[
\mathcal O^m[R] =  \Bigl\{ ((X_1,V_1),\ldots,(X_m,V_m))  \in [\mathbb R^{3} \times \mathbb R^{3}]^{m} \text{ s.t. }  |X_i - X_j| > 2R \quad \forall \, i \neq j \Bigr\}.
\]
We note that we have then $\mathcal O^N = \mathcal O^N[\tfrac{1}{N}]$ in particular. 
Then, the $m$-th marginal of $F^N$ is given by
$$
F^N_m(\mathbf{z}) = \int_{\mathbb R^{6(N-m)}} \mathbf{1}_{\mathbf{(\mathbf{z},\mathbf{z}')} \in \mathcal O^N} F^N(\mathbf{z},\mathbf{z}') \d\mathbf{z}' ,\qquad \forall \,  \mathbf{z} \in \mathcal O^m[\tfrac{1}{N}].
$$
Such marginals are constructed by remarking that, if we split an $N-$particle distribution by giving the $m$ first particle state $\mathbf{z}$ and the remaining $(N-m)$ particle state $\mathbf{z}'$ we must require that 
$\mathbf{z} \in \mathcal O^m[\frac{1}{N}]$ in order that $(\mathbf{z},\mathbf{z}') \in \mathcal O^{N}$ be possible. 
We apply here again with small letters the convention that
$z_i \in \mathbb R^6$ splits into $z_i = (x_i,v_i)$ and that bold symbols 
encode vectors of unknowns $x,$ $v$ or $z.$

We are now able to state our main assumptions. Let $(\mathbf Z^N)_{N \in \mathbb N^*}$ be a sequence of exchangeable $\mathcal O^N$-valued random variables, and let $(F^N)_{N \in \N^*}$ be the sequence of their associated laws, that is, symmetric probability measures on $\mathcal O^N$. 

\begin{assumption}\label{assump:A1}
We assume that $(F^N)_{N \in \mathbb N^*}$ are distribution functions, that is belong to $L^1(\mathcal O^N)$, and satisfy the following properties:
\begin{enumerate} \addtocounter{enumi}{-1}
\item $\mathrm{Supp}(F^N) \subset (\Omega_0 \times \mathbb R^3)^N$, for some bounded open $\Omega_0 \subset \mathbb R^3$ and any $N \in \mathbb N^*$.

\medskip

\item There exists a constant $C_1 \ge 1 $ such that for any $N \in \mathbb N^*$ and $1 \le m \le N$
$$
\| F^N_m \|_{L^\infty_x L^1_v \left(\mathcal O^m[\frac{1}{N}]\right)} 
 := \sup_{\mathbf{x} \in \mathbb R^{3m}} \int_{\mathbb R^{3m}} \mathbf 1_{ \mathbf z \, \in \mathcal O^m[\frac{1}{N}] } \, F^N_m(\mathbf z) \, \mathrm{d} \mathbf v
\le (C_1)^m.
$$

\item There exists $k_0 \ge 5$ and a constant $C_2 >0 $ such that 
$$
\sup_{N \in \mathbb N^*} \| |z_1|^{k_0} F^N_1 \|_{L^1_x L^1_v (\mathbb R^3 \times \mathbb R^3)}
=\sup_{N \in \mathbb N^*} \int_{\mathbb R^3 \times \mathbb R^3} |z_1|^{k_0} F^N_1( z_1) \,  \mathrm{d}  z_1 \le C_2.
$$

\item There exists a constant $C_3 >0 $ such that 
$$
\sup_{N \in \mathbb N^*} \| |v_1| F^N_2 \|_{L^\infty_x L^1_v \left(\mathcal O^2[\frac{1}{N}] \right)}
=\sup_{N \in \mathbb N^*} \sup_{ x_1,x_2} \int_{\mathbb R^6} \mathbf 1_{ (z_1,z_2) \in \mathcal O^2[\frac{1}{N}] } \, |v_1| F^N_2( z_1,z_2) \,  \mathrm{d} v_1 \mathrm{d} v_2 \le C_3.
$$
\end{enumerate}
\end{assumption}

In this set of assumptions, (2) corresponds to the classical assumption that the law has a sufficient number of bounded moments; (1) would be satisfied in particular by tensorized laws; (0) is reminiscent of the fact that the cloud occupies the bounded region $\Omega_0$ and (3) shall enable to control the interactions between close particles through the flow.

Given a sequence $(\mathbf Z^N)_{N \in \mathbb N^*}$ of exchangeable random variables on $\mathcal O^N$, we define the associated empirical measure by
\begin{equation}\label{empirical}
\mu^N [\mathbf Z^N] := \frac{1}{N} \sum_{i=1}^N \delta_{Z_i^N} ,
\end{equation}
as well as the empirical density and the empirical flux respectively by
\begin{equation}\label{empirical-rho-j}
\rho^N [\mathbf Z^N] = \rho^N[\mathbf X^N] := \frac{1}{N} \sum_{i=1}^N \delta_{X_i^N} , \quad
j^N [\mathbf Z^N] := \frac{1}{N} \sum_{i=1}^N V_i^N \delta_{X_i^N} .
\end{equation}
The first formula defines a standard probability measure while the second one is a vectorial measure on $\mathbb R^3.$

We now state our assumptions concerning the asymptotic behavior of the sequence of configurations.

\begin{assumption}\label{assump:A2}
Under Assumption~\ref{assump:A1}, we suppose that there is a probability measure $f$ on $\R^3 \times \R^3$ with support on $\Omega_0 \times \R^3$ such that, defining the probability measure $\rho (\d x) = \int_{\R^3} f(\d x , \d v)$ and the vectorial measure $j(\d x):= \int_{\R^3} v f(\d x, \d v)$ (both with support on $\Omega_0$), we have:

\begin{enumerate}[label={\normalfont(\roman*)},itemsep=4pt]	

\item $\displaystyle \lim_{N \to \infty} \mathbb E \left[W_1( \rho^N[\mathbf Z^N] , \rho ) \right] = 0$;

\item $\displaystyle \lim_{N \to \infty} \mathbb E \left[ \| j^N[\mathbf Z^N] - j \|_{[C^{0,1}_b(\mathbb R^3)]^*} \right] = 0$.
\end{enumerate}

\end{assumption}

We denote here $W_1$ for the Wasserstein distance (with cost $c(x,y) = |x-y|$) and $\| \cdot \|_{C^{0,1}_b(\R^3)]^*}$ for the dual norm of Lipschitz bounded functions on $\R^3$ (see Section~\ref{sec:assumption} below).


\begin{remark}
Given the random variable $\mathbf Z^N$ with law $F^N$, we can consider the random variable $\mathbf X^N$ on $\mathcal O^N_x := \{ (X^N_1, \dots , X^N_N) \in \R^{3N} \mid |X^N_i - X^N_j| > \tfrac{2}{N} \; \forall\, i \neq j   \}$ which has a symmetric law $R^N \in \mathcal P(\mathcal O^N_x)$, given by $R^N (\d \mathbf x^N ) = \int_{\R^{3N}}  F^N (\d \mathbf x^N, \d \mathbf v^N)$. Point (i) in Assumption~\ref{assump:A2} is equivalent to the fact that the sequence $(R^N)_{N \in \N^*}$ is $\rho$-chaotic (roughly speaking that $R^N$ is asymptoticly i.i.d.\ with law $\rho$, see Definition~\ref{def:chaos}) thanks to e.g.\ \cite{HM}.
\end{remark}

\begin{remark}
We will be interested in conditions on the sequence $(F^N)_{N \in \N^*}$ in order to ensure the convergences of Assumption~\ref{assump:A2}. In particular we will show in Lemma~\ref{lem:chaos-A2} that if the sequence $(F^N)_{N \in \N^*}$ is $f$-chaotic (see Definition~\ref{def:chaos}) then it satisfies Assumption~\ref{assump:A2}. (But clearly this is not 
a necessary condition.)
\end{remark}

%
%
%
%

With these notations, our main theorem reads:

\begin{theorem}\label{theo:main}
Let $f \in  L^1(\mathbb R^3 \times \mathbb R^3)$ be a probability measure having support in $\Omega_0 \times \mathbb R^3$ and define $ \rho(x) = \int_{\R^3} f(x,v) \, \mathrm{d}v$ and $j(x) = \int_{\R^3} v f(x,v) \, \mathrm{d}v.$ Assume that $\rho \in  L^3(\Omega_0)$  and $j \in L^{6/5}(\Omega_0)$ so that there exists a unique solution $(u,p) \in \dot{H}^1(\mathbb R^3) \times L^2(\mathbb R^3)$ to the Stokes-Brinkman problem \eqref{intro:SB}-\eqref{intro:cab_SB} associated to $\rho$ and $j$.
Consider a sequence of exchangeable random variables $(\mathbf Z^N)_{N \in \mathbb N^*}$ on $\mathcal O^N$ and their associated symmetric laws $(F^N)_{N \in \N^*}$ satisfying Assumption~\ref{assump:A1}.

Then, given $\alpha \in (2/3,1)$ and for $N$ large enough,
the map $U_N$ given by \eqref{def:UN} satisfies:
\begin{equation}\label{eq:theo-erreur-0}
\mathbb E \left[ \| U_N [\mathbf Z^N] - u \|_{L^2_{\mathrm{loc}} (\R^3)} \right] 
\lesssim \mathbb E \left[  W_1( \rho^N [\mathbf Z^N] , \rho )  \right]^{\frac{1}{57}}
+\mathbb E \left[  \| j^N [\mathbf Z^N] - j \|_{[C^{0,1}_b(\R^3)]^{*}}  \right]^{\frac{1}{3}}
+N^{-e_1(\alpha)} ,
\end{equation}
where $e_1(\alpha) = \min (\tfrac{1-\alpha}{95},\frac{(3 \alpha-2)}{2})$.

As a consequence, if $(F^N)_{N \in \N^*}$ satisfies moreover Assumption~\ref{assump:A2} (i)-(ii), then 
$$
\lim_{N \to \infty} \mathbb E \left[ \| U_N [\mathbf Z^N] - u \|_{L^2_{\mathrm{loc}} (\R^3)} \right] = 0.
$$

\end{theorem}

A key-point in the result of this theorem is that the right-hand side of \eqref{eq:theo-erreur-0} depends on powers of  
$\mathbb E \left[  W_1( \rho^N [\mathbf Z^N] , \rho )  \right]$ and $\mathbb E \left[  \| j^N [\mathbf Z^N] - j \|_{[C^{0,1}_b(\R^3)]^{*}}  \right]$, and on a residual power of $N$ (depending only on the parameter $\alpha$). 
We remark that both densities and flux differences estimates in \eqref{eq:theo-erreur-0} are in fact estimates of the same type, since here, for probability measures such as the densities, the Wasserstein distance $W_1$ is equivalent to the distance given by the $[C^{0,1}_b(\R^3)]^{*}$-norm. However, the fluxes $j^N[\mathbf Z^N]$ and $j$ are not probability measures (they do not even share the same mass {\em a priori}) so that $W_1$ is not a distance.
We emphasize  that the explicit values of our exponents need not be optimal in all contexts and that it is also possible to obtain a $L^p$ version of estimate \eqref{eq:theo-erreur-0} with different exponents, under the condition that $W^{2,p}$ embeds into some H\"older space.

A further result of our study (see Section~\ref{sec:UN-prop}), is that, with the assumptions of 
Theorem~\ref{theo:main}, $\mathbb E[U_N [\mathbf Z^N]]$ defines a bounded sequence in $\dot{H}^1(\mathbb R^3).$ 
Theorem~\ref{theo:main} then implies that this sequence converges (at least weakly in $\dot{H}^1(\mathbb R^3)$) to the solution to the Stokes-Brinkmann problem with the corresponding flux $j$ and density $\rho.$ This consequence is yet another hint that the Stokes-Brinkman problem \eqref{intro:SB}-\eqref{intro:cab_SB} is indeed the right macroscopic model to compute the behavior of a viscous fluid in presence of a cloud of moving particles under the asymptotic convergences of Assumption~\ref{assump:A2}.

To show one application of the previous theorem, we shall construct an explicit example of probability measure on $\mathcal O^N$ satisfying the assumptions of Theorem~\ref{theo:main} and for which we obtain a quantitative estimate of the convergence \eqref{eq:theo-erreur-0}.

\begin{corollary}\label{cor:main}
Let $f \in L^1(\R^3 \times \R^3)$ be a probability measure satisfying the hypotheses of Theorem~\ref{theo:main} and such that the associated density $\rho \in L^\infty(\Omega_0)$ and $\int_{\Omega_0 \times \R^3} |v|^k  f(\d z) $ for some $k \ge 5$. Then we can construct a sequence of symmetric probability measures $(F^N)_{N \in \N^*}$ on $\mathcal O^N$ satisfying Assumptions~\ref{assump:A1} and \ref{assump:A2}, and for which there holds
$$
\mathbb E \left[ \| U_N [\mathbf Z^N] - u \|_{L^2_{\mathrm{loc}} (\R^3)} \right] 
\lesssim  N^{-\frac{1}{171}}  +  N^{-e_1(\alpha)}.
$$

\end{corollary}

On the basis of computations in \cite{HMS}, we expect that the content of Theorem \ref{theo:main} can be extended to particles with arbitrary shapes and possibly rotating. We recall that, in this framework, the limit Stokes-Brinkman problem is related to the distribution of shapes for the particles in the cloud, that is quantified in terms of Stokes resistance matrix associated with these shapes. The particle rotations influence the effective model only {\em via} their contribution to the drag force exerted on the particles. We refer to \cite{HMS} for more details.

\subsection{Overview of the proof}
The proof of Theorem~\ref{theo:main} faces several difficulties. First, for fixed $N$, we must identify a sufficiently large set of data $\mathbf Z^N$ for which the solution $U_N[\mathbf Z^N]$ to the Stokes problem in the punctured domain is close to the solution to the Stokes-Brinkman problem. In comparison with \cite{DGR}, a key-difficulty is to have a quantified estimate at-hand. A second difficulty is that, since the velocities $V_i^N$ that we impose on the particles are arbitrary, the solution to the Stokes problem may diverge in $\dot{H}^1(\mathbb R^3)$ when two particles become close. It is then necessary to obtain a bound on the solution to the Stokes problem associated with these configurations in order to ensure that they won't perturb the computation of the limit in mean.

Having in mind these two important difficulties, we propose an approach that is divided into five steps that we explain in more details below:

\begin{itemize}

\item As a first step, we prove in Section~\ref{sec_propertiesdata} some estimates associated to the convergence of the sequence of configurations (the random variables $(\mathbf Z^N)_N$ and their laws $(F^N)_N$) with respect to the expected limit (the marginals $\rho$ and $j$ of the distribution $f$).

\medskip

\item We then identify some ``concentrated configurations'' and prove that they are negligible in the asymptotic limit $N \to \infty$. These configurations correspond to $\mathbf Z^N \in \mathcal O^N$ such that there exists a couple of particles too close to each other or that there exist too many particles in a same cell of small volume. This is done in Section~\ref{sec_stat}.

\medskip

\item Furthermore, we compute uniform estimates satisfied by the map $U_N[\mathbf Z^N]$. We obtain simultaneously that:
\begin{itemize}
\item the mean of $U_N[\mathbf Z^N]$ is well-defined
and uniformly bounded in $\dot{H}^1(\mathbb R^3)$; 
\item the weight of contribution of the concentrated configurations vanishes when $N \to \infty.$  
\end{itemize}
This enables to get rid of concentrated configurations in the asymptotic description of $U_N$. This step is treated in Section~\ref{sec:UN-prop}.

\medskip

\item In a further step, developed in Section~\ref{sec:bonneconfig}, we prove a mean-field result for non-concentrated configurations which  is the cornerstone of our proof. We combine here the duality method of \cite{MH} with covering arguments of \cite{Hil17}. In comparison with these previous references, we consider in this paper an unbounded container. So, these arguments need to be adapted carefully.

\medskip

\item Finally, in the last step presented in Section~\ref{sec:maintheo}, we gather previous estimates together in order to obtain Theorem~\ref{theo:main}. Furthermore, we construct a particular example of sequence of probability measures $(F^N)_N$ in order to obtain Corollary~\ref{cor:main}.

\end{itemize}

The paper ends with a series of appendices. In Appendix \ref{app_wi} are gathered the technical computations underlying the third step of the above analysis (corresponding to Section \ref{sec:UN-prop}). 
In Appendix \ref{app_stokes}, we give some material on the resolution of the Stokes problem in a punctured box. These results are used in Section \ref{sec:bonneconfig}. Finally,  in the last Appendix \ref{sec_constants} we 
provide also some computations of constants that are involved in Section \ref{sec:bonneconfig}.

\subsection*{Notations}
In this paper, we shall denote $A \lesssim B$ if there is some constant $C>0$ (insignificant to  our computation) such that $A \le C B$. We use the classical notations $L^p(\mathcal O)$ and $H^{m}(\mathcal O)$ for Lebesgue and Sobolev spaces. The space $\dot{H}^1(\mathbb R^3)$ will play a crucial role in the analysis. We recall that we can see this space as the closure of $C^{\infty}_c(\mathbb R^3)$ for the norm
\[
\|w\|_{\dot{H}^1(\mathbb R^3)} = \left( \int_{\mathbb R^3} |\nabla w|^{2}\right)^{\frac 12} ,\quad \forall \, w \in C^{\infty}_c(\mathbb R^2).
\] 
We also apply below constantly the Bogovskii operator \cite[Section III.3]{Galdi}.
We recall that this operator is constructed to lift a divergence. Namely, given $f \in L^p(\mathcal O)$ it creates (under some compatibility conditions on $f$) a vector-field $w \in W^{1,p}(\mathcal O)$ such that ${\rm div} w = f.$  Concerning the homogeneity properties of this operator we refer to \cite[Appendix A]{Hil17} among
others.

\medskip
\noindent
{\bf Acknowledgments.} K.C.\ thanks N. Fournier for fruitful discussions on empirical measures. K.C.\ was partially supported by the EFI project ANR-17-CE40-0030 and the KIBORD project ANR-13-BS01-0004 of the French National Research Agency (ANR). 
M.H. is supported by the IFSMACS project ANR-15-CE40-0010,
the Dyficolti project ANR-13-BS01-0003-01.


\section{Properties of the sequence of configurations}\label{sec:assumption}

In this section we gather some properties of the sequence of configurations $(\mathbf Z^N)_{N \in \N^*}$ on $\mathcal O^N$ under the sequence of associated laws $(F^N)_{N \in \N^*}$ satisfying Assumptions~\ref{assump:A1}. We recall that
\begin{equation}
\mathcal O^{N} := \Big\{ \mathbf Z^N \in (\mathbb R^3 \times \mathbb R^3)^N \mid |X_i - X_j| > \frac{2}{N} \quad \forall i \neq j  \Big\},
\end{equation}
where hereafter we shall use the Assumption~\ref{assump:A1}-(0) saying that $\mathrm{Supp}(F^N) \subset \Omega_0 \times \R^3$ for some bounded open set $\Omega_0 \subset \R^3$, and where we denote 
$$
\begin{aligned}
& \mathbf Z^N = (Z_1, \dots , Z_N) \in  (\mathbb R^3 \times \mathbb R^3)^N , \\
& \mathbf X^N = (X_1 , \dots , X_N) \in \mathbb R^{3N}  , \quad \mathbf V^N = (V_1 , \dots , V_N) \in \mathbb R^{3N} , \\
& Z_i = (X_i, V_i) \in \mathbb R^3 \times \mathbb R^3 .
\end{aligned}
$$
We shall denote by $\mathcal O^N_x$ the projection of the space of configurations $\mathcal O^N$ onto the $\mathbf X^N $-variables, more precisely
$$
\mathcal O^{N}_x := \Big\{ (X_1 , \dots , X_N) \in \mathbb R^{3N} \mid |X_i - X_j| > \frac{2}{N} \quad \forall i \neq j  \Big\},
$$
in such a way that $\mathcal O^{N} \simeq \mathcal O^{N}_x \times \R^{3N}$.


In the first part of this section, we focus on the convergence of the family of measures $(\rho^N[\mathbf Z^N])_{N\in\mathbb N^*}$ and $(j^N[\mathbf Z^N])_{N\in \mathbb N^*}$ seen as random variables. As mentioned in the introduction, we metrize the convergence of measures on $\mathbb R^3$ by two different topologies: either we see (by restriction) vectorial measures as bounded linear forms on H\"older spaces:
$$
C_b^{0,\theta}(\mathbb R^3) := \left\{ \varphi \in C(\mathbb R^3) \cap L^{\infty}(\mathbb R^3)\,, 
\text{ s.t.\ } \sup_{x \neq y} \dfrac{|f(x) - f(y)|}{|x-y|^{\theta}} < \infty\right\},
$$ or we use the (Monge-Kantorovich-)Wasserstein $W_1$-distance on probability measures. Hereafter, the $1$-Wasserstein distance $W_1(f,g)$, with $f$ and $g$ probability measures on $\mathbb R^3 \times \mathbb R^3$, is defined by (see e.g.\ \cite{Villani-OT})
\begin{equation}
\begin{aligned}
W_1 (g,f) 
&:= \inf_{\pi \in \Pi(g,f)} \int_{(\mathbb R^3 \times \mathbb R^3)^2} |z-z'|  \d\pi (z,z')   = \sup_{[\psi]_{\mathrm{Lip}} \le 1 } \,
\int_{\mathbb R^3 \times \mathbb R^3} \psi(z) \left( g(\d z)-f(\d z) \right)    ,
\end{aligned}
\end{equation}
with $\Pi (g,f)$ being the set of probability measures on $(\mathbb R^3 \times \mathbb R^3)^2$ whose first marginal equals $g$ and second marginal $f$, and $[\cdot ]_{\mathrm{Lip}}$ denotes the Lipschitz semi-norm
$$
[\psi]_{\mathrm{Lip}} := \sup_{z \neq z'} \frac{|\psi(z) - \psi(z')|}{|z-z'|}.
$$
Correspondingly, $[\cdot]_{C^{0,\theta}}$ with $0 < \theta \le 1$ stands for the $C^{0,\theta}$ semi-norm
$$
[\psi]_{C^{0,\theta}} := \sup_{z \neq z'} \frac{|\psi(z) - \psi(z')|}{|z-z'|^{\theta}}.
$$
and $\| \cdot \|_{C^{0,\theta}_b(\mathbb R^3)} := \| \cdot \|_{L^\infty(\mathbb R^3)} + [\cdot]_{C^{0,\theta}}$ the $C^{0,\theta}$-norm. We then define, for finite signed measures $m$ and $\bar m$ in $\mathbb R^3$, the dual metric $\| \cdot \|_{(C^{0,\theta}_b(\mathbb R^3))^*}$ by
\begin{equation}
\| m - \bar m \|_{(C^{0,\theta}_b (\R^3))^*}  
:= \sup_{\| \phi \|_{C^{0,\theta}_b(\mathbb R^3)} \le 1 } \,
\int_{\mathbb R^3 } \phi(z) \left( m(\d z)- \bar m(\d z) \right).
\end{equation}
Finally, for vectorial measures $j = (j_\alpha)_{1\le \alpha \le 3}$ and $\bar j = (\bar j_\alpha)_{1 \le \alpha \le 3}$ in $\mathbb R^3$, we define
\begin{equation}
\| j - \bar j \|_{(C^{0,\theta}_b (\R^3))^*}
:= \sum_{\alpha=1}^3 \| j_\alpha - \bar j_\alpha \|_{(C^{0,\theta}_b (\R^3))^*}.
\end{equation}

We remark here that, when dealing with probability measures $\rho, \bar \rho \in \mathcal P(\mathbb R^3)$ with support included in $\Omega_0$, the $W_1$ distance between $\rho$ and $\bar \rho$ is equivalent to the dual distance given by the $C^{0,1}_b$-norm of their difference, and we shall always use the former, which is of more common use.

In the second part of this section, we measure the weights of configurations in which the particles are concentrated, meaning that
the minimal distance between two particles is small or that there are too many particles in a small subset of $\mathbb R^3.$

\subsection{On the convergence Assumption~\ref{assump:A2}}  \label{sec_propertiesdata}
Let us describe some properties concerning the asymptotic convergence of the data, where we always assume that  Assumption~\ref{assump:A1} is in force. We first obtain some estimates for different metrics concerning the convergences of Assumption~\ref{assump:A2}, and then we give a sufficient condition on the sequence $(F^N)_{N \in \N^*}$ to satisfy Assumption~\ref{assump:A2}.

We recall below the notion of chaoticity for a sequence of probability measures, see \cite{Kac,Sznitman}.
\begin{definition}\label{def:chaos}
Let $E \subset \R^d$. Consider a sequence $(\mathbf Y^N)_{N \in \N^*}$ of exchangeable random variables on $E^N$ and the associated sequence of laws $(\pi^N)_{N \in \N^*}$, that are symmetric probability measures on $E^N$. 
We say that $(\pi^N)_{N \in \N^*}$ (or that $(\mathbf Y^N)_{N \in \N^*}$) is $\pi$-chaotic, for some probability measure $\pi$ on $E$, if one of the following equivalent conditions is fulfilled:
\begin{enumerate}[label=(\alph*),itemsep=4pt]
\item $\pi^N_m$ converges to $\pi^{\otimes m}$ weakly in $ \mathcal P(E^m)$ as $N\to\infty$ for any fixed $m \ge 1$ (or some $m \ge 2$);

\item the $\mathcal P (E)$-valued random variable $\mu^N [\mathbf Y^N]$ converges in law to $\pi$ as $N\to\infty$.

\end{enumerate}
Here $\pi^N_m$ denotes the $m$-marginal of $\pi^N$ given by $\pi^N_m (\d z_1, \dots, \d z_m) :=\int_{E^{N-m}} \pi^N (\d z_1, \dots, \d z_m, \d z_{m+1}, \dots , \d z_N)$, and $\mu^N [\mathbf Y^N] = \frac{1}{N} \sum_{i=1}^{N} \delta_{Y^N_i}$ is the empirical measure associated to $\mathbf Y^N$.

\end{definition} 
We remark that \cite{HM} obtains a quantitave version of the above equivalence. More precisely, assuming that $\pi^N_1$ possesses a finite moment of order $k > 1$, $(\pi^N)_{N \in \N^*}$ is $\pi$-chaotic is equivalent to
\begin{enumerate}[label=(\alph*'),itemsep=4pt]
\item $\displaystyle \lim_{N\to \infty} W_1 (\pi^N_m , \pi^{\otimes m} ) = 0$ for any fixed $m \ge 1$ (or some $m \ge 2$); 

\item $\displaystyle \lim_{N \to \infty} \mathbb E \left[ W_1 (\mu^N [\mathbf Y^N] , \pi)  \right] = 0$;
\end{enumerate}
with a quantitative estimate in $N$ for the equivalence between (a') and (b').
As a consequence of this, and arguing similarly for the case of finite vectorial measures (more precisely for finite signed measures, corresponding to each component of $j^N [\mathbf Z^N]$ and $j$), we hence remark that Assumption~\ref{assump:A2} is equivalent to
\begin{enumerate}[label=(\roman*')]
\item the random variable $\rho^N [\mathbf Z^N]$ converges in law to $\rho$ as $N\to\infty$ ;

\item the random variable $j^N [\mathbf Z^N]$ converges in law to $j$ as $N\to\infty$.

\end{enumerate}

We now give some estimates concerning different metrics. For any $k>0$ and any probability measure $f \in \mathcal P (\R^3 \times \R^3)$ with support on $\Omega_0 \times \R^3$, we denote 
its moment of order $k>0$ by
$$
M_k(f) := \int_{\Omega_0 \times \mathbb R^3} (1 + |v|^2)^{k/2} \, f(\mathrm d x, \mathrm d v).
$$
We remark that $M_k(f) \geq 1$ for any $k >0$
and $k \mapsto M_k(f)$ is non-decreasing.
On the other hand, under the Assumption~\ref{assump:A1}-(2), we have a uniform bound for $(M_{k_0}(F^N))_{N\in \mathbb N^*}.$
So, below, we focus on probability measures with bounded $k_0$-momentum {\em i.e.}:
$$
\mathcal B_{k_0}(C_2) := \{ f \in \mathcal P (\R^3 \times \R^3) \text{ s.t.\ } \mathrm{Supp}(f) \subset \Omega_0 \times \R^3 \text{ and } M_{k_0}(f) \leq C_2\}
$$
where $k_0 \in [1,\infty)$ and $C_2 \geq 1$ are fixed by Assumption~\ref{assump:A1}-(2).
Standard arguments show that this set is closed w.r.t.\ the weak topology on $\mathcal P(\mathbb R^3 \times \mathbb R^3).$

\begin{lemma}\label{lem:equiv-distances}
Let $f, g \in \mathcal B_{k_0}(C_2)$ and define $\rho_f = \int_{\R^3} f(\cdot, \d v)$, $\rho_g = \int_{\R^3} g(\cdot, \d v)$, $j_f = \int_{\R^3} v f (\cdot, \d v)$ and $j_g = \int_{\R^3} v g(\cdot, \d v)$. Given $k>0$ we denote $\mathcal M_k := M_k(f) + M_k(g)$ and $K_0>0$ a constant depending on $\Omega_0$. 

\begin{enumerate}

\item For any $\theta \in (0,1)$ there holds
$$
\begin{aligned}
\| \rho_f - \rho_g \|_{(C^{0,\theta}_b (\R^3))^*} 
& \le K_0 W_1(\rho_f, \rho_g)^{\frac{\theta}{\theta + 1}} 
 \le K_0 W_1(f, g)^{\frac{\theta }{\theta + 1}}.
\end{aligned}
$$

\item For any $\theta \in (0,1)$ there holds
$$
\begin{aligned}
\| j_f - j_g \|_{(C^{0,\theta}_b (\R^3))^*} 
& \le K_0 \| j_f - j_g \|_{(C^{0,1}_b (\R^3))^*}^{\frac{\theta }{\theta+1 }} 
 \lesssim K_0  \mathcal M_{k_0}^{\frac{1}{k_0}\frac{\theta }{(\theta+1) }} \, W_1(f, g)^{\frac{(k_0-1)}{k_0}\frac{\theta  }{(\theta+1)}}.
\end{aligned}
$$

\end{enumerate}

\end{lemma}

\begin{proof}
These estimates are standard but we give the proof here for completeness.

\noindent
(1) We  prove first that
\begin{equation}\label{rho:C0theta*-W1}
\| \rho_f - \rho_g \|_{(C^{0,\theta}_b (\R^3))^*} \lesssim W_1(\rho_f, \rho_g)^{\frac{\theta }{\theta+1 }},
\end{equation}
from which we conclude by remarking that
\begin{equation}\label{W1rho-W1f}
W_1(\rho_f, \rho_g) \lesssim W_1(f,g).
\end{equation}
Recall that
$$
\| \rho_f - \rho_g \|_{(C^{0,\theta}_b (\R^3))^*}
= \sup_{\| \phi \|_{C^{0,\theta}_b (\R^3)} \le 1} \int_{\Omega_0} \phi(x) (\rho_f(\mathrm{d}x) - \rho_g(\mathrm{d}x)).
$$
We consider a sequence of mollifiers $(\zeta_\epsilon)_{\epsilon >0}$, that is, $\zeta_\epsilon(x) = \epsilon^{-3} \zeta( \epsilon^{-1} x )$, $\zeta \in C^\infty_c(\mathbb R^3)$ nonnegative, $\int \zeta(x) \, \d x = 1$, and $\mathrm{supp}(\zeta) \subset B(0,1)$. 
We split
$$
\begin{aligned}
\int_{\R^3 } \phi(x) (\rho_f - \rho_g)(\mathrm{d}x)
&= \int_{\R^3 } (\phi  * \zeta_\epsilon)(x) (\rho_f - \rho_g)(\mathrm{d}x)  
+ \int_{\R^3 } [ \phi  (x) - (\phi * \zeta_\epsilon)(x) ] (\rho_f - \rho_g)(\mathrm{d}x) \\
&=: T_1 + T_2 .
\end{aligned}
$$
For the term $T_2$, we easily remark that
$$
\begin{aligned}
\phi(x) - (\phi * \zeta_\epsilon)(x) 
&= \int_{\R^3 } [\phi(x) - \phi(x-y)] \zeta_\epsilon(y) \,\mathrm dy 
\le [\phi]_{C^{0,\theta} } \int_{\R^3 } |y|^\theta \zeta_\epsilon(y) \, \mathrm dy
\le [\phi]_{C^{0,\theta} } \, \epsilon^\theta .
\end{aligned}
$$
Hence the previous estimate yields
$$
T_2 \le \|  \phi   - (\phi  * \zeta_\epsilon) \|_{L^\infty(\mathbb R^3)} \, \int_{\R^3}    (\rho_f + \rho_g)(\mathrm{d}x)
\lesssim \| \phi \|_{C^{0,\theta}_b(\mathbb R^3)} \, \epsilon^\theta .
$$
For the term $T_1$ we observe that $x \mapsto  (\phi * \zeta_\epsilon)(x)$ lies in $\mathrm{Lip} (\R^3)$, indeed, for any $x \in \mathbb R^3$, we have
$$
\begin{aligned}
|\nabla_x (\phi  * \zeta_\epsilon) (x) |
&=  | (\phi  * \nabla_x \zeta_\epsilon) (x) | 
\le  \int_{\R^3} |\phi(x-y)|  \frac{|\nabla_x \zeta (y/\epsilon)|}{\epsilon^{4}} \, \mathrm dy \\
&\le \int_{\R^3} |\phi(x-\epsilon w)|  \frac{|\nabla \zeta (w)|}{\epsilon} \, \mathrm dw \\
&\lesssim \epsilon^{-1} \| \phi \|_{L^\infty(\mathbb R^3)} \| \nabla_x \zeta \|_{L^1(\mathbb R^3)}  ,
\end{aligned}
$$
which implies $ [ \phi  * \zeta_\epsilon ]_{\mathrm{Lip}}  \lesssim \epsilon^{-1}
\| \phi \|_{L^\infty(\mathbb R^3) }$. From that last estimate we get
$$
\begin{aligned}
T_1 
&\lesssim [ \phi  * \zeta_\epsilon ]_{\mathrm{Lip} } \int_{\R^3} \frac{\phi  * \zeta_\epsilon (x)}{[ \phi  * \zeta_\epsilon ]_{\mathrm{Lip} } } \, (\rho_f - \rho_g)(\mathrm{d}x) \\
&\lesssim \epsilon^{-1} \| \phi \|_{L^\infty(\mathbb R^3)}  \, 
\sup_{[\psi]_{\mathrm{Lip}} \le 1} \int_{\R^3} \psi(x) (\rho_f - \rho_g)(\mathrm{d}x)
= \epsilon^{-1} \| \phi \|_{L^\infty(\mathbb R^3)} \, W_1(\rho_f, \rho_g),
\end{aligned}
$$
Gathering previous estimates and choosing $\epsilon = W_1(\rho_f, \rho_g)^{\frac{1}{\theta+1}}$ completes the proof of \eqref{rho:C0theta*-W1}.
We now easily prove \eqref{W1rho-W1f} by remarking that
$$
\begin{aligned}
W_1(\rho_f, \rho_g) 
&= \sup_{[\psi]_{\mathrm{Lip} } \le 1} \int_{\R^3 } \psi(x) (\rho_f - \rho_g)(\mathrm{d}x) 
= \sup_{[\psi]_{\mathrm{Lip}} \le 1} \int_{\R^3  \times \mathbb R^3} \psi(x) (f - g)(\mathrm{d}x \mathrm{d}v) \\
&\le \sup_{[\Psi]_{\mathrm{Lip} } \le 1} \int_{\R^3  \times \mathbb R^3} \Psi(x,v) (f - g)(\mathrm{d}x, \mathrm{d}v) = W_1(f,g).
\end{aligned}
$$

\noindent
(2) By reproducing {\em mutatis mutandis} the arguments for \eqref{rho:C0theta*-W1} 
we obtain
\begin{equation}\label{j:C0theta*-C01*}
\| j_f - j_g \|_{(C^{0,\theta}_b (\R^3 ))^*}
\lesssim  \| j_f - j_g \|_{(C^{0,1}_b (\R^3 ))^*}^{\frac{\theta }{\theta+1}}.
\end{equation}
So we prove next
\begin{equation}\label{j:C01*-W1}
\| j_f - j_g \|_{(C^{0,1}_b (\R^3 ))^*}
\lesssim  \mathcal M_k^{\frac{1}{k_0}} \, W_1(f,g)^{ \frac{k_0-1}{k_0} }.
\end{equation}
For $R \ge 1 $ we define the smooth cutoff function $\chi_R(v) = \chi(v/R)$ with $\chi \in C^\infty_c(\mathbb R^d)$ nonnegative and $\mathbf 1_{B(0,1)} \le \chi \le \mathbf 1_{B(0,2)}$, and we write, denoting 
$j_f = (j_f^\alpha)_{1 \le \alpha \le 3}$ and $j_g = (j_g^\alpha)_{1 \le \alpha \le 3}$, for any $\alpha \in \{ 1,2,3 \}:$
$$
\begin{aligned}
\| j_f^\alpha - j_g^\alpha \|_{(C^{0,1}_b (\R^3 ))^*}
&:= \sup_{\| \phi \|_{C^{0,1}_b (\R^3 )} \le 1} \int_{\R^3 } \phi(x) (j_f^\alpha(\mathrm{d}x) - j_g^\alpha(\mathrm{d}x))
= \sup_{\| \phi \|_{C^{0,1}_b (\R^3 )} \le 1} \int_{\R^3  \times \mathbb R^3} \phi(x) v_\alpha (f - g)(\mathrm{d}x , \mathrm{d}v) \\
&= \sup_{\| \phi \|_{C^{0,1}_b (\R^3 )} \le 1} \left\{ \int_{\R^3  \times \mathbb R^3} \phi(x) v_\alpha \chi_R(v) (f - g)(\mathrm{d}x ,\mathrm{d}v)
+ \int_{\mathbb R^3 \times \mathbb R^3} \phi(x) v_\alpha (1-\chi_R(v)) (f - g)(\mathrm{d}x ,\mathrm{d}v) \right\} \\
&=: I_1 + I_2.
\end{aligned}
$$
Observe that, given $\phi \in C^{0,1}_b (\R^3 )$ such that $\|\phi\|_{C^{0,1}_b(\R^3 )} \leq 1,$ the mapping  $(x,v) \mapsto \phi(x) v_\alpha \chi_R(v) $ lies in $\mathrm{Lip} (\mathbb R^3 \times \mathbb R^3)$ with $[\phi v_\alpha \chi_R]_{\mathrm{Lip} } \lesssim R$. Indeed, we have
$$
[\phi v_\alpha \chi_R]_{\mathrm{Lip} } \leq [\phi]_{\mathrm{Lip}} \|v_{\alpha} \chi_{R}\|_{L^{\infty}(\mathbb R^3)} 
+ \|\phi\|_{L^{\infty}(\mathbb R^3)} \|\nabla_v v_{\alpha}\chi_R\|_{L^{\infty}(\mathbb R^3)}
$$
and, for any $v \in \R^3$,
$$
|v_{\alpha}\chi_R(v)| \lesssim R, \qquad 
|\nabla_v (v_\alpha \chi_R) (v)| 
\lesssim |v_\alpha| |\nabla_v \chi (\tfrac{v}{R})| \frac{1}{R}
\lesssim \| v\nabla_v \chi \|_{L^\infty(\R^3)} \lesssim 1,
$$
which implies
$$
I_1 \lesssim R W_1(f,g).
$$
For the second term, since $f,g \in \mathcal B_{k_0}(C_2)$, we have
$$
I_2 \lesssim \sup_{\| \phi \|_{C^{0,1}_b (\R^3)} \le 1} \int_{\mathbb R^3 \times \mathbb R^3} \phi(x) v_\alpha (1-\chi_R(v)) (f - g)(\mathrm{d}x \mathrm{d}v)
\lesssim \frac{\mathcal M_{k_0}}{R^{k_0-1}}
$$
and we conclude to \eqref{j:C01*-W1} by choosing $R = \frac{\mathcal M_{k_0}^{1/k_0}}{W_1(f,g)^{1/k_0}} \ge 1$ (since $W_1(f,g) \le \mathcal M_{k_0}$) if not infinite. 
\end{proof}

With the above lemma we can show the following sufficient condition for the convergences in Assumption~\ref{assump:A2} to hold.

\begin{lemma}\label{lem:chaos-A2}
Consider a sequence $(\mathbf Z^N)_{N \in \N^*}$ of exchangeable random variables on $\mathcal O^N$ and the associated sequence of symmetric laws $(F^N)_{N \in \N^*}$ on $\mathcal O^N$ satisfying Assumption~\ref{assump:A1}. Suppose that $(F^N)_{N \in \N^*}$ is $f$-chaotic (Definition~\ref{def:chaos}), for some probability measure $f$ on $\R^3 \times \R^3$ with support on $\Omega_0 \times \R^3$, and denote $\rho := \int_{\R^3} f(\cdot, \d v)$ and $j := \int_{\R^3} v f(\cdot, \d v)$.
Then $(F^N)_{N \in \N^*}$ satisfies Assumption~\ref{assump:A2}, more precisely there holds
$$
\mathbb E \left[ W_1 (\rho^N[\mathbf Z^N] , \rho) \right] 
\lesssim \mathbb E \left[ W_1( \mu^N[\mathbf Z^N] , f) \right] \xrightarrow[N \to \infty]{} 0
$$
and
$$
\mathbb E \left[ \| j^N[\mathbf Z^N] - j \|_{(C^{0,1}_b(\R^3))^*} \right] 
\lesssim \mathbb E \left[ W_1( \mu^N[\mathbf Z^N] , f) \right]^{\frac{k_0-1}{k_0}} \xrightarrow[N \to \infty]{} 0.
$$
\end{lemma}

\begin{proof}
Thanks to the moment condition Assumption~\ref{assump:A1}-(2) and the fact that $(F^N)_{N \in \N^*}$ is $f$-chaotic, we know from \cite{HM} that
$$
\lim_{N \to \infty} \mathbb E \left[ W_1( \mu^N[\mathbf Z^N] , f) \right] = 0 .
$$
We conclude the proof by applying Lemma~\ref{lem:equiv-distances} and remarking that $ 
\mathbb E \left[ M_{k_0}(\mu^N[\mathbf Z^N]) \right] = M_{k_0}(F^N_1)
$ is uniformly bounded thanks to Assumption~\ref{assump:A1}-(2).
\end{proof}

\subsection{Estimating the weight of concentrated configurations} \label{sec_stat}

For $\lambda, \alpha >0$, and any integer $M \le N$, we define
\begin{equation}
\mathcal O^N_{\alpha} := \big\{ \mathbf Z^N \in \mathcal O^N \mid \min_{i \neq j} |X_i^N - X_j^N| < N^{-\alpha}    \big\}
\end{equation}
and
\begin{equation}
\mathcal O^N_{\lambda,M} := \big\{ \mathbf Z^N \in \mathcal O^N \mid  \text{there exist at least } M \text{ particles } (X_i^N) \text{ in the same cell } \mathcal C(\lambda) \text{ of size } \lambda >0  \big\}.
\end{equation}
Here the cell $\mathcal C(\lambda)$ is given by, for some $y\in \mathbb R^3$, $(y_1 - \lambda/2 , y_1 + \lambda/2) \times (y_2 - \lambda/2 , y_2 + \lambda/2) \times (y_3 - \lambda/2 , y_3 + \lambda/2)$, so that $|\mathcal C_y(\lambda)| = \lambda^3$. 

Below, we study the weight of the sets $\mathcal O^N_{\lambda,M}$ and $\mathcal O^N_{\alpha}.$ 
For this, we allow that the parameters $\lambda$ and $M$ depend on $N.$ Namely, we denote:
\begin{equation} \label{eq_lambdaM} 
 M_N = N^{\beta}, \qquad \lambda_N := \left(\eta \frac{M_N}{N} \right)^{1/3}\,, \quad \forall \, N \in \mathbb N^*
\end{equation}
with positive parameters $\alpha,\beta,\eta$ to be fixed later on. 

We now state the main result of this section.

\begin{proposition}\label{prop:config}

Consider a sequence of random variables $(\mathbf Z^N)_{N \in \mathbb N^*}$ and the sequence of their associated laws $(F^N)_{N \in \N^*}$ satisfying Assumption~\ref{assump:A1}. Let $\alpha \in (2/3,1),$ $\beta \in (0,1/2)$ and $\eta \in (0,\infty)$ sufficiently small. Then, the sequences $(M^N)_{N\in\mathbb N^*}$ and $(\lambda^N)_{N\in \mathbb N^*}$ given by formula \eqref{eq_lambdaM} satisfy:
$$
 \mathbb P ( \mathbf Z^N \in  \mathcal O^N_{\lambda_N,M_N} \cup  \mathcal O^N_\alpha ) 
\lesssim \dfrac{1}{N^{{3\alpha -2}}}  \xrightarrow[N\to \infty]{} 0.
$$
\end{proposition}

We emphasize that the smallness restriction in the previous statement is explicit. With the notations of Assumption~\ref{assump:A1} it reads $\eta < 1/(2e C_1).$ The proof of 
Proposition~\ref{prop:config} is split into the two following lemmas.

\begin{lemma} \label{lem_sp1bis}
Under the assumptions of Proposition~\ref{prop:config}, there holds
$$
\mathbb P (  \mathbf Z^N \in  \mathcal O^N_{\lambda_N,M_N} ) \lesssim \left( 2 \eta C_1 e  \right)^{N^\beta} .
$$
\end{lemma}

\begin{proof}
By symmetry of $F^N$, given $\lambda>0$ and $M \in \mathbb N^*$ with $M \le N$, we have
$$
\mathbb P (\mathbf Z^N \in  \mathcal O^N_{\lambda,M} )
= \binom{N}{M} \, \mathbb P (  (X_1^N, \dots, X_M^N) \text{ are in the same cell } \mathcal C(\lambda)).
$$
In order to compute the last probability, again by symmetry, we only need to compute the probability of particles $i \in \{ 1, \dots, M-1 \}$ to be in the same cell $\mathcal C(\lambda)$ containing $X_M^N$, the position of particle number $M$. Since a cell $\mathcal C(\lambda)$ has diameter $\lambda$ (with respect to $\ell^{\infty}$-norm), we obtain:
$$
\begin{aligned}
&\mathbb P (  X_1^N, \dots, X_M^N  \text{ are in the same cell } \mathcal C(\lambda) ) \\
&\qquad \le  \mathbb P \left(  \bigcap_{i=1}^{M-1} \{ |X_i^N-X_M^N|_{\infty} < {\lambda}\} \right)\\
&\qquad \le  \left( \lambda^3 C_1  \right)^{M-1} ,
\end{aligned}
$$
where we have used Assumption~\ref{assump:A1}-(1) in last line.

%

When $N,M \to \infty$ with $(N-M)\to \infty$ and $N/M \to \infty$, Stirling's formula gives
$$
\begin{aligned}
\binom{N}{M} 
= \frac{N!}{M! (N-M)!}
&\sim  \frac{\sqrt{2 \pi}  N^{1/2} N^N e^{-N}}{\sqrt{2 \pi}  M^{1/2} M^M e^{-M} \sqrt{2 \pi}  (N-M)^{1/2} (N-M)^{N-M} e^{-(N-M)}} \\
&\sim \frac{1}{\sqrt{2\pi}} \left( \frac{N}{M}  \right)^{M}  \frac{1}{ M^{1/2}(1-\tfrac{M}{N})^{N-M+1/2}},
\end{aligned}
$$
which implies
$$
\mathbb P( \mathbf Z^N \in  O^N_{\lambda,M}) \lesssim \frac{1}{\sqrt{2\pi}}   \left( \frac{N}{M}  \right)^{M}  \frac{  1}{ M^{1/2}(1-\tfrac{M}{N})^{N-M+1/2}} \, \left({\lambda^3} C_1 \right)^{M-1}.
$$
We now consider the given sequences $(M_N)_{N \in \mathbb N^*}$ and $(\lambda_N)_{N \in \mathbb N^*}$ given by formula \eqref{eq_lambdaM}, and we get
$$
\begin{aligned}
\mathbb P( \mathbf Z^N \in  \mathcal O^N_{\lambda_N,M_N}) 
&\lesssim  \left( \frac{N}{M_N}  \right)^{M_N}  \frac{  1 }{ M_N^{1/2}(1-\tfrac{M_N}{N})^{N(1-M_N/N)+1/2}} \, \left( {\eta} \frac{M_N}{N} C_1 \right)^{M_N-1} \\
&\lesssim \left( \frac{N}{M_N}  \right) \frac{  1 }{ M_N^{1/2}(1-\tfrac{M_N}{N})^{N(1-M_N/N)+1/2}} \, \left( {\eta}  C_1 \right)^{M_N-1}
\end{aligned}
$$
Since $\beta \in (0,1/2),$ we have that $M_N^2/N \to 0$ so that we can simplify the denominator of the right-hand side:
$$
\begin{aligned}
\mathbb P( \mathbf Z^N \in  \mathcal O^N_{\lambda_N,M_N}) 
&\lesssim   \frac{N}{  M_N^{3/2} } \left( \eta C_1 e  \right)^{M_N-1} = N^{(1-  \frac{3\beta}{2} )}  \left( {\eta C_1 e } \right)^{N^\beta-1} \\
&\lesssim \exp\left( N^\beta \log( \eta C_1 e ) + \left(1 - \tfrac{3\beta}{2} \right)\log N  \right) 
\lesssim  \left( 2 \eta C_1 e   \right)^{N^\beta}.
\end{aligned}
$$ 
\end{proof}

\begin{lemma} \label{lem_sp2}
Under the assumptions of Proposition~\ref{prop:config}, there holds
$$
\mathbb P ( \mathbf Z^N \in \mathcal O^N_\alpha ) 
\lesssim   C_1 N^{2-3\alpha}.
$$
\end{lemma}

\begin{proof}
By symmetry of $F^N$ we have 
$$
\mathbb P (\mathbf Z^N \in  \mathcal O^N_{\alpha} ) = \binom{N}{2} \, \mathbb P ( |X_1^N - X_2^N| < N^{-\alpha} ),
$$
and we easily compute
$$
\mathbb P (|X_1^N - X_2^N| < N^{-\alpha} ) 
\lesssim \| F^N_1 \|_{L^\infty_x L^1_v(\mathbb R^3 \times \mathbb R^3)} \, N^{-3\alpha} 
\lesssim C_1 N^{-3\alpha},
$$
which completes the proof.
\end{proof}


\section{Properties of the mapping $U_N$ for fixed $N$}\label{sec:UN-prop}

In this section, we fix an arbitrary strictly positive $N \in \mathbb N$ and we analyze the properties of the mapping $U^N.$
As $N$ is fixed, we drop the exponents in notations (except $\mathcal O^N$). For example, we denote $U = U^N$, $\mathbf X = \mathbf X^N$, $\mathbf V = \mathbf V^N$, $X_i = X_i^N$ and $V_i = V_i^N$... The main result of this section reads:

\begin{proposition}
The mapping $U$ defined in \eqref{def:UN} satisfies $U \in C(\mathcal O^N ; \dot{H}^1(\mathbb R^3)).$
Moreover, if $F \in L^1(\mathcal O^N)$ is a sufficiently regular symmetric probability density, we have $U \in L^1(\mathcal O^N,F(\mathbf{Z}){\rm d}\mathbf Z).$
\end{proposition}
More quantitative statements on the integrability properties of $U$ are stated in due course. In particular, the meaning of 
``$F$ sufficiently regular'' is made precise in Section~\ref{sec_integrability} below.

Let first recall classical statement on the well-definition of the mapping $U.$ For fixed $\mathbf Z \in \mathcal O^N,$ by definition,   the restriction $ u$  of $U[\mathbf Z]$ to 
$$
\mathcal F := \mathbb R^3 \setminus \bigcup_{i=1}^N B_i \qquad \left( B_i = \overline{B(X_i,1/N)}, \quad \forall \, i =1,\ldots,N\right),
$$
should be the unique $\dot{H}^1(\mathcal F)$ vector-field for which 
there exists a pressure $ p$ such that $( u, p)$ is a solution to:
\begin{equation} \label{eq_stokes2}
\left\{
\begin{array}{rcl}
- \Delta { u} + \nabla { p} &=& 0 \,, \\[4pt]
{\rm div} \, { u} &= & 0 \,,
\end{array}
\right.
\quad \text{ in $\mathcal F$}
\end{equation}
completed with boundary conditions:
\begin{equation} \label{cab_stokes2}
\left\{
\begin{array}{rcll}
{ u}(x) &=& V_i  &  \text{on $B_i$} \, \\[4pt]
\lim_{|x| \to \infty}{ u}(x) &=& 0. \, &
\end{array}
\right.
\end{equation}
We recall here shortly the function spaces and analytical arguments underlying the mathematical treatment of this problem  \cite[Section 3]{Hil17}.
We refer the interested reader also to \cite[Sections IV-VI]{Galdi} for more details.

We denote $\mathcal D(\mathbb R^3) := \{ w \in C^{\infty}_c(\mathbb R^3) \quad {\rm div} \ w = 0 \}$
and $D(\mathbb R^3)$ its closure for the $\dot{H}^1(\mathbb R^3)$-norm. We recall that $D(\mathbb R^3)$
is a Hilbert space for the scalar product:
$$
(u,v) \mapsto \int_{\mathbb R^3} \nabla u : \nabla v
$$
and that $D(\mathbb R^3) \subset L^6(\mathbb R^3).$ For a smooth exterior domain $\mathcal F$ ({\em i.e.} the complement
of some bounded compact set $B \subset \mathbb R^3$) we can then set 
\[
D(\mathcal F) = \{ u_{|_{\mathcal F}}, \quad u \in D(\mathbb R^3)\}.
\]
By restriction, $D(F)$ is also a Hilbert space for the scalar product:
\begin{equation} \label{eq_H1-2}
(u,v) \mapsto \int_{\mathcal F} \nabla u : \nabla v.
\end{equation}

\begin{remark}
We just remarked that:
\[
D(\mathcal F) \subset \{ u \in L^6(\mathcal F) \text{ s.t. } \nabla u \in L^2(\mathcal F) \text{ and } {\rm div }\ u = 0\}.
\]
One may then wonder if this inclusion is an equality. However, we have that $D(\mathcal F) \subset H^1_{loc}(\mathcal F).$ 
Since $\partial \mathcal F$ is compact the trace of elements of $D(\mathcal F)$ is then well-defined in $H^{1/2}(\partial \mathcal F).$ 
Standard manipulations show also that, if $u \in D(\mathcal F)$ then 
\begin{equation} \label{eq_flux}
\int_{\Gamma} u \cdot n {\rm d} \sigma = 0 \quad \text{ for every connected component $\Gamma$ of $\partial \mathcal F.$}
\end{equation}
Conversely, if this latter condition is satisfied then one can extend $u$ by the solutions of the Stokes problem inside the connected component
of $\mathbb R^3 \setminus \mathcal F$ surrounded by $\Gamma.$ Finally, we may then characterize:
\[
D(\mathcal F) = \{ u \in L^6(\mathcal F) \text{ satisfying } \nabla u \in L^2(\mathcal F),\; {\rm div }\ u = 0 \text{ and  \eqref{eq_flux}}  \}.
\]
In particular $D(\mathcal F)$ contains $D_0(\mathcal F)$, the subset of divergence-free vector-fields vanishing on $\partial \mathcal F$, which 
can also be seen as the closure of  $\mathcal D_0(\mathcal F) := \{ w \in C^{\infty}_c(\mathcal F) , \;  {\rm div} \ w = 0 \}$ for the
$\dot{H}^1(\mathcal F)$-scalar product \eqref{eq_H1-2}.
Remarking that extensions of vector-fiels in $D_0(\mathcal F)$ are the trivial ones, and recalling that we have the  embedding $D(\mathbb R^3) \subset L^6(\mathbb R^3)$ we infer that $D_0(\mathcal F)$ embeds 
continuously into $L^6(\mathcal F).$
\end{remark}

\medskip

With these definitions,  problem \eqref{eq_stokes2}-\eqref{cab_stokes2} is associated with a(n equivalent) weak formulation:
\begin{center}
\begin{minipage}{.8\textwidth}
Find $u \in D(\mathcal F)$ such that $u = V_i$ on $\partial B_i$ for $i=1,\ldots,N$ and
$$
\int_{\mathcal F} \nabla u : \nabla w = 0 \,, \quad \text{for arbitrary $w \in D_0(\mathcal F).$}
$$
\end{minipage}
\end{center}
Existence of a weak-solution yields by applying a standard Riesz-Fr\'echet or Lax-Milgram argument which also yields the following 
variational property:
\begin{theorem} \label{thm_varcar}
The vector-field $U[\mathbf Z] \in D(\mathbb R^3)$ is the unique minimizer of
$$
\left\{ \int_{\mathbb R^3} |\nabla v|^2 , \quad v \in D(\mathbb R^3) \text{ s.t. } v_{|_{B_i}} = V_i  \quad 
\text{ for all $i \in \{1,\ldots,N\}$}\right\}\,. 
$$
\end{theorem}
We refer the reader to \cite[Theorem 3]{Hil17} for a similar proof on a bounded domain that can be adapted easily to this case
with the functional framework depicted above.
The remainder of this section is organized as follows. In the next subsection, we consider the continuity properties
of the mapping $U.$ We continue by deriving a pointwise estimate and end up the section with an analysis of integrability
properties of $U.$

\subsection{Continuity of the mapping $U$}
At first, we obtain that:
\begin{lemma}
The mapping $U$ satisfies $U \in C(\mathcal O^{N} ; D(\mathbb R^3)).$
\end{lemma}
As only continuity is required for our purpose, we give below a proof of this lemma based on monotonicity arguments only.
Nonetheless, one may prove much finer properties by using change of variables methods
(see \cite{Simon90,ChambrionMunnier} for instance).   

\begin{proof}
The problem \eqref{eq_stokes2}-\eqref{cab_stokes2} being linear with respect to its boundary data, we have that, for fixed $\mathbf X \in \mathbb R^{3N}$ such that 
$| X_i -  X_j| > 2/N$ when $i\neq j,$ the mapping $\mathbf V \mapsto U[\mathbf Z]$  is linear. Consequently, it is sufficient to 
consider the continuity of the mapping $\mathbf X \mapsto U[\mathbf Z]$ for fixed $\mathbf V.$  

Let $\mathbf V \in \mathbb R^{3N}$ be fixed and consider $\mathbf X \in \mathbb R^{3N}$ -- such that $| X_i -  X_j| > 2/N$ for any $i\neq j$ -- and 
a sequence $(\mathbf X^{(k)})_{k \in \mathbb N}$ in $\mathbb R^{3N}$ such that
\begin{itemize}
\item $\mathbf Z^{(k)} = (\mathbf X^{(k)} , \mathbf V) \in \mathcal O^{N}$ for any $k \in \mathbb N,$
\item $\lim_{k \to \infty} X^{(k)}_{i} = X_i,$ for $ i =1,\ldots,N.$
\end{itemize}
We are interested in proving that $U[\mathbf Z^{(k)} ]$ converges to $U[\mathbf Z]$ in $D(\mathbb R^3).$
Due to the variational characterization of $U[\mathbf Z],$ we remark that it is sufficient to prove that the sequence 
$(m^{(k)})_{k \in \mathbb N}$ defined by
$$
m^{(k)} := \inf \left\{ \int_{\mathbb R^3} |\nabla v|^2 , \quad v \in D(\mathbb R^3) \text{ s.t. } v_{|_{B(X_i^{(k)},1/N)}} = V_i  \quad 
\text{ for all $i \in \{1,\ldots,N\}$}\right\}\, \quad \forall \, k \in \mathbb N
$$ 
satisfies: 
\begin{equation} \label{eq_propmk}
\lim_{k \to \infty} m^{(k)} = m_{\infty} := \inf \left\{ \int_{\mathbb R^3} |\nabla v|^2 , \quad v \in D(\mathbb R^3) \text{ s.t. } v_{|_{B(X_i,1/N)}} = V_i  \quad 
\text{ for all $i \in \{1,\ldots,N\}$}\right\}.
\end{equation}
Indeed,  for arbitrary $k \in \mathbb N,$ there holds: $m^{(k)} = \|\nabla U[\mathbf Z^{(k)} ]\|_{L^2(\mathbb R^3)}^2.$ Consequently, if $(m^{(k)})_{k\in \mathbb N}$ converges, $U[\mathbf Z^{(k)} ]$ is bounded in $D(\mathbb R^3)$. We may then pass to the limit in the weak formulation of the Stokes problem (restricted to test-function 
in $\mathcal D_0(\mathcal F)$) and we obtain that $U[\mathbf Z ]$ is the weak limit of $U[\mathbf Z^{(k)} ]$ in $D(\mathbb R^3).$
The convergence of $(m^{(k)})_{k\in \mathbb N}$ implies then that  $(\|\nabla U[\mathbf Z^{(k)} ] \|_{L^2(\mathbb R^3)})_{k\in\mathbb N}$ converges to $\|\nabla U[\mathbf Z ]\|_{L^2(\mathbb R^3)}.$ 
As $D(\mathbb R^3)$ is a Hilbert space, this ends the proof. 

To prove \eqref{eq_propmk}, we analyze the continuity properties of the function $m_{\infty}(\cdot)$
as defined by:
$$
m_{\infty}(R) = \inf \left\{ \int_{\mathbb R^3} |\nabla v|^2 , \quad v \in D(\mathbb R^3) \text{ s.t. } v_{|_{B(X_i,R)}} = V_i  \quad 
\text{ for all $i \in \{1,\ldots,N\}$}\right\}, \quad \forall \, R > 0,
$$
We note that $m_{\infty} = m_{\infty}(1/N)$ and that, as $|X_i-X_j| > 2/N$ for $i\neq j,$ this function is well defined for $R$ close to $1/N.$ 
Left continuity in $1/N$ is for free. Indeed, by construction, $m_{\infty}(\cdot)$ is increasing and, if we had $\lim_{R \to 1/N^{-}} m_{\infty}(R) < m_{\infty}(1/N),$ we would
be able to construct a vector-field $v \in D(\mathbb R^3)$ satisfying simultaneously $v_{|_{B(X_i,1/N)}} = V_i$ for $i=1,\ldots,N$ and
$$
\int_{\mathbb R^3} |\nabla v|^2 \leq \lim_{R \to 1/N^{-}} m_{\infty}(R)  < m_{\infty}(1/N),
$$
which yields a contradiction. Right continuity in $1/N$ is a bit more intricate. To this end, we note that $m_{\infty}(1/N)$ is achieved by
$U[\mathbf Z].$ Remarking that, on the one hand, for an arbitrary truncation function $\chi$ there holds:
$$
\nabla \times\left[ \chi(x)  \dfrac{V_i \times x}{2} \right] = 
\left\{
\begin{array}{rl}
V_i  & \text{ on the  set $\{\chi = 1\}$}\\
0 & \text{ on the set $\{\chi = 0 \}$},
\end{array}
\right.
$$
and that, on the other hand $\mathcal D_0(\mathcal F)$ is dense in $D_0(\mathcal F),$ we may construct a sequence 
$(w^{(l)})_{l \in \mathbb N} \in [D(\mathbb R^3)]^{\mathbb N}$ converging to $U[\mathbf Z]$
and a sequence $(\varepsilon^{(l)})_{l \in \mathbb N} \in (0,\infty)^{\mathbb N}$ converging to $0$ such that, for arbitrary $l$ there holds:
$$
w^{(l)} = V_i \quad  \text{ on $B(X_i,1/N+\varepsilon^{(l)})$}\,, \quad \forall \, i=1,\ldots,N.
$$
This implies that:
$$
\|\nabla U[\mathbf Z ] \|_{L^2(\mathbb R^3)}= m_{\infty}(1/N) \leq m_{\infty}(1/N+\varepsilon^{(l)}) \leq \|\nabla w^{(l)} \|_{L^2(\mathbb R^3)} \,, \quad \forall \, k \in \mathbb N,
$$
and consequently, by comparison, that:
$$
\lim_{R \to  1/N^{+} } m_{\infty}(R) = \lim_{l \to \infty} m_{\infty}(1/N+\varepsilon^{(l)}) = m_{\infty}(1/N).
$$

To conclude, we apply a simple geometric argument implying that, associated with the sequence $(X^{(k)})_{k\in \mathbb N},$ 
we may construct a sequence $(\eta^{(k)})_{k\in\mathbb N} \in (0,\infty)$ converging to $0$ for which, for arbitrary
$k \in \mathbb N$ we have:
$$
B(X_i,1/N-\eta^{(k)}) \subset  B(X_i^{(k)},1/N)  \subset  B(X_i,1/N+ \eta^{(k)}) \quad \forall\, i=1,\ldots,N.
$$
Consequently, for arbitrary $k \in \mathbb N,$ by comparing the sets on which $U[\mathbf Z^{(k)} ]$ is equal to $V_i$ with balls of center $X_i,$ we obtain:
$$
m_{\infty}(1/N-\eta^{(k)}) \leq m^{(k)} \leq  m_{\infty}(1/N+\eta^{(k)}).
$$
We conclude the proof thanks to the previous continuity analysis  of $R \mapsto m_{\infty}(R)$ in $R=1/N.$
\end{proof}

\subsection{A pointwise estimate} We obtain now a bound for given configurations:
\begin{lemma}\label{lem:U[ZN]-boundD}
There exists a universal constant $C$ for which, given $\mathbf Z \in \mathcal O^N$,
there holds: 
$$
\|\nabla U[\mathbf Z] \|_{L^2(\mathbb R^3)}^2  \leq 
  \dfrac{C}{N} \sum_{i=1}^N |V_i|^2 \left( 1+ {\dfrac 1N}  \sum_{j \neq i} \dfrac{\mathbf{1}_{|X_i-X_j| < \frac 5{2N}}}{{|X_i-X_j| - \dfrac{2}{N}}}    \right) \,.
$$ 
\end{lemma}
\begin{proof}
In this proof $ \mathbf Z \in \mathcal O^N$ is fixed and splits into ${\mathbf X}$ and ${\mathbf V}$. The idea of the proof is to construct a suitable function
$$
w  \in Y[\mathbf Z] :=  \left\{ v \in D(\mathbb R^3) \text{ s.t.\ } v_{|_{B_i}} = V_i 
\text{ for all $i \in \{1,\ldots,N\}$}\right\}
$$
whose norm can be bounded by the right-hand side of the above inequality. The bound is then transferred to $U[\mathbf Z]$ {\em via} its variational characterization  (see {\bf Theorem~\ref{thm_varcar}}). To construct the candidate $w$ we consider successively the spheres $B_i$ in the cloud. Given a sphere $B_i$ we construct a divergence-free vector-field $w_i$ which satisfies the boundary condition 
$w_i = V_i$ on $B_i$ and $w_i =0$ on the $B_j,$ for $j \neq i.$ A na\"ive construction of $w_i$ would be to truncate away from $\partial B_i$
as a function of $|x-X_i|$. This would create a non-optimal vector-field (because it requires to choose the distance at which the truncation vanishes smaller than the minimum distance between $B_i$ and the $B_j$'s). Our method consists in drawing a virtual sphere of radius $3/2N$ around $X_i$. We then intersect this sphere with 
$\mathcal F$. This creates a connected domain with two boundaries: an internal one corresponding to $\partial B_i$ and an external one 
made partially of the boundary of $B(X_i,3/2N)$ and partially of small spherical caps corresponding to the $B_j$'s that intersect $B(X_i,3/2N)$.
We create then a vector-field that satisfies $w_i = V_i$ on the internal boundary and $w_i = 0$ on this virtual external boundary by truncating the (constant) vector-field $V_i$. The key-point is that we make the truncation to depend not only on the distance $|x-X_i|$ but also on the projection of the point $x$ and the sphere $B_i.$
We treat then differently the truncation in a zone between $\partial B_i$ 
and spherical cap by adapting the construction of \cite{HillairetTakfarinas}.

\medskip

Technical details of the proof are are rather long, hence we stick to the main ideas here and postpone them to {Appendix {\ref{app_wi}}}.
The first intermediate result concerns the treatment of sphere $B_i$:
\begin{lemma} \label{lem_wi}
Given $i\in \{1,\ldots,N\},$ there exists $w_i \in D(\mathbb R^3)$ satisfying 
\begin{align}
\label{eq_wi1}& w_i = V_i \text{ in $B_i$ and $w_i =0$ in $B_j$ for $j \neq i\,,$}\\[4pt]
\label{eq_wi2}& \mathrm{Supp} (w_i) \subset B(X_i, \tfrac{3}{2N}),
\end{align}
such that:
\begin{equation} \label{eq_bornewi}
\|\nabla w_i \|_{L^2(\mathbb R^3)}^2  \leq  \dfrac{C |V_i|^2}{N} \left( 1
 +  { \dfrac 1N} \sum_{j \neq i} \dfrac{\mathbf{1}_{|X_i-X_j| < \frac 5{2N}}}{{|X_i-X_j| - \dfrac{2}{N}}}    \right) \,.
\end{equation}
for a universal constant $C.$
\end{lemma}
Let $(w_i)_{i=1,\ldots,N}$ be given by Lemma~\ref{lem_wi}. By combining \eqref{eq_wi1} for $i=1,\ldots,N$, it is straightforward that:
$$
w = \sum_{i=1}^N w_i \in Y[\mathbf Z].
$$
Furthermore:
$$
\int_{\mathbb R^3} |\nabla w|^2 = \sum_{i=1}^{N} \sum_{j=1}^N \int_{\mathbb R^3} \nabla w_i : \nabla w_j .
$$
At this point, we use the property \eqref{eq_wi2}  in order to bound the  right-hand side. 
Given $i \in \{1,\dots,N\}$ let denote 
\[
\mathcal I_i := \left\{ j \in  \{1,\ldots,N\}  \text{ s.t. } B(X_i, \tfrac{3}{2N}) \cap B(X_j, \tfrac{3}{2N}) \neq \emptyset  \right\}.
\]
We remark that, given two indices $i$ and $j$ we have the equivalence between 
$j \in {\mathcal I}_i$ and $i \in {\mathcal I}_j$.

On the one hand, applying \eqref{eq_wi2}, there holds:
$$
 \sum_{j = 1}^N \int_{\mathbb R^3} \nabla w_i : \nabla w_j  = \sum_{j \in \mathcal I_i} \int_{\mathbb R^3} \nabla w_i : \nabla w_j  \quad \forall \, i=1,\ldots,N.
$$
On the other hand, we have:
\begin{lemma} \label{lem_geometrie}
Given $i \in \{1,\ldots,N\}$ the set $\mathcal I_i$ contains at most $64$ distinct indices. 
\end{lemma}
This lemma is obtained thanks to simple geometric argument that we develop in 
Appendix~\ref{app_wi}.
Applying standard inequalities, we can then bound:
$$
\left|  \sum_{j=1}^N  \int_{\mathbb R^3} \nabla w_i : \nabla w_j \right| \leq 32 \int_{\mathbb R^3} |\nabla w_i|^2 +  \dfrac{1}{2}  \sum_{j \in \mathcal I_i} \int_{\mathbb R^3} |\nabla w_j|^2, 
\quad \forall \, i=1,\ldots,N,
$$
which entails:
\begin{align*}
\int_{\mathbb R^3} |\nabla w|^2  & \leq { 32} \sum_{i=1}^N \int_{\mathbb R^3} |\nabla w_i|^2 + \dfrac 12 \sum_{i=1}^N \sum_{j \in \mathcal I_i} \int_{\mathbb R^3} |\nabla w_j|^2 \\
&  \leq { 32}  \sum_{i=1}^N \int_{\mathbb R^3} |\nabla w_i|^2 + \dfrac{1}{2}\sum_{j=1}^N |\mathcal I_j| \int_{\mathbb R^3} |\nabla w_j|^2. \\
& \leq { 64} \sum_{i=1}^N \int_{\mathbb R^3} |\nabla w_i|^2.
\end{align*}
We then conclude the proof by applying \eqref{eq_bornewi}.
\end{proof}

\subsection{Integrability properties of the mapping $U$.} \label{sec_integrability}
In this last part, we envisage to integrate the mapping $U$ against a sufficiently regular symmetric  probability density $F\in L^1(\mathcal O^N).$
To state the regularity assumption, we recall the notations:
\begin{align*}
& F_1 (z)  = \int_{\mathbb R^{6(N-1)}} \mathbf{1}_{\mathcal O^{N}}(z,\mathbf{z}') F( z,\mathbf z'){\rm d} \mathbf z', \quad \forall \, z \in \mathbb R^6,
\\
& F_2(z_1,z_2) = \int_{\mathbb R^{6(N-2)}} \mathbf{1}_{\mathcal O^{N}}(z_1,z_2,\mathbf{z}') F( z_1, z_2,\mathbf z'){\rm d} \mathbf z',
\quad \forall \, (z_1,z_2) \in \mathcal O_{2}^N,
\end{align*}
where $\mathcal O_{2}^N := \Big\{(z_1,z_2) \in \mathbb R^{6} \text{ s.t. } |x_1-x_2| > \dfrac 2N \Big\}.$ We introduce also:
\[
\mathfrak{j}(x_1,x_2) = \int_{\mathbb R^6} |v_1|F_2((x_1,v_1),(x_2,v_2)) {\rm d}v_1 {\rm d}v_2,  \quad \forall \, (x_1,x_2) \text{ s.t.\ } |x_2-x_1| > \dfrac{2}{N}.
\]
With these notations, we prove 
\begin{proposition} \label{prop_unifL1}
Let $F \in L^1(\mathcal O^N)$ be a symmetric probability density satisfying
\begin{align} \label{eq_condition1}
\int_{\mathbb R^6} (1+|z|^2) F_1(z) {\rm d}z < \infty , \\
\label{eq_condition2}
\int_{\mathbb R^3} \left[ \sup_{x_2 \in \mathbb R^3 \setminus \overline{B(x_1,2/N)}}  |\mathfrak j(x_1,x_2)| \right]{\rm d}x_1 < \infty.
\end{align} 
There holds $U \in L^1(\mathcal O^N, F(\mathbf Z){\rm d}\mathbf{Z})$ and there exists a universal constant $C$
such that: 
\[
\mathbb E[ \|\nabla U\|_{L^2(\mathbb R^3)}] \leq C \left[ \left( \int_{\mathbb R^6} (1+|z|^2) F_1(z) {\rm d}z \right)^{\frac 12} + \frac 1N \int_{\mathbb R^3} \left[ \sup_{x_2 \in \mathbb R^3 \setminus \overline{B(x_1,2/N)}}  |\mathfrak j(x_1,x_2)| \right] {\rm d}x_1\right] .
\]

\end{proposition}
 
\begin{proof}
Since $U \in C(\mathcal O^N; D(\R^3))$, our proof reduces to show that $\mathbb E[ \|\nabla U\|_{L^2(\mathbb R^3)}]$ is finite.
Let $\mathbf Z \in \mathcal O^N,$ applying the bound of Lemma~\ref{lem:U[ZN]-boundD} together with a standard comparison  
argument, we obtain that:
$$
 \|\nabla U[\mathbf Z]\|_{L^2(\mathbb R^3)} \leq C \left[ \dfrac{1}{\sqrt{N}}\left[  \sum_{i=1}^N |V_i|^2  \right]^{\frac 12} + \dfrac{1}{{N}}\sum_{i=1}^N \sum_{j \neq i} |V_i| \dfrac{\mathbf 1_{|X_i - X_j| < \frac{5}{2N} }}{\sqrt{|X_i - X_j| - \frac 2N}}\right].
$$
We have then 
$$
\mathbb E[\|\nabla U[\mathbf Z]\|_{L^2(\mathbb R^3)}] \leq C \left(   \mathbb E \left[  \left(\frac1N\sum_{i=1}^N |V_i^N|^2  \right)^{\frac 12} \right]  +   \mathbb E \left[ \dfrac{1}{{N}}\sum_{i=1}^N \sum_{j \neq i}   |V_i| \, \dfrac{\mathbf 1_{|X_i-X_j| < \frac{5}{2N} }}{\sqrt{|X_i - X_j| - \frac 2N}} \right] \right)
$$
We split  the right-hand side into two integrals $I_1$ and $I_2.$
First applying a Jensen inequality and then symmetry properties of the measure $F$  we have:
\begin{eqnarray*}
I_1 &:=&  \mathbb E \left[  \left(\frac1N\sum_{i=1}^N |V_i^N|^2  \right)^{\frac 12} \right]
\leq  \mathbb E \left[  \frac1N\sum_{i=1}^N |V_i^N|^2   \right]^{\frac 12} \\ 
&\le&  
\left( \int_{\mathbb R^6} (1+|z|^2) F_1(z) {\rm d}z\right)^{1/2}.
\end{eqnarray*}
By assumption \eqref{eq_condition1}, we have then $I_1 < \infty.$
Furthermore, using symmetry properties,  the definition of $\mathfrak j$ and assumption \eqref{eq_condition2}, we infer:
\begin{eqnarray*}
I_2  & :=&  \mathbb E \left[ \dfrac{1}{{N}}\sum_{i=1}^N \sum_{j \neq i}   |V_i| \, \dfrac{\mathbf 1_{|X_i-X_j| < \frac{5}{2N} }}{\sqrt{|X_i - X_j| - \frac 2N}} \right]  \\
&\leq & N  \mathbb E \left[  |V_1| \,  \dfrac{\mathbf 1_{|X_1-X_2| < \frac{5}{2N}}}{\sqrt{|X_1 - X_2| - \frac 2N}} \right] \\
&=& N  \int_{|x_1-x_2|>\frac2N} \left\{ \int_{\mathbb R^3 \times \mathbb R^3} |v_1|  \dfrac{\mathbf 1_{|x_1-x_2| < \frac{5}{2N}}}{\sqrt{|x_1 - x_2| - \frac 2N}} \, F_2(z_1,z_2){\rm d}v_1 {\rm d}v_2 \right\} {\rm d}x_1{\rm d}x_2  \\
& \leq & N \int_{\mathbb R^3} \int_{B(x_1,\frac{5}{2N})\setminus B(x_1, \frac{2}{N})} \dfrac{1}{\sqrt{|x_1-x_2| - \frac 2N}} \,  \mathfrak j(x_1,x_2) \, {\rm d}x_2{\rm d}x_1.\\
& \leq &  \dfrac{1}{{N}^{3/2}}  \int_{\mathbb R^3} \left[ \sup_{x_2 \in \mathbb R^3 \setminus \overline{B(x_1,2/N)}}  |\mathfrak j(x_1,x_2)| {\rm d}x_1\right] \int_{B(0,\frac{5}{2}) \setminus B(0,2)} \dfrac{1}{\sqrt{|y| - 2}} \, {\rm d}y.
\end{eqnarray*}
The last integral appearing on this last line being finite, we obtain that $I_2 < \infty$ and our proof is complete.
\end{proof}
With similar arguments as in the proof of this theorem, we also obtain the following corollary:
\begin{corollary} \label{cor_unifL1}
Under the assumptions of Proposition~\ref{prop_unifL1}, given $\tilde{\mathcal O}^N \subset \mathcal O^N$ we have:
\[
\mathbb E[ \|\nabla U\|_{L^2(\mathbb R^3)} \mathbf 1_{\tilde{\mathcal O}_N}] \leq  C \left[ |\mathbb P(\tilde{\mathcal O}_N)|^{\frac 12}  \left( \int_{\mathbb R^6} (1+|z|^2) F_1(z) {\rm d}z \right)^{\frac 12} +  \dfrac{1}{{N}^{3/2}}   \int_{\mathbb R^3} \left[ \sup_{x_2 \in \mathbb R^3 \setminus \overline{B(x_1,2/N)}}  |\mathfrak j(x_1,x_2)| \right] {\rm d}x_1\right] .
\]
\end{corollary}


\section{Main estimate for non-concentrated configurations}\label{sec:bonneconfig}

In this section, we compute a  bound for the distance between a solution to the $N$-particle problem and the limit Stokes-Brinkman system in a ``favorable'' case.
For this, let first state a stability estimate for the Stokes-Brinkman system suitable to our purpose.

Let consider a nonnegative density $\tilde{\rho} \in L^3(\Omega_0)$
and a momentum $\tilde{\text{\j}} \in L^{2}(\mathcal O)$ where $\Omega_0$ and $\mathcal O$ are bounded open subsets of $\mathbb R^3.$ The subset $\Omega_0$ is the one given in the introduction, corresponding to the domain occupied by the cloud of particles. We denote below $\Omega_1 = \Omega_0 + B(0,1).$ The subset $\mathcal O$ is another bounded open subset, not necessarily the same one. We apply the convention that we extend $\tilde{\rho}$ and $\tilde{\text{\j}}$ by $0$ in order
to yield functions on $\mathbb R^3.$ In this framework, the existence/uniqueness theorem in bounded domains (as mentioned in \cite{MH}) extends to the Stokes-Brinkman problem on the whole space: 
\begin{equation} \label{eq_SB}
\left\{
\begin{array}{rcl}
-\Delta u + \nabla p + 6\pi \tilde{\rho} u &=& 6\pi \tilde{\text{\j}} \\
{\rm div} \, u &=& 0  
\end{array}
\right.
\quad \text{ in $\mathbb R^3$},
\end{equation}
\begin{equation} \label{cab_SB}
\lim_{|x| \to \infty} |u(x)| = 0.
\end{equation}
Indeed, as in the case of the Stokes problem, the system \eqref{eq_SB}-\eqref{cab_SB} is associated with the weak formulation
\begin{center}
\begin{minipage}{.8\textwidth}
Find $u \in D(\mathbb R^3)$ such that 
$$
\int_{\mathbb R^3} \nabla u : \nabla w + 6\pi  \int_{\mathbb R^3} \tilde{\rho} u \cdot w = 6\pi \int_{\mathbb R^3} \tilde{\text{\j}} \cdot w,
\qquad 
\forall \, w \in D(\mathbb R^3).
$$
\end{minipage}
\end{center}
For positive $\tilde{\rho} \in L^3(\Omega_0) \subset L^{3/2}(\mathbb R^3),$ the left-hand side of the weak formulation represents a bilinear mapping $a_{\rho}$ which is in the same time coercive and continuous on $D(\mathbb R^3)$
(we recall that $D(\mathbb R^3) \subset L^6(\mathbb R^3)$). Hence, for arbitrary $\tilde{\text{\j}} \in L^{2}(\Omega_0) \subset L^{6/5}(\mathbb R^3) \subset [D(\mathbb R^3)]^*$ we can apply a standard Lax-Milgram argument to obtain that \eqref{eq_SB}-\eqref{cab_SB} admits a unique weak solution $u := u[\tilde{\rho},\tilde{\text{\j}}] \in D(\mathbb R^3)$.
At this point, we note that any weak solution $u$ to \eqref{eq_SB}-\eqref{cab_SB} is also a weak solution to the Stokes equations with data $6\pi ( \tilde{\text{\j}} - \tilde{\rho} u).$ Since $\tilde{\text{\j}} \in L^2(\mathbb R^3)$ and $\tilde{\rho} \in L^3(\mathbb R^3)$ we obtain that the source term is in $L^2(\mathbb R^3)$ and apply 
elliptic regularity estimates for the Stokes equations
on $\mathbb R^3$ (see \cite[Theorem IV.2.1]{Galdi}). This yields:
\begin{proposition} \label{prop_SB1}
For arbitrary $\tilde{\text{\j}} \in L^2(\mathcal O)$ and non-negative $\tilde{\rho} \in L^{3}(\Omega_0)$ the  unique weak solution $u := u[\tilde{\rho},\tilde{\text{\j}}]$ to the Stokes-Brinkman problem \eqref{eq_SB}-\eqref{cab_SB} satisfies $\nabla^2 u \in L^2(\mathbb R^3)$ and there exists constants $K_0,K_1$ whose dependencies are mentioned in parenthesis such that: 
$$
\|\nabla u\|_{L^2(\mathbb R^3)}  \leq K_0 \|\tilde{\text{\j}}\|_{L^{6/5}(\mathbb R^3)}, \qquad  \|\nabla^2 u\|_{L^2(\mathbb R^3)} \le K_1(\|\tilde{\rho}\|_{L^{3}(\mathbb R^3)}) \left[ \|\tilde{\text{\j}}\|_{L^2(\mathbb R^3)} + \|\tilde{\text{\j}}\|_{L^{6/5}(\mathbb R^3)}\right].
$$
\end{proposition} 

By duality, the previous elliptic-regularity statement entails a regularity statement in negative Sobolev spaces. Namely, given a nonnegative density $\tilde{\rho} \in L^{3}(\Omega_0)$, we denote, for arbitrary $v \in D(\mathbb R^3):$
$$
[v]_{\tilde{\rho},2}  :=  \sup \left\{ \left|\int_{\mathbb R^3} \nabla v : \nabla w  + 6\pi \int_{\mathbb R^3} \tilde{\rho} v \cdot w \right| \,, w \in  \mathcal D(\mathbb R^3) \text{ with } \|\nabla w\|_{L^2(\mathbb R^3)} + \|\nabla^2 w\|_{L^2(\mathbb R^3)} \leq   1  \right\}.
$$
Reproducing the arguments of  \cite[Lemma 2.4]{MH}, we obtain then the following proposition:
\begin{proposition} \label{prop_SB2}
Given a bounded open subset $\mathcal O  \subset \mathbb R^3,$ there exists $K:= K(\mathcal O,\|\tilde{\rho}\|_{L^{3}(\Omega_0)})$ such that 
\[\|v\|_{L^2(\mathcal O)} \leq K [v]_{\rho,2}. 
\]
\end{proposition}
We refer the reader to the proof of \cite[Lemma 2.4]{MH}  for more details.

The computations below are then based on the following remark. Let $\mathbf Z^N 
= (X_1^N,V_1^N,\ldots,X_N^N,V_N^N) \in \mathcal O^N,$  $U = U_N[\mathbf Z^N]$
and $P$ the associated pressure.  For arbitrary divergence-free vector-field $w \in C^{\infty}_c(\mathbb R^3),$
we have by integration by parts:
\[
\int_{\mathbb R^3} \nabla U : \nabla w \sim \int_{\mathcal F^N} \nabla U : \nabla w 
\sim \sum_{i=1}^N \int_{\partial B(X_i^N,1/N)} \Sigma(U,P) n \cdot w {\rm d}\sigma  , 
\]  
where $\Sigma(U,P) =  (\nabla U + \nabla^{\top} U) - P \mathbb I_3$ is the fluid stress tensor and $n$ is the normal to $\partial B(X_i^N,1/N)$ directed inward the obstacle.
In the favorable configurations under consideration here, we can replace $w$ -- in the boundary integrals on the right-hand side -- by the value in the center of $B(X_i^N,1/N)$ and compute the integral of the stress tensor on $\partial B(X_i^N,1/N)$ by using Stokes law (see \cite[formula (4)]{DGR}):
\[
\int_{\partial B(X_i^N,1/N)} \Sigma(U,P) n \cdot w {\rm d}\sigma \sim \int_{\partial B(X_i^N,1/N)} \Sigma(U,P) n \cdot w(X_i^N) {\rm d}\sigma
\sim \dfrac{6\pi}{N} \sum_{i=1}^N (V_i^N - \bar{U}_{i}^N) \cdot w(X_i^N) ,
\]
where $V_i^N - \bar{U}_{i}^N$ stands for the difference between the velocity on the obstacle $B(X_i^N,1/N)$ and the velocity ``at infinity'' seen by this obstacle. One important step of the analysis is to justify that we can choose for such asymptotic velocity
a mean of $U$ around $B(X_i^N,1/N).$ We obtain finally the identity:
\[
\int_{\mathbb R^3} \nabla U : \nabla w + \dfrac{6\pi}{N}\sum_{i=1}^N \bar{U}_i^N \cdot w(X_i^N)  \sim  \dfrac{6\pi}{N} \sum_{i=1}^N  V_i^N \cdot w(X_i^N).
\]
We recognize an identity of the form:
\[
\int_{\mathbb R^3} \nabla U : \nabla w + 6\pi \langle \rho^N[\mathbf Z^N], \bar{U} \rangle \sim 6\pi \langle j^N[\mathbf Z^N], w \rangle  .
\]
We  compare then this weak-formulation with the weak formulation for the 
Stokes-Brinkman problem \eqref{eq_SB}-\eqref{cab_SB}. Taking the difference between both formulations, we apply the duality argument above to relate
the $L^2_{\mathrm{loc}}$-norm of the difference $U^N[\mathbf{Z}^N]-u[\rho,j]$ to duality 
distances between $\rho^N[\mathbf Z^N]$ and $\rho,$ on the one hand, and
$j^N[\mathbf Z^N]$ and $j$, on the other hand. The core of the proof below
is to quantify the error terms induced by the symbol ``$\sim$'' above, especially 
to justify the application of Stokes law for ``favorable'' configurations.

%
%
%

\subsection{Main result of this section.}
To state the main result of this section, we recall the notations introduced in \cite{Hil17} to handle the 
convergence of $U_N$ towards $u[\rho,j].$ Given $N \in \mathbb N^*$ and  $\mathbf{Z} = (X_1,V_1,\ldots,X_N,V_N) \in \mathcal O^{N},$ we denote:
\begin{itemize}
\item $d_{\mathrm{min}}[\mathbf Z]$ the minimal distance between two different centers $X_i$;\\[-8pt]
\item  $\lambda[\mathbf Z]$ a chosen size for a partition of $\mathbb R^3$ in cubes; \\[-8pt]
\item $M[\mathbf Z]$  the maximum number of centers $X_i$ inside one cell of size $\lambda[\mathbf Z].$
\end{itemize} 
If $d_{\mathrm{min}}[\mathbf Z]$ is sufficiently large and $M[\mathbf Z]$ is sufficiently small, the particles are distant and do not concentrate in a small box.
This is the reason for the name ``non-concentrated configurations'' of this section.
With these latter notations, the main result of this section is the following estimate:
\begin{theorem} \label{theo_bonneconfig}
Let $\alpha \in (2/3,1),$ $\eta  \in (0,1),$ $R>0$ and $\delta >1/2$ be given.  There exists a positive constant $K := K(\alpha,R,\Omega_0)$ such that, for $N \geq 1$,
given $\mathbf Z^N \in\mathcal O^N$ such that
\begin{equation}\label{eq:parametres}
d_{\mathrm{min}}[\mathbf Z^N] \geq \dfrac{1}{N^{\alpha}}, \quad 
M[\mathbf Z^N] \leq \dfrac{N^{3(1-\alpha) /5}}{\eta}, \quad 
\lambda[\mathbf Z^N] = \left( \dfrac{\eta M[\mathbf Z^N]}{N}\right)^{\frac 13},
\end{equation}
we have
\begin{multline*}
\|U_N[\mathbf Z^N] - u[\rho,j]\|_{L^2(B(0,R))} 
 \leq \dfrac{K}{\eta}\Bigg[ 
 \|j[\mathbf Z^N] - j\|_{[C^{0,1/2}_b(\mathbb R^3)]^{*}}
 \\ 
\quad   + \left(1 +\frac{1}{N} \sum_{i=1}^N |V_i^N|^2 \right)^{\frac 54} 
\Bigg( 
 \dfrac{1 + \|\rho\|_{L^2(\Omega_0)}}{{\delta}^{1/3}} 
+ \delta^{6} \left( \dfrac{1}{N^{\frac{1-\alpha}{5}}}+ \|\rho[\mathbf Z^N] - \rho\|_{[C^{0,1/2}_b(\mathbb R^3)]^{*}} \right)\Bigg)   \Bigg] .
\end{multline*}
where we recall that
$$
\rho[\mathbf Z^N] = \dfrac{1}{N} \sum_{i=1}^N \delta_{X_i^N}\,, \qquad 
j[\mathbf Z^N] =  \dfrac{1}{N} \sum_{i=1}^N V_i^N \delta_{X_i^N}.
$$
\end{theorem}

The remainder of this section is devoted to the proof of this {theorem}. It is based on interpolating the method of \cite{MH}
for dilute suspensions with the construction of \cite{Hil17}. Though the computations follow the line of these previous reference,
we give an extensive proof for completeness because estimates have to be adapted at each line.

\medskip

\begin{proof}[Proof of Theorem~\ref{theo_bonneconfig}]
From now on, we pick $\alpha,\eta,\delta,R$ as in the assumptions of our Theorem~\ref{theo_bonneconfig}, $N \geq 1$  and $\mathbf{Z} = (X_1,V_1,\ldots,X_N,V_N) \in \mathcal O^N$ 
such that \eqref{eq:parametres} holds true.  For legibility, we forget the $N$-dependencies in many notations in the proof.  We  recall that, by assumption, $\mathrm{Supp}(\rho[\mathbf Z]) \cup \mathrm{Supp}(j[\mathbf Z]) \subset \Omega_0$ and we denote $\Omega_1 := \Omega_0 + B(0,1).$

To begin with, we note that, by applying the variational characterization associated with the Stokes problem (see \cite[Theorem 3]{Hil17}), 
we can construct a constant $C_0$ such that:
\begin{equation} \label{eq_defEinfini}
\|\nabla U[\mathbf Z]\|^2_{L^2(\mathbb R^3)} \leq \dfrac{C_0}{N}\sum_{i=1}^N |V_i|^2.
\end{equation}
This property relies mostly on the fact that $Nd_{\mathrm{min}}[\mathbf Z]$ is bounded below by a strictly positive constant. 
We refer the reader to \cite[Section 3]{Hil17} for more details.

We want to compute a bound by above on $\|U[\mathbf Z] - u[\rho,j]\|_{L^2(B(0,R))}.$
Applying {\bf Proposition~\ref{prop_SB2}}, this reduces to compute a bound for: 
\begin{multline*}
[U[\mathbf Z] - u[\rho,j]]_{\rho,2} := \sup\Bigl\{\left|\int_{\mathbb R^3} \nabla ( U[\mathbf Z] - u[\rho,j] ) : \nabla w + 6\pi \int_{\mathbb R^3} \rho ( U[\mathbf Z] - u[\rho,j])\cdot w \right|,  
w \in\mathcal D(\mathbb R^3) \text{ with } \\ \|\nabla w\|_{L^2(\mathbb R^3)} + \|\nabla^2 w\|_{L^2(\mathbb R^3)} \le 1     \Bigr\},
\end{multline*}
or to find a constant $K$ independent of $U[\mathbf Z]$ and $w \in \mathcal D(\mathbb R^3)$ for which there holds
$$
 \left|\int_{\mathbb R^3} \nabla (U[\mathbf Z] - u[\rho,j]) : \nabla w + 6\pi \int_{\mathbb R^3} \rho (U[\mathbf Z] - u[\rho,j])\cdot w \right| \leq K \left[ \|\nabla w\|_{L^2(\mathbb R^3)} + \|\nabla^2 w\|_{L^2(\mathbb R^3)} \right].
$$
Hence, in what follows we fix $w \in \mathcal D(\mathbb R^3)$ and we focus on:
$$
E[w] :=  \int_{\mathbb R^3} \nabla [U[\mathbf Z] - u[\rho,j]] : \nabla w + 6\pi \int_{\mathbb R^3} \rho (U[\mathbf Z] - u[\rho,j])\cdot w .
$$
We apply without mention below that, since $\Omega_1$ is bounded, there holds:
$$
\|w\|_{C^{0,1/2}(\overline{\Omega_1})} + \|\nabla w\|_{L^6(\Omega_1)} \lesssim  \|\nabla w\|_{L^2(\mathbb R^3)} + \|\nabla^2 w\|_{L^2(\mathbb R^3)} =: \|w\|_{D^2(\mathbb R^3)} .
$$

First, we decompose the error term $E[w]$ into several pieces that are treated independently in the rest of the proof.
Since $u[\rho,j]$ is the weak solution to the Stokes-Brinkman problem associated with $(\rho,j),$ this error term rewrites:
 \begin{eqnarray*}
E[w] & =&   \int_{\mathbb R^3} ( \nabla U[\mathbf Z]  : \nabla w + 6\pi \rho U[\mathbf Z] \cdot w )  - \int_{\mathbb R^3} ( \nabla u[\rho,j] : \nabla w + 6\pi \rho  u[\rho,j]\cdot w ),  \\
		&=&  \int_{\mathbb R^3} ( \nabla U[\mathbf Z]  : \nabla w + 6\pi \rho U[\mathbf Z] \cdot w )  - 6\pi \int_{\mathbb R^3} j \cdot w.
\end{eqnarray*}

We now work on the gradient term involved in this error:
$$
\int_{\mathbb R^3}  \nabla U[\mathbf Z]  : \nabla w,
$$
in the spirit of \cite{Hil17}. Applying the construction in \cite[Appendix B]{Hil17}, we obtain a covering $(T_{\kappa})_{\kappa \in \mathbb Z^3}$ of $\mathbb R^3$ with cubes of width $\lambda[\mathbf Z]$ such that, denoting
$$
\mathcal Z_{\delta} := \left\{i \in \{1,\ldots,N\} \text{ s.t.\ } \text{dist} \left( X_i\,, \bigcup_{\kappa \in \mathbb Z^3} \partial T_{\kappa} \right) < \dfrac{\lambda[\mathbf Z]}{\delta} \right\}\,,
$$
there holds:
\begin{equation} \label{eq_couloir}
\dfrac{1}{N}\sum_{i \in \mathcal Z_{\delta}} (1+|V_i|^2) \leq \dfrac{12}{\delta} \dfrac{1}{N} \sum_{i=1}^N (1+|V_i|^2).
\end{equation}
Moreover, keeping only the indices $\mathcal K$ such that $T_{\kappa}$ intersects the $1/N$ neighborhood of $\Omega_0,$ we obtain a covering $(T_{\kappa})_{\kappa \in \Kappa}$ 
of the $1/N$-neighborhood of $\Omega_0.$  We  do not make precise the set of indices $\Kappa$.  The only relevant property to our computations is that 
\begin{equation} \label{eq_propKappaN}
\card \Kappa \lesssim \frac{|\Omega_1|}{|\lambda|^3} \,.
\end{equation}
Associated with this covering, we introduce the following notations. For arbitrary $\kappa \in \Kappa,$
we set 
$$
 \mathcal I_{\kappa} := \{ i \in \{1,\ldots, N \} \text{ s.t. } X_{i} \in T_{\kappa}\}\,, \quad M_{\kappa}[\mathbf Z] := \card \mathcal{I}_{\kappa} \,.
$$
We note that, since $T_{\kappa}$ has width $\lambda[\mathbf Z],$ we have that $M_{\kappa}[\mathbf Z] \leq M[\mathbf Z]$
for all $\kappa.$ Moreover, by construction of $\Kappa,$ all the $X_i$ are included in one $T_{\kappa}$ so that 
the $(\mathcal I_{\kappa})_{\kappa \in \Kappa}$ realizes a partition of $\{1,\ldots,N\}.$

\medskip

We construct then an approximate test-function $w^{s}$ piecewisely on the covering of $\Omega_0.$ Given $\kappa \in \mathcal \Kappa,$ we set:
\begin{equation} \label{eq_wsN}
w^s_{\kappa}(x) = \sum_{i \in \mathcal I_{\kappa} \setminus \mathcal Z_{\delta}} G^N[w(X_i)](x-X_i)\,, \quad \forall \, x \in \mathbb R^3\,,
\end{equation}
where $G^N[v]$ is the unique weak solution to the Stokes problem outside $B(0,1/N)$ with vanishing condition at infinity  and constant boundary condition
equal to $v \in \mathbb R^3$ on $\partial B(0,1/N)$. Explicit formulas are available in textbooks and are recalled in Appendix {\ref{app_stokes}}. 
We set:
$$
w^{s} = \sum_{\kappa \in \Kappa} w^s_{\kappa}  \mathbf{1}_{T_{\kappa}}.
$$
We  note that $w^{s} \notin H^1_0(\mathbb R^3)$ because of jumps at interfaces $\partial T_{\kappa}.$ 
 It will be sufficient for our purpose that $w^s \in H^1(\mathring{T}_{\kappa})$ 
for arbitrary $\kappa \in \Kappa.$ Setting:
$$
E_0[w] := \int_{\mathbb R^3} \nabla U[\mathbf Z]  :\nabla w - \sum_{\kappa \in \Kappa} \int_{\mathbb R^3} \nabla U[\mathbf Z] : \nabla w_{\kappa}^s,
$$
we have:
$$
E[w] = E_0[w] + \sum_{\kappa \in \Kappa} \int_{T_{\kappa}} \nabla U[\mathbf Z]  : \nabla w_{\kappa}^s + 6\pi \int_{\mathbb R^3} \rho U[\mathbf Z]  \cdot w - 6\pi\int_{\mathbb R^3} j  \cdot w.
$$
Now for arbitrary $\kappa \in \Kappa,$ we apply in {\bf Section~\ref{sec_E1N}} the properties of $G^N$  and integrate by parts the 
integral on $T_{\kappa}.$ We obtain an integral on $\partial T_{\kappa}$ in which we approximate $U[\mathbf Z] $ by:
$$
\bar{u}_{\kappa} := \dfrac{1}{|[T_{\kappa}]_{2\delta}|} \int_{[T_{\kappa}]_{2\delta}} U[\mathbf Z] (x){\rm d}x,
$$
where $[T_{\kappa}]_{2\delta}$ is the $\lambda[\mathbf Z]/(2\delta)$-neighborhood of $\partial T_{\kappa}$
inside $\mathring{T}_{\kappa}.$ In this way we obtain that 
$$
 \int_{T_{\kappa}} \nabla U[\mathbf Z]  : \nabla w_{\kappa}^s = \dfrac{6 \pi}{N} \sum_{i \in \mathcal I_{\kappa} \setminus \mathcal Z_{\delta}}  w(X_i)    \cdot ( V_i -  \bar{u}_{\kappa} ) + Err_{\kappa}.
$$
where it will arise that $Err_{\kappa}$ is due to the approximation of $U[\mathbf Z]$ by $\bar{u}_{\kappa}$ on $\partial T_{\kappa}$ only. So, we set:
$$
E_1[w] = \sum_{\kappa \in \Kappa}  \left( \int_{T_{\kappa}} \nabla U[\mathbf Z] : \nabla w_{\kappa}^s -  \dfrac{6 \pi}{N} \sum_{i \in \mathcal I_{\kappa} \setminus \mathcal Z_{\delta}} w(X_i) \cdot (V_i - \bar{u}_{\kappa}) \right)
$$
and we rewrite:
\begin{multline*}
E[w] = E_0[w] + E_1[w] + \dfrac{6 \pi}{N} \sum_{\kappa \in \Kappa} \sum_{i \in \mathcal I_{\kappa} \setminus \mathcal Z_{\delta}} w(X_i) \cdot V_i -  \dfrac{6\pi}{N}  \sum_{\kappa \in \Kappa} \sum_{i \in \mathcal I_{\kappa} \setminus \mathcal Z_{\delta}} w(X_i) \cdot \bar{u}_{\kappa}  \\+ 6\pi \int_{\mathbb R^3} \rho U[\mathbf Z]  \cdot w - 6\pi \int_{\mathbb R^3} j \cdot w.
\end{multline*}
Eventually, we obtain:
\begin{equation} \label{eq_EN}
E[w] = E_0[w] + E_1[w] - E_{\rho}[w] + E_{j}[w]  ,
\end{equation}
where we denote:
\begin{align*}
& E_{j}[w] :=  \dfrac{6 \pi}{N} \sum_{\kappa \in \Kappa}  \sum_{i \in \mathcal I_{\kappa} \setminus \mathcal Z_{\delta}} w(X_i) \cdot V_i- 6 \pi \int_{\mathbb R^3} j \cdot w  ,\\
& E_{\rho}[w] := \sum_{\kappa \in \Kappa} \left[ \dfrac{6\pi}{N} \sum_{i \in \mathcal I_{\kappa} \setminus \mathcal Z_{\delta}} w(X_i)\right] \cdot \bar{u}_{\kappa}  -   6\pi \int_{\mathbb R^3} \rho U[\mathbf Z]  \cdot w.
\end{align*}

Applying successively Lemma~\ref{lem_EN0} , Lemma~\ref{lem_EN1}, Lemma~\ref{lem_ENj} and Lemma~\ref{lem_ENrho} below, and recalling \eqref{eq:parametres} to replace $\lambda[\mathbf Z],d_{\mathrm{min}}[\mathbf Z]$
and $M[\mathbf Z],$ we obtain respectively:
\begin{align*}
|E_0[w]| &  \lesssim \dfrac{1}{\eta}\left(\dfrac{1}{\delta^{2/3}} + \dfrac{1}{N^{\frac 25(1-\alpha)}} +  \dfrac{\delta}{N^{\frac 45 \alpha - \frac{2}{15}}} \right)^{\frac 12} \left( 1 + \frac 1N \sum_{i=1}^N |V_i|^2 \right) \|w\|_{D^2(\mathbb R^2)} , \\
|E_1[w]| & \lesssim \dfrac{\delta^6}{\sqrt{\eta}} \dfrac{1}{N^{\frac{2+3\alpha}{15}}}  \left( 1 + \frac 1N \sum_{i=1}^N |V_i|^2 \right)^{\frac 12} \|w\|_{D^2(\mathbb R^2)} , \\
|E_j[w]| & \lesssim \left( \|j[\mathbf Z] - j\|_{[C^{0,1/2}_b(\mathbb R^3)]^*} + \dfrac{1}{\delta} \left( 1+ \frac 1N\sum_{i=1}^N |V_i|^2 \right)\right)\|w\|_{D^2(\mathbb R^2)} , \\
|E_{\rho}[w]| & \lesssim \left[\dfrac{1}{\sqrt{\delta \sqrt{\eta}}}  + \dfrac{\|\rho\|_{L^2(\Omega_0)}}{\delta} + \delta^{\frac 92}\left( \dfrac{1}{N^{\frac{2+3\alpha}{15}}} + \|\rho[\mathbf Z] -\rho\|_{[C^{0,1/2}_b(\mathbb R^3)]^*} \right)\right] \left( 1+ \dfrac{1}{N} \sum_{i=1}^{N} |V_i|^2 \right)^{\frac 54}\|w\|_{D^2(\mathbb R^3)} .
\end{align*}
Gathering the above estimates, recalling that $\eta \in (0,1),$ $\delta >1/2$, 
and remarking that, since $2/3 \leqslant \alpha <1$ there holds 
\[ 
\frac{1-\alpha}{5} < \frac{2}{5}\left(\alpha - \frac 13\right) < \frac{2+3\alpha}{15},
\] 
we finally obtain: 
\begin{multline*}
|E[w]| \lesssim \dfrac 1\eta \Bigg[ \left( \dfrac{(1+\|\rho\|_{L^2(\Omega_0)})}{\delta^{1/3} }  
+  \delta^{6} \left(  \dfrac{1}{N^{\frac 15(1-\alpha)}} +  \|\rho[\mathbf Z] -\rho\|_{[C^{0,1/2}_b(\mathbb R^3)]^*} \right) \right) \left( 1+ \dfrac{1}{N} \sum_{i=1}^{N} |V_i|^2 \right)^{\frac 54} \\
+ \|j[\mathbf Z] - j\|_{[C^{0,1/2}_b(\mathbb R^3)]^*} \Bigg] \|w\|_{D^2(\mathbb R^3)},
\end{multline*}
which ends the proof of Theorem~\ref{theo_bonneconfig}. 
\end{proof}

We  proceed now to estimate the different error terms $E_0[w]$, $E_1[w]$, $E_j[w]$ and $E_{\rho}[w]$ appearing in the proof of Theorem~\ref{theo_bonneconfig} above.  
This is done in Sections~\ref{sec_EN0}, 
\ref{sec_E1N}, \ref{sec_ENj} and \ref{sec_ENrho}, respectively.

\subsection{Estimating $E_0[w]$.}\label{sec_EN0}
We recall that, with the notations above, there holds:
$$
E_0[w] =  \sum_{\kappa \in \Kappa} \left(  \int_{T_{\kappa}} \nabla U[\mathbf Z]  :\nabla w - \int_{T_{\kappa}} \nabla U[\mathbf Z]  : \nabla w_{\kappa}^s\right),
$$
We have the following result:
\begin{lemma} \label{lem_EN0}
For $N \geq 1$,  we have:
\begin{multline} \label{eq_decomposition}
|E_0[w]|
\lesssim  
 \left(1+\frac 1N \sum_{i=1}^N|V_i|^2\right)  \ldots \\
\ldots \left(\dfrac{1}{\delta^{2/3}} \left(1 + \dfrac{M[\mathbf Z]^2}{|Nd_{\mathrm{min}}[\mathbf Z]|^2} + \dfrac{M[\mathbf Z]^{\frac 53}}{{Nd_{\mathrm{min}}[\mathbf Z]}} + \dfrac{M[\mathbf Z]^2}{|Nd_{\mathrm{min}}[\mathbf Z]|^4}
 \right)+  \dfrac{M[\mathbf Z]}{Nd_{\mathrm{min}}[\mathbf Z]}+  \delta \dfrac{M[\mathbf Z]^{\frac 53}}{N\lambda[\mathbf Z]}  \right)^{1/2} \, \|w\|_{D^2(\mathbb R^3)}.
\end{multline}

\end{lemma}

\begin{proof}
The proof is a simpler version of \cite[Proposition 11]{Hil17} but keeping track of the dependencies on $w$ of all constants.

\medskip

First, we construct an intermediate test-function similar to \cite[pp. 25-26]{Hil17}. We recall here the ideas of the construction. For arbitrary $\kappa \in \Kappa,$ we consider the Stokes problem 
on $\mathring{T}_{\kappa} \setminus \bigcup_{i \in \mathcal I_{\kappa}\setminus \mathcal Z_{\delta}} \overline{B_i}$ with boundary conditions:
\begin{equation} \label{cab_stokesenplus}
\left\{
\begin{array}{rcll}
u(x) &=& w(x)  \,, & \text{ on $\partial B_{i}$ for $i \in  \mathcal I_{\kappa}\setminus \mathcal Z_{\delta}$}\,, \\[6pt]
u(x) &=& 0 \,, & \text{ on $\partial T_{\kappa}$}\,.
\end{array}
\right.
\end{equation}
The analysis of this problem is done in Appendix~\ref{app_stokes} and yields a solution 
%
%
$\bar{w}_{\kappa}$. We keep the symbol $\bar{w}_{\kappa}$ to denote its extension to 
 $\R^3$ (by $w$ on the holes and
by $0$ outside $\mathring{T}_{\kappa}^N$). We obtain a divergence-free $\bar{w}_{\kappa} \in H^1(\mathbb R^3)$ having support in $\Omega_1$. We then add the $\bar{w}_{\kappa}$ into:
$$
\bar{w} = \sum_{\kappa \in \Kappa} \bar{w}_{\kappa}\,.
$$  
and correct the values of $\bar{w}$ on the $B_{i}$ when $i \in \mathcal Z_{\delta}$ in order that it fits the same boundary
conditions as $w$ on the $B_i,$ $i=1,\ldots,N.$ We introduce $\chi^N$ a truncation function such that $\chi^N = 1$ in 
$B(0,1/N)$ and $\chi^N = 0$ outside $B(0,2/N)$ and we denote:
\begin{eqnarray*}
\tilde{w} &=& \sum_{i \in \mathcal Z_{\delta}} \left[ \chi^N(\cdot -X_i) w -  \mathfrak B_{X_i,\frac1N,\frac2N} [x \mapsto w(x) \cdot \nabla \chi^N(x-X_i)] \right]  \\
		&& \; + \prod_{i \in \mathcal Z_{\delta}} (1- \chi^N(\cdot -X_i)) \bar{w} + \sum_{i \in \mathcal Z_{\delta}} \mathfrak B_{X_i,\frac1N,\frac2N} [x \mapsto \bar{w}(x) \cdot \nabla \chi^N(x-X_i)]\,.
\end{eqnarray*}
where $\mathfrak B_{X,r_1,r_2}$ is the Bogovskii operator that lifts the divergence in bracket
with a vector-field in $H^1_0(B(X,r_2) \setminus \overline{B(X,r_1)})$.
Consequently, $w - \tilde{w} \in H^1_0(\mathcal F)$ is an available test-function in the weak-formulation of the
Stokes problem satisfied by $U[\mathbf Z].$ This yields:
$$
\int_{\mathbb R^3} \nabla U[\mathbf Z] : \nabla (w- \tilde{w}) = 0.
$$
We rewrite this identity as follows:
\begin{equation} \label{eq_cpreskesa}
E_0[w] =  \epsilon_1 + \epsilon_2\,,
\end{equation}
with:
\[
\epsilon_1 = \sum_{\kappa \in \Kappa} \int_{T_{\kappa}} \nabla U[\mathbf Z]  : \nabla (\bar{w}_{\kappa} - w^s_{\kappa} )  \,, \qquad
\epsilon_2 =  \int_{\Omega_1} \nabla U[\mathbf Z]  : \nabla ( \tilde{w} - \bar{w})\,.  
\]
We control now the error term $\epsilon_1.$
 For arbitrary $\kappa \in \Kappa,$ we apply {Proposition~\ref{prop_truncationprocess}} to $\bar{w}_{\kappa}$ (noting that $"d_{m}" = \min(d_{min}[\mathbf Z],\lambda[\mathbf Z]/\delta)$ and the remark at the end of Section \ref{app_stokes}) and we obtain: 
$$
 \|\nabla (w^s_{\kappa} - \bar{w}_{\kappa})\|_{L^2(T_{\kappa})} \lesssim  \dfrac{M_{\kappa}[\mathbf Z]}{N}\left(\dfrac{1}{{d_{\mathrm{min}}[\mathbf Z]}} +{\dfrac{\delta}{\lambda[\mathbf Z]}} \right)^{1/2}\left( \|w\|_{C^{0,1/2}(T_{\kappa})} + 
 \|\nabla w\|_{L^6(T_{\kappa})}\right)\,.
 $$
Introducing this bound in the computation of $\epsilon_1$ and recalling the  two properties of $M_{\kappa}[\mathbf Z]$  :
\begin{equation} \label{eq_fundamentalMkappa}
\sum_{\kappa \in \Kappa} M_{\kappa}[\mathbf Z] \leq N\,, \qquad \sup_{\kappa \in \Kappa} M_{\kappa} \leq M[\mathbf Z]\,,
\end{equation}
yield:
\begin{equation} \label{eq_boundE1}
|\epsilon_1 |  \lesssim \left(\dfrac{M[\mathbf Z]}{Nd_{\mathrm{min}}[\mathbf Z]} + \delta {\dfrac{M[\mathbf Z]}{N\lambda[\mathbf Z]}} \right)^{\frac 12} \|\nabla U[\mathbf Z]\|_{L^2(\mathbb R^3)} \|w\|_{D^2(\mathbb R^3)}\,.
 \end{equation}

We compute now a bound for $\epsilon_2.$ For this, we replace $\tilde{w}$ by its explicit construction. We recall that the supports of the $(\chi^N(\cdot - X_i))_{i \in \{1,\ldots,N\}}$ are disjoint so that:
$$
1 - \prod_{i \in \mathcal Z_{\delta}} (1- \chi^N(x-X_i)) =   \sum_{i \in \mathcal Z_{\delta}} \chi^N(x-X_i)\,, \quad \forall \, x \in \mathbb R^3.
$$
Consequently, we  split:
\begin{eqnarray*}
\bar{w} - \tilde{w} &=& \sum_{i \in  \mathcal Z_{\delta}} \left[ \chi^N(\cdot -X_i) \bar{w} -  \mathfrak B_{X_i,\frac1N,\frac2N} [x \mapsto \bar{w}(x) \cdot \nabla \chi^N(x-X_i)] \right]  \\
		&& \; -  \sum_{i \in \mathcal Z_{\delta}} \left[ \chi^N(\cdot - X_i) {w} -   \mathfrak B_{X_i,\frac1N,\frac2N} [x \mapsto {w}(x) \cdot \nabla \chi^N(x-X_i)]\right]\,.
\end{eqnarray*}
and $\nabla( \bar{w} - \tilde{w}) =  \sum_{i \in  \mathcal Z_{\delta}}  \sum_{\ell=1}^{3} \epsilon^{(\ell)}_{2,i} $ where, for $i\in \mathcal Z_{\delta},$
we denote:
\begin{align*}
&\epsilon_{2,i}^{(1)} = -\nabla \left[ \chi^N(\cdot - X_i) {w} -   \mathfrak B_{X_i,\frac1N,\frac2N} [x \mapsto {w}(x) \cdot \nabla \chi^N(x-X_i)]\right]  , \\
&\epsilon_{2,i}^{(2)} =  \nabla  \chi^N(\cdot -X_i) \otimes \bar{w} - \nabla \mathfrak B_{X_i,\frac1N,\frac2N} [x \mapsto \bar{w}(x) \cdot \nabla \chi^N(x-X_i)] ,\\
&\epsilon_{2,i}^{(3)} =  \chi^N(\cdot -X_i) \nabla \bar{w} .
\end{align*}
We remark here that $\epsilon_{2,i}^{(\ell)}$ has support in $B(X_i,2/N)$ whatever the value of $\ell.$ 
As previously, a standard Cauchy-Schwarz argument yields:
\begin{equation} \label{eq_epsilon2-full}
|\epsilon_{2}| \lesssim \|\nabla U[\mathbf Z]\|_{L^2(\mathbb R^3)} \left( \sum_{\ell=1}^3 \sum_{i\in \mathcal Z_{\delta}} |\epsilon_{2,i}^{(\ell)} |^2   \right)^{\frac 12}.
\end{equation}
To complete the proof, it remains to bound the last term in the right-hand side of the above inequality.

\medskip

First, by applying standard homogeneity properties of the Bogovskii operator (see \cite[App. A]{Hil17}) and explicit computations, we have, for $i\in \mathcal Z_{\delta}$:
\begin{align*}
\int_{B(X_i,2/N)} |\epsilon_{2,i}^{(1)}|^2&  \leq \dfrac{1}{N}\|w\|^2_{L^{\infty}(\Omega_1)}  + \|\nabla w\|^2_{L^2(B(X_i,2/N))}  \\
							& \lesssim \dfrac{1}{N} \left( \|w\|^2_{L^{\infty}(\Omega_1)} +   \|\nabla w\|^2_{L^6(B(X_i,2/N))} \right).
 \end{align*}
But, by the choice of the covering (see \eqref{eq_couloir}), we have:
\begin{equation} \label{eq_propdelta}
\sharp \mathcal Z_{\delta}  \lesssim \dfrac{N}{\delta}\left(1+\frac{1}{N}\sum_{i=1}^N |V_i|^2\right),
\end{equation}
so, we obtain finally:
\begin{equation} \label{eq_epsilon1}
\sum_{i\in \mathcal Z_{\delta}}  \int_{B(X_i,2/N)} |\epsilon_{2,i}^{(1)}|^2 \leq \dfrac{1}{\delta}  \left( 1+\frac{1}{N}\sum_{i=1}^N |V_i|^2\right) \|w\|^2_{D^2(\mathbb R^3)} .
\end{equation}
Secondly, with similar arguments as for $\epsilon_{2,i}^{(1)},$ we obtain, for $i \in \mathcal Z_{\delta}$:
$$
\int_{B(X_i,2/N)} |\epsilon_{2,i}^{(2)}|^2 \lesssim N^2 \|\bar{w}\|_{L^2(B(X_i,2/N))}^2
$$
and
\begin{eqnarray*}
\sum_{i\in \mathcal Z_{\delta}} \int_{B(X_i,2/N)} |\epsilon_{2,i}^{(2)}|^2 &\lesssim& N^2 \sum_{i \in \mathcal Z_{\delta}} \sum_{\kappa \in \Kappa} \|\bar{w}\|^2_{L^2(B(X_i,\frac 2N ) \cap T_{\kappa})} \,,  \\
		&\lesssim & N^2 \sum_{\kappa \in \Kappa}  \sum_{i \in \mathcal Z_{\delta}}  \|\bar{w}_{\kappa} - w^s_{\kappa} \|^2_{L^2(B(X_i,\frac2N) \cap T_{\kappa})} + 
		N^2\sum_{i \in \mathcal Z_{\delta}} \sum_{\kappa \in \Kappa} \|{w}_{\kappa}^s\|^2_{L^2(B(X_i,\frac2N) \cap T_{\kappa})}\,.
\end{eqnarray*}

We compute the first term on the last right-hand side thanks to the expansion \eqref{eq_Stokeslet} of $G^N$  and
remarking that, since the diameter of $B(X_i,\frac 2N)$ is infinitely smaller than the one of  $T_{\kappa}$ for $N$ sufficiently large, one 
$B(X_i,2/N)$ intersects at most $8$ distinct $T_{\kappa}.$ Repeating \eqref{eq_propdelta}, we conclude:
\begin{align*}
\sum_{i \in \mathcal{Z}_{\delta}} \sum_{\kappa \in \Kappa} \|{w}_{\kappa}^s\|^2_{L^2(B(X_i,\frac2N) \cap T_{\kappa})}& \lesssim \sum_{i \in \mathcal Z_{\delta}}8 \sup_{\kappa \in \Kappa} \|{w}_{\kappa}^s\|^2_{L^2(B(X_i,\frac2N))}\,, \\
& \lesssim   \dfrac{|M[\mathbf Z]|^2}{N^{4}d^2_{\mathrm{min}}[\mathbf Z]}\dfrac{1+ \frac 1N \sum_{i=1}^N |V_i|^2}{\delta} \|w\|^2_{L^{\infty}(\Omega_1)}.
\end{align*}

As for the other term,  we introduce, for $\kappa \in \Kappa,$ the set $\mathcal Z_{\delta,\kappa}$ of indices $i$ such that $B(X_i,\frac2N)\cap T_{\kappa} \neq \emptyset,$ 
and we obtain, by repeated use of H\"older's inequality, that:
\begin{eqnarray*}
\sum_{\kappa \in \Kappa}  \sum_{i \in \mathcal Z_{\delta}}  \|\bar{w}_{\kappa} - w^s_{\kappa} \|^2_{L^2(B(X_i,\frac2N) \cap T_{\kappa})}   
&=& \sum_{\kappa \in \Kappa}  \sum_{i \in \mathcal Z_{\delta,\kappa}}  \|\bar{w}_{\kappa} - w^s_{\kappa} \|^2_{L^2(B(X_i,\frac2N) \cap T_{\kappa})}    \\
&\lesssim& \sum_{\kappa \in \Kappa} \dfrac{|\sharp \mathcal Z_{\delta,\kappa}|^{\frac 23}}{N^2}  \| (\bar{w}_{\kappa} - w^s_{\kappa})\|^2_{L^6(T_{\kappa})}\\
&\lesssim &  \dfrac{1}{N^2} \left[ \sum_{\kappa \in \Kappa} \sharp \mathcal Z_{\delta,\kappa}\right]^{\frac 23}  \left( \sum_{\kappa \in \Kappa} \| (\bar{w}_{\kappa} - w^s_{\kappa})\|^6_{L^6(T_{\kappa})}\right)^{\frac 13}.
\end{eqnarray*}
By comparing the size of $T_{\kappa}$ and $B(X_i,2/N),$ we obtain again that:
$$
\left[ \sum_{\kappa \in \Kappa} \sharp \mathcal Z_{\delta,\kappa}\right]^{\frac 23} \lesssim \left| \sharp \mathcal Z_{\delta} \right|^{\frac 23} \lesssim \left[\dfrac{N}{\delta} \left(1+\frac 1N\sum_{i=1}^N|V_i|^2\right) \right]^{\frac 23} \, ,
$$
which, combined with Proposition~\ref{prop_truncationprocess} and \eqref{eq_fundamentalMkappa}, yields:
\[ 
\sum_{\kappa \in \Kappa} \sum_{i \in \mathcal Z_{\delta}} \| (\bar{w}_{\kappa} - w^s_{\kappa})\|^2_{L^2(B_{\infty}(X_i,\frac2N) \cap T_{\kappa})} 
\lesssim  \dfrac{(1+\frac 1N \sum_{i=1}^N|V_i|^2)^{\frac 23}}{\delta^{\frac 23}N^2} \dfrac{|M[\mathbf{Z}]|^{5/3} }{N} \left( \dfrac{1}{{d_{\mathrm{min}}[\mathbf Z]}} + {\dfrac{{\delta}}{\lambda[\mathbf Z]}}\right) \|w\|^2_{D^2(\mathbb R^3)} .
\]

Combining the above inequalities and recalling  \eqref{eq_defEinfini}, we conclude that:
\begin{multline} \label{eq_epsilon2}
\sum_{i\in \mathcal Z_{\delta}} \int_{B(X_i,2/N)} |\epsilon_{2,i}^{(2)}|^2 
 \lesssim  \left(1+\frac 1N \sum_{i=1}^N|V_i|^2\right) \dots \\ \ldots \left(\frac 1{\delta} \dfrac{|M[\mathbf Z]|^2}{|N d_{\mathrm{min}}[\mathbf Z]|^2} +\frac 1{\delta^{2/3}} \dfrac{|M[\mathbf Z]|^{5/3}}{N d_{\mathrm{min}}[\mathbf Z]} + {\delta}^{1/3} \dfrac{|M[\mathbf Z]|^{5/3}}{N \lambda[\mathbf Z]}  \right)    \|w\|^2_{D^2(\mathbb R^3)}.
\end{multline}
Finally, we have similarly:
\begin{eqnarray*}
\sum_{i\in \mathcal Z_{\delta}} \int_{B(X_i,2/N)} |\epsilon_{2,i}^{(3)}|^2 &\lesssim&  \sum_{i \in \mathcal Z_{\delta}} \sum_{\kappa \in \Kappa} \|\nabla \bar{w}\|^2_{L^2(B(X_i,2/N ) \cap T_{\kappa})}   \\
		&\lesssim &  \sum_{\kappa \in \Kappa}  \sum_{i \in \mathcal Z_{\delta}}  \|\nabla \bar{w}_{\kappa} - \nabla w^s_{\kappa} \|^2_{L^2(B(X_i,\frac2N) \cap T_{\kappa})} + 
		\sum_{i \in \mathcal Z_{\delta}} \sum_{\kappa \in \Kappa} \|\nabla {w}_{\kappa}^s\|^2_{L^2(B(X_i,\frac2N) \cap T_{\kappa})}\,.
\end{eqnarray*}
and we can reproduce the previous arguments relying on Proposition~\ref{prop_truncationprocess}. 
This yields, on the one hand:
\[
\sum_{i \in \mathcal Z_{\delta}} \sum_{\kappa \in \Kappa} \|\nabla {w}_{\kappa}^s\|^2_{L^2(B(X_i,\frac2N) \cap T_{\kappa})} \lesssim \dfrac{M^2[\mathbf Z]}{\delta|Nd_{\mathrm{min}}[\mathbf Z]|^4} \left(1+\frac 1N \sum_{i=1}^N|V_i|^2\right) \|w\|^2_{D^2(\mathbb R^3)}, 
\]
and, on the other hand:
\begin{align*}
  \sum_{\kappa \in \Kappa}  \sum_{i \in \mathcal Z_{\delta}}  \|\nabla \bar{w}_{\kappa} - \nabla w^s_{\kappa} \|^2_{L^2(B(X_i,\frac2N) \cap T_{\kappa})} 
 & \lesssim   \sum_{\kappa \in \Kappa} \|\nabla \bar{w}_{\kappa} - \nabla w^s_{\kappa} \|^2_{L^2(T_{\kappa})}  \\
& \lesssim    \dfrac{M[\mathbf{Z}]}{N} \left( \dfrac{1}{{d_{\mathrm{min}}[\mathbf Z]}} + {\dfrac{{\delta}}{\lambda[\mathbf Z]}}\right)\|w\|^2_{D^2(\mathbb R^3)} .
\end{align*}
We obtain finally that:
\begin{equation} \label{eq_epsilon3}
\sum_{i\in \mathcal Z_{\delta}} \int_{B(X_i,2/N)} |\epsilon_{2,i}^{(3)}|^2
  \lesssim  \left(1+\frac 1N \sum_{i=1}^N|V_i|^2\right)  \left( \frac 1{\delta} \dfrac{M^2[\mathbf Z]}{|Nd_{\mathrm{min}}[\mathbf Z]|^4} + \dfrac{M[\mathbf Z]}{N d_{\mathrm{min}}[\mathbf Z]} + \delta \dfrac{M[\mathbf Z]}{N\lambda[\mathbf Z]}\right)
\|w\|^2_{D^2(\mathbb R^3)} .
\end{equation}
Introducing \eqref{eq_epsilon1}, \eqref{eq_epsilon2} and \eqref{eq_epsilon3} into \eqref{eq_epsilon2-full} yields:
\begin{multline} \label{eq_boundE2}
|\epsilon_2|  \lesssim  \left(1+\frac 1N \sum_{i=1}^N|V_i|^2\right)  \ldots \\
\ldots \left(\dfrac{1}{\delta^{2/3}}\left(1 + \dfrac{M[\mathbf Z]^2}{|Nd_{\mathrm{min}}[\mathbf Z]|^2} + \dfrac{M[\mathbf Z]^{\frac 53}}{{Nd_{\mathrm{min}}[\mathbf Z]}} + \dfrac{M[\mathbf Z]^2}{|Nd_{\mathrm{min}}[\mathbf Z]|^4}
 \right)+  \dfrac{M[\mathbf Z]}{Nd_{\mathrm{min}}[\mathbf Z]}+  \delta \dfrac{M[\mathbf Z]^{\frac 53}}{N\lambda[\mathbf Z]}  \right)^{\frac 12} \|w\|_{D^2(\mathbb R^3)}.
\end{multline}
We complete the proof by combining \eqref{eq_boundE1}-\eqref{eq_boundE2}.
\end{proof}

\subsection{Estimating $E_1[w]$}  \label{sec_E1N}
We proceed with the computation of $E_1[w]$ defined by:
$$
E_1[w] = \sum_{\kappa \in \Kappa}  \left( \int_{T_{\kappa}} \nabla U[\mathbf Z] : \nabla w_{\kappa}^s -  6 \pi \sum_{i \in \mathcal I_{\kappa} \setminus \mathcal Z_{\delta}} w(X_i) \cdot (V_i - \bar{u}_{\kappa}) \right).
$$
We control this error term with the following lemma:
\begin{lemma} \label{lem_EN1}
Given $N \geq 1$, we have:
$$
\begin{aligned}
|E_1[w]| 
&\lesssim %
\delta^{6} \sqrt{\dfrac{M[\mathbf Z]}{N\lambda[\mathbf Z]}} \left(1+\frac 1N \sum_{i=1}^N|V_i|^2\right)^{\frac 12}  \|w\|_{D^2(\mathbb R^3)}.
\end{aligned}
$$

\end{lemma}
\begin{proof}
For $N$ sufficiently large and $\kappa \in \Kappa,$ let simplify at first:
 $$
 \tilde{I}_{\kappa} := \int_{T_{\kappa}} \nabla U[\mathbf Z] : \nabla w^{s}_{\kappa}\,.
 $$
By definition, we have that:
$$
w^{s}_{\kappa} (x) = \sum_{i \in \mathcal I_{\kappa} \setminus \mathcal Z_{\delta}} G^N[w(X_i)](x-X_i)\,, \quad \forall \, x \in \mathbb R^3\,,
$$
so that, introducing the associated pressures $ x \mapsto P^N[w(X_i^N)](x-X_i^N),$ we obtain
(after several integration by parts as depicted in \cite[pp. 32-33]{Hil17}):
\begin{equation} \label{eq_INkappa}
\tilde{I}_{\kappa} = \dfrac{6\pi}{N} \sum_{i \in \mathcal I_{\kappa} \setminus \mathcal Z_{\delta}} ( w(X_i) \cdot V_i - w(X_i) \cdot \bar{u}_{\kappa} )+ Err_{\kappa}
\end{equation}
with: 
$$
Err_{\kappa} = \int_{\partial T_{\kappa}}  \left\{ \sum_{i \in \mathcal I_{\kappa} \setminus \mathcal Z_{\delta}} 
 \partial_n  G^N[w(X_i)](\cdot-X_i) - P^N[w(X_i)](\cdot-X_i)n\right\} \cdot (U[\mathbf Z] - \bar{u}_{\kappa}) {\rm d}\sigma\,.
$$
Summing over $\kappa \in \Kappa,$ we obtain that
\begin{eqnarray*}
\sum_{\kappa \in \Kappa} \int_{T_{\kappa}} \nabla U[\mathbf Z] : \nabla w_{\kappa}^s &= & \sum_{\kappa \in \Kappa} \tilde{I}_{\kappa}  \\
		&=& \sum_{\kappa \in \Kappa}  \dfrac{6\pi}{N} \sum_{i \in \mathcal I_{\kappa} \setminus \mathcal Z_{\delta}} ( w(X_i) \cdot V_i - w(X_i) \cdot \bar{u}_{\kappa} ) +  \sum_{\kappa \in \Kappa} Err_{\kappa},
\end{eqnarray*}
and also:
$$
E_1[w] = \sum_{\kappa \in \Kappa} Err_{\kappa}.
$$

\medskip 
For $\kappa \in \Kappa,$ we adapt (up to notations) the computations of \cite[pp. 34-35]{Hil17}. 
The point here is to lift the boundary condition $U[\mathbf Z] - \bar{u}_{\kappa}$ via a standard 
truncation process in order to yield a divergence-free vector-field $v$ which vanishes at a distance $\lambda[\mathbf Z]/(2\delta)$ of $\partial T_{\kappa}.$ Applying that $(G^N[w(X_i)],P^N[w(X_i)])$ solves the
Stokes equation on $[T_{\kappa}]_{2\delta}$ (since this subset contains no holes with index in $\mathcal I_{\kappa} \setminus \mathcal Z_{\delta}$) we obtain:
\begin{equation} \label{eq_preskekappa}
|Err_{\kappa}| \leq  C_{\mathfrak B}[2\delta](1+C_{PW}[2\delta])\left\{\sum_{i \in \mathcal I_{\kappa} \setminus \mathcal Z_{\delta}} 
 \|\nabla  G^N[w(X_i)](\cdot-X_i)\|_{L^2([T_{\kappa}]_{2\delta})} \right\}   \|\nabla U[\mathbf Z]\|_{L^2([T_{\kappa}]_{2\delta})}\,,
\end{equation}
where, denoting $A(0,1-1/\delta,1)$ the cubic annulus $]-1,1[^3 \setminus [-(1-1/\delta),1/\delta]^3,$ we used the symbols: 
\begin{itemize}
\item $C_{\mathcal B}[\delta]$ for the norm of the Bogovskii operator $\mathfrak B_{0,1-1/\delta,1}$
seen as a continuous linear mapping $L^2_{0}(A(0,1-1/\delta,1)) \to H^1_0(A(0,1-1/\delta,1)),$ 
\item $C_{PW}[\delta]$ for the constant of the Poincar\'e-Wirtinger inequality on $H^1(A(0,1-1/\delta,1)).$
\end{itemize}
The asymptotics of these constants when $\delta \to \infty$ are analyzed in 
Appendix~\ref{sec_constants}.

To bound the first term on the right-hand side of this inequality, we remark again that for any $i \in \mathcal I_{\kappa} \setminus \mathcal Z_{\delta}$ the minimum distance between  $X_{i}$ and $[T_{\kappa}]_{2\delta}$ is larger than $\lambda[\mathbf Z]/(2\delta).$ Hence, applying the explicit formula \eqref{eq_Stokeslet} of the Stokeslet $G^N[w(X_i)]$ we obtain that
\begin{eqnarray*}
\|\nabla G^N[w(X_i)](\cdot - X_{i})\|_{L^2([T_{\kappa}]_{2\delta})} &\leq&  \left(\int_{\lambda[\mathbf Z]/(2\delta)}^{\infty} \dfrac{{\rm d}r}{N^2r^2} \right)^{\frac 12}|w(X_i)| \\
&\leq& \dfrac{\sqrt{2\delta}}{N\sqrt{\lambda[\mathbf Z]}} |w(X_i)|\,.
\end{eqnarray*} 
Combining these computations for the (at most) $M_{\kappa}[\mathbb Z]$ indices $i \in \mathcal I_{\kappa} \setminus \mathcal Z_{\delta}$   entails  that:
\begin{equation} \label{eq_controlStokeslet}
\sum_{i \in \mathcal I_{\kappa} \setminus \mathcal Z_{\delta}} 
 \|\nabla  G^N[w(X_i)](\cdot-X_i)\|_{L^2([T_{\kappa}]_{2\delta})} \lesssim \dfrac{M_{\kappa}[\mathbb Z]}{N} \sqrt{\dfrac{2\delta}{\lambda[\mathbf Z]}} \| w \|_{L^\infty (\Omega_1)} \,.
\end{equation}
Plugging \eqref{eq_controlStokeslet} into \eqref{eq_preskekappa} and recalling the fundamental properties \eqref{eq_fundamentalMkappa} of $M_{\kappa}[\mathbb Z]$ we conclude that
\begin{eqnarray*}
|E_1[w]| & \lesssim  &  C_{\mathfrak B}[2\delta](1+C_{PW}[2\delta])\sqrt{\dfrac{2\delta  M[\mathbf Z]}{N\lambda[\mathbf Z]}} \|\nabla U[\mathbf Z]\|_{L^2(\mathbb R^3)}  \|w\|_{C^{0,1/2}(\overline{\Omega_1})}. 
\end{eqnarray*}
We conclude the proof of Lemma~\ref{lem_EN1} by applying that $C_{\mathfrak B}[2\delta](1+C_{PW}[2\delta]) \lesssim \delta^{11/2}$ (see Appendix~\ref{sec_constants}) and recalling \eqref{eq_defEinfini}.
\end{proof}

\subsection{Estimating $E_{j}[w]$.} \label{sec_ENj}
We proceed with the error term
$$
E_j[w] =  \sum_{\kappa \in \Kappa} \dfrac{6 \pi}{N} \sum_{i \in \mathcal I_{\kappa} \setminus \mathcal Z_{\delta}} w(X_i) \cdot V_i- 6\pi \int_{\mathbb R^3} j \cdot w.
$$
\begin{lemma} \label{lem_ENj}
Given $N \geq 1,$ there holds:
$$
\left| E_j[w] \right| \lesssim \left( \|j[\mathbf Z]-j\|_{[C^{0,1/2}_b(\mathbb R^3)]^*} +\frac 1{\delta} \left(1+\dfrac{1}{N} \sum_{i=1}^N |V_i|^2 \right) \right) \|w\|_{D^2(\mathbb R^3)}\,. 
$$
\end{lemma}
\begin{proof}
As $w \in C^{\infty}_c(\R^3)$ and $(T_{\kappa})_{\kappa \in \Kappa}$ is a covering of $\text{Supp}(j[\mathbf Z])$ we have that:
$$
\sum_{\kappa \in \Kappa} \sum_{i \in \mathcal I_{\kappa}} w(X_i) \cdot V_i  = \langle j[\mathbf Z] , w \rangle. 
$$
Consequently, complementing the sum in $E_j$ with the indices in $\mathcal Z_{\delta},$ we have:
$$
E_{j}[w] = 6\pi \langle j[\mathbf Z] -j ,w \rangle +\dfrac{6\pi}{N} \sum_{i \in \mathcal Z_{\delta}} w(X_i) \cdot V_i.
$$
The first term on the right-hand side is estimated straightforwardly:
$$
|\langle j[\mathbf Z] -j ,w \rangle| \leq \|j[\mathbf Z]-j\|_{[C^{0,1/2}(\mathbb R^3)]^*} \|w\|_{C^{0,1/2}_b(\overline{\Omega_1})} ,
$$
while repeating the proof of \cite[Lemma 15]{Hil17}, we obtain, for $N \geq 1$ :
$$
\left|\dfrac{6\pi}{N} \sum_{i \in \mathcal Z_{\delta} \cap \, \mathcal I_{\delta}} w(X_i^N) \cdot V_i^N \right|  \lesssim \dfrac{1}{\delta}\left(1 + \frac 1N \sum_{i=1}^N |V_i|^2 \right) \| w \|_{L^{\infty}(\Omega_1)},
$$
which yields the expected result and completes the proof of Lemma~\ref{lem_ENj}.
\end{proof}
 
 \subsection{Estimating $E_{\rho}[w]$}\label{sec_ENrho}
 We end up by estimating the remainder term
 $$
E_{\rho}[w] =   \sum_{\kappa \in \Kappa} \left[ \dfrac{6\pi}{N} \sum_{i \in \mathcal I_{\kappa} \setminus \mathcal Z_{\delta}} w(X_i)\right] \cdot \bar{u}_{\kappa}  -   6\pi \int_{\mathbb R^3} \rho \,  U[\mathbf Z] \cdot w. 
 $$
\begin{lemma} \label{lem_ENrho}
For $N$ sufficiently large, there holds:
\begin{multline*}
|E_{\rho}[w]| \lesssim  \Biggl( \dfrac{1}{\delta^{\frac 52}} \left( \sqrt{\dfrac{M[\mathbf Z]}{N|\lambda[\mathbf Z]|^3}}+ \|\rho\|_{L^2(\Omega_0)}\right) + 
 \dfrac{1}{\sqrt{\delta}} \left(  \dfrac{M[\mathbf Z]}{N|\lambda[\mathbf Z]|^3} \right)^{1/4} 
 \\
 +\delta^{9/2}  \left( \lambda[\mathbf Z] +  \|\rho[\mathbf Z] - \rho\|_{[C^{0,1/2}_b(\mathbb R^3)]^*} \right) \Biggr)
   \left(1+ \frac 1N \sum_{i=1}^N |V_i|^2 \right)^{\frac 54} 
     \|w\|_{D^2(\mathbb R^3)}.
\end{multline*}

\end{lemma}

\begin{proof} The proof is adapted from \cite[Proposition 3.7]{MH}. As previously, let first complete the sum by reintroducing the $\mathcal Z_{\delta}$ indices:
\begin{equation} \label{eq_rewriterho}
\dfrac{6\pi}{N} \sum_{\kappa \in \Kappa} \sum_{i \in \mathcal I_{\kappa} \setminus \mathcal Z_{\delta}} w(X_i) \cdot \bar{u}_{\kappa} 
= \dfrac{6\pi}{N} \sum_{\kappa \in \Kappa} \sum_{i \in \mathcal{I}_{\kappa}} w(X_i) \cdot \bar{u}_{\kappa}  - \widetilde{E}rr
\end{equation}
where:
$$
\widetilde{E}rr = \dfrac{6\pi}{N} \sum_{\kappa \in \Kappa} \sum_{i \in \mathcal I_{\kappa} \cap \mathcal Z_{\delta}} w(X_i) \cdot \bar{u}_{\kappa} .
$$
We have then:
$$
E_{\rho}[w] = \dfrac{6\pi}{N}  \sum_{\kappa \in \Kappa} \sum_{i \in \mathcal{I}_{\kappa}} w(X_i) \cdot \bar{u}_{\kappa} - 6\pi \int_{\Omega_1} \rho \, U[\mathbf Z] \cdot w - \widetilde{E}rr.
$$
We remark that we may rewrite the first term on the right-hand side of this equality by introducing:
$$
\sigma= 
\left(1 - \left(1-\frac 1{2\delta}\right)^{3}\right)^{-1} \dfrac{1}{N|\lambda[\mathbf Z]|^3} \sum_{\kappa \in \Kappa}  \left(\sum_{i \in \mathcal I_{\kappa}} w(X_i) \right) \mathbf{1}_{[T_{\kappa}]_{2\delta}}\,,
$$
which yields
$$
E_{\rho}[w] = 6\pi \int_{\Omega_1} [\sigma - \rho w ] \cdot U[\mathbf Z] -  \widetilde{E}rr.
$$
Finally, we introduce $U_{\delta}[\mathbf Z] := U[\mathbf Z] * \zeta_{\delta^3}$ in this identity (in order to regularize $U[\mathbf Z]$ so that we may make the difference between $\rho[\mathbf Z]$ and $\rho$ appear) where we recall that $(\zeta_n)_n$ is a sequence of mollifiers. We apply below that
\begin{equation} \label{eq_convudelta}
\|U_{\delta}[\mathbf Z]\|_{C^{0,1}(\overline{\Omega_1})} \lesssim \delta^{\frac92} \| \nabla U[\mathbf Z] \|_{L^2(\mathbb R^3)} , \qquad  \|U[\mathbf Z]_{\delta} - U[\mathbf Z]\|_{L^2(\Omega_1)} \lesssim \dfrac{\|\nabla U[\mathbf Z]\|_{L^2(\mathbb R^3)}}{\delta^{3}}.
\end{equation}
Indeed, by classical computations there holds
\[
\|U_\delta[\mathbf Z]\|_{C^{0,1}(\overline{\Omega_1})} 
\lesssim \| U_{\delta}[\mathbf Z] \|_{L^{\infty}(\Omega_1)} + \| \nabla U_{\delta}[\mathbf Z] \|_{L^\infty(\Omega_1)} 
\lesssim 
\| \nabla U[\mathbf Z] \|_{L^2(\mathbb R^3)}  
(1+  \|\zeta_{\delta^3} \|_{L^2(\mathbb R^3)}),
\]
which yields the first inequality, and moreover
\[
\begin{aligned}
\|U_\delta[\mathbf Z] - U[\mathbf Z]\|_{L^2(\Omega_1)}^2 
&= \int \left| \int_{|z|\le 1} [U[\mathbf Z](x-\tfrac{z}{\delta^3}) - U[\mathbf Z](x)] \zeta(z) \mathrm d z  \right|^2 \mathrm d x \\
& \lesssim \int \int_{|z|\le 1} | U[\mathbf Z](x-\tfrac{z}{\delta^3}) - U[\mathbf Z](x)|^2 \mathrm d z \mathrm d x \\
&\lesssim \frac{1}{\delta^6}  \int \int_{|z|\le 1} |z|^2 \int_0^1 |\nabla U[\mathbf Z](x-t\tfrac{z}{\delta^3})|^2 \mathrm d t \mathrm d z \mathrm d x \\
&\lesssim \frac{1}{\delta^6} \| \nabla U[\mathbf Z] \|_{L^2(\mathbb R^3)}^2,
\end{aligned}
\]
which implies the second one.

This entails that:
$$
E_{\rho} [w] = \check{E}rr + \widehat{E}rr  - \widetilde{E}rr
$$
where
$$
 \check{E}rr = 6\pi \int_{\Omega_1} (\sigma - \rho w)\cdot (U[\mathbf Z] - U_{\delta}[\mathbf Z])\,,  \qquad \widehat{E}rr  =  6\pi \int_{\Omega_1} (\sigma - \rho w) \cdot U_{\delta}[\mathbf Z].
$$
We proceed by estimating these three error terms independently. 

We first remark  that the $(\mathcal I_{\kappa})_{\kappa \in \Kappa}$ form a partition of $\{1,\ldots,N\}.$ This entails that:
$$
\|\sigma\|_{L^{1}(\Omega_1)} \leq \sum_{\kappa \in \Kappa} \dfrac{M_{\kappa}}{N} \|w\|_{L^{\infty}(\Omega_1)} \leq \|w\|_{L^{\infty}(\Omega_1)}.
$$
Straightforward computations imply also that:
\begin{eqnarray*}
\|\sigma\|_{L^{\infty}(\Omega_1)} & \leq&  \left(1 - \left(1-\frac 1{2\delta}\right)^{3}\right)^{-1} \dfrac{M[\mathbf Z]}{N|\lambda[\mathbf Z]|^3} \|w\|_{L^{\infty}(\Omega_1)} \\
									&\lesssim &  \delta \dfrac{M[\mathbf Z]}{N|\lambda[\mathbf Z]|^3} \|w\|_{L^{\infty}(\Omega_1)} .
\end{eqnarray*} 
By interpolating the above inequalities to control the $L^2$-norm of $\sigma$ and combining with  \eqref{eq_convudelta}, 
we deduce:
\begin{eqnarray} \notag
|\check{E}rr|  &\lesssim &  \left( \|\sigma\|_{L^2(\Omega_1)} + \|\rho\|_{L^2(\Omega_1)} \|w\|_{L^{\infty}(\Omega_1)}\right) \| U_\delta[\mathbf Z] - U[\mathbf Z] \|_{L^2(\Omega_1)} \\
& \lesssim & \left( \sqrt{\delta \dfrac{M[\mathbf Z]}{N|\lambda[\mathbf Z]|^3}}+ \|\rho\|_{L^2(\Omega_1)}\right)\dfrac{\|\nabla U[\mathbf Z]\|_{L^2(\mathbb R^3)}}{\delta^{3}}  \|w\|_{L^{\infty}(\Omega_1)} \notag  \\
& \lesssim & \dfrac{1}{\delta^{5/2}} \left( \sqrt{\dfrac{M[\mathbf Z]}{N|\lambda[\mathbf Z]|^3}}+ \|\rho\|_{L^2(\Omega_1)}\right) \|\nabla U[\mathbf Z]\|_{L^2(\mathbb R^3)} \|w\|_{L^{\infty}(\Omega_1)} .
\label{eq_check} 
\end{eqnarray}

Then, we note that we may rewrite:
$$
\widehat{E}rr = \dfrac{6\pi}{N} \sum_{\kappa \in \Kappa} \sum_{i \in \mathcal I_{\kappa}} \int_{[T_{\kappa}]_{2\delta}}   \dfrac{w(X_i) \cdot U_{\delta}[\mathbf Z](x)}{|[T_{\kappa}]_{2\delta}|} - 6\pi \int_{\Omega_1} \rho U_{\delta}[\mathbf Z] \cdot w
$$
where we rewrite the first term:
\begin{multline*}
\dfrac{6\pi}{N} \sum_{\kappa \in \Kappa} \sum_{i \in \mathcal I_{\kappa}} 
\int_{[T_{\kappa}]_{2\delta}}   \dfrac{w(X_i) \cdot U_{\delta}[\mathbf Z]}{|[T_{\kappa}]_{2\delta}|} \\
 = \dfrac{6\pi}{N}\sum_{i=1}^N w(X_i) \cdot U_{\delta}[\mathbf Z](X_i) 
+ \dfrac{6\pi}{N} \sum_{\kappa \in \Kappa} \sum_{i\in \mathcal I_{\kappa}} 
\int_{[T_{\kappa}]_{2\delta}} \dfrac{w(X_i) \cdot (U_{\delta}[\mathbf Z] - U_{\delta}[\mathbf Z](X_i))}{|[T_{\kappa}]_{2\delta}|} .
\end{multline*}
Because $U_{\delta}[\mathbf Z]$ is Lipschitz, and by the estimate \eqref{eq_convudelta} on its Lipschitz norm, we have:
\begin{eqnarray*}
\left| \dfrac{6\pi}{N} \sum_{\kappa \in \Kappa} \sum_{i\in \mathcal I_{\kappa}} \int_{[T_{\kappa}]_{2\delta}} \dfrac{w(X_i) \cdot (U_{\delta}[\mathbf Z] - U_{\delta}[\mathbf Z](X_i))}{|[T_{\kappa}]_{2\delta}|} \right| & \lesssim & \lambda[\mathbf{Z}] \|U_{\delta}[\mathbf Z]\|_{C^{0,1}(\overline{\Omega_1})}\|w\|_{L^{\infty}(\Omega_1)}  \\
& \lesssim &  \delta^{9/2} \lambda[\mathbf Z]\|\nabla U[\mathbf Z]\|_{L^2(\mathbb R^3)} 
\|w\|_{L^{\infty}(\Omega_1)} .
\end{eqnarray*}
On the other hand, we have:
$$
\dfrac{6\pi}{N}\sum_{i=1}^N w(X_i) \cdot U_{\delta}[\mathbf Z](X_i)  - 6\pi \int_{\Omega_1} \rho\, U_{\delta}[\mathbf Z] \cdot w = 6\pi \langle \rho[\mathbf Z] - \rho , w \cdot U_{\delta} [\mathbf Z]\rangle  
$$
so that, introducing again the control on the $C^{0,1}$-norm of $U_{\delta}[\mathbf Z]$, we derive:
$$
\left|\dfrac{6\pi}{N}\sum_{i=1}^N w(X_i) \cdot U_{\delta}[\mathbf Z](X_i)  - 6\pi \int_{\Omega_1} \rho U_{\delta}[\mathbf Z] \cdot w \right| \lesssim \delta^{9/2}
\|\nabla U[\mathbf Z]\|_{L^2(\mathbb R^3)} \|\rho[\mathbf Z] - \rho\|_{[C^{0,1/2}_b(\mathbb R^3)]^*} \|w\|_{C^{0,1/2}(\overline{\Omega_1})}.
$$
We finally obtain
\begin{equation} \label{eq_hat}
|\widehat{E}rr| \lesssim  \delta^{9/2}  \left( \lambda[\mathbf Z] +  \|\rho[\mathbf Z] - \rho\|_{[C^{0,1/2}(\mathbb R^3)]^*} \right) \|\nabla U[\mathbf Z]\|_{L^2(\mathbb R^3)}  \|w\|_{C^{0,1/2}(\overline{\Omega_1})},
\end{equation}
which completes the proof for the term $\widehat{E}rr$.

For the remaining term, we introduce:
$$
\tilde{\sigma}  = \left(1 - \left(1-\frac 1{2\delta}\right)^{3}\right)^{-1} \dfrac{1}{N|\lambda[\mathbf Z]|^3} \sum_{\kappa \in \Kappa}  \left(\sum_{i \in \mathcal I_{\kappa} \cap \mathcal Z_{\delta}} |w(X_i)| \right) \mathbf{1}_{[T_{\kappa}]_{2\delta}}\, ,
$$
so that:
$$
|\widetilde{E}rr| \leq \int_{\Omega_1} \tilde{\sigma}(x) |U[\mathbf Z](x)|{\rm d}x\,.
$$
With similar arguments as in the previous computations,  we have, applying \eqref{eq_couloir}:
$$
\|\tilde{\sigma}\|_{L^1(\Omega_1)} \leq \dfrac{1}{N} \card \mathcal Z_{\delta} \|w\|_{L^{\infty}(\Omega_1)} \leq \dfrac{1}{\delta} \|w\|_{L^\infty(\Omega_1)}  \left(1+\frac 1N \sum_{i=1}^N |V_i|^2\right).
$$
Furthermore, we have:
$$
\|\tilde{\sigma}\|_{L^{\infty}(\Omega_1)} \lesssim \delta \dfrac{M[\mathbf Z]}{N|\lambda[\mathbf Z]|^3} \|w\|_{L^{\infty}(\Omega_1)}.
$$
Consequently, by interpolation, we obtain:
$$
\|\tilde{\sigma}\|_{L^{\frac 43}(\Omega_1)} \lesssim  \dfrac{1}{\sqrt{\delta}}  \left(  \dfrac{M[\mathbf Z]}{N|\lambda[\mathbf Z]|^3} \right)^{1/4} \|w\|_{L^{\infty}(\Omega_1)} \left(1+\frac 1N \sum_{i=1}^N |V_i|^2\right)^{\frac 34}\,.
$$
Applying Sobolev embedding $\dot{H}^1(\mathbb R^3) \subset L^{4}(\Omega_1)$ with \eqref{eq_defEinfini} we conclude that:
\begin{equation} \label{eq_tilde}
|\widetilde{E}rr| \lesssim \dfrac{1}{\sqrt{\delta}} \left(1+ \frac 1N \sum_{i=1}^N |V_i|^2 \right)^{\frac 34}  \left(  \dfrac{M[\mathbf Z]}{N|\lambda[\mathbf Z]|^3} \right)^{1/4}  \|w\|_{C^{0,1/2}(\overline{\Omega_1})} \|\nabla U[\mathbf Z]\|_{L^2(\mathbb R^3)}\,.
\end{equation}

We conclude the estimate of $E_\rho[w]$ by adding up \eqref{eq_check}, 
\eqref{eq_hat}, \eqref{eq_tilde} and recalling \eqref{eq_defEinfini}.
\end{proof}

\section{Proof of the main result}\label{sec:maintheo}

We are now able to prove our main result Theorem~\ref{theo:main} as well as the Corollary~\ref{cor:main}. 

We hence consider the framework of Theorem~\ref{theo:main}. The main idea is to split the expectation we want to estimate into two parts: one taking into account the non-concentrated configurations (which has been treated in Section~\ref{sec:bonneconfig}), and the other taking into account the concentrated configurations (treated in Section~\ref{sec:assumption}). 

Let us fix $\alpha \in (2/3,1),$  $\eta = \min(1/(2C_1e),1)$ (see Assumption~\ref{assump:A1} or 
Proposition~\ref{prop:config} to remind the meaning of constant $C_1$) and $R>0$. Given $N \in \mathbb N^*$ we denote:
\[ 
M_N = N^{\frac{3(1-\alpha)}{5}} \quad \text{ and }  \quad 
\lambda_N = \left(\dfrac{\eta M_N}{N }\right)^{1/3}.
\] 
We can then introduce the corresponding decomposition of configurations with $N$ particles: 
$$
\mathcal O^N = \big( \mathcal O^N \setminus (\mathcal O^N_{\lambda_N, M_N} \cup \mathcal O^N_\alpha) \big) \cup \big(\mathcal O^N_{\lambda_N, M_N} \cup \mathcal O^N_\alpha \big).
$$
We emphasize that, since $\eta <1,$ for any $\mathbf Z^N \in \mathcal O^N \setminus (\mathcal O^N_{\lambda_N, M_N} \cup \mathcal O^N_\alpha),$ the associated configuration satisfies \eqref{eq:parametres}.

\subsection{Proof of Theorem~\ref{theo:main}}
We want to compute the expectation of the distance with $u:=u[\rho,j].$
We split the expectation into the non-concentrated configurations and the concentrated configurations as follows
$$
\begin{aligned}
\mathbb E \left[ \| U_N [\mathbf Z^N] - u \|_{L^2(B(0,R))} \right]
& = \mathbb E \left[ \mathbf 1_{\mathcal O^N \setminus (\mathcal O^N_{\lambda_N, M_N} \cup \mathcal O^N_\alpha)} (\mathbf Z^N) \, \| U_N [\mathbf Z^N] - u \|_{L^2(B(0,R))} \right] \\
&\quad
+ \mathbb E \left[ \mathbf 1_{ \mathcal O^N_{\lambda_N, M_N} \cup \mathcal O^N_\alpha} (\mathbf Z^N) \,\| U_N [\mathbf Z^N] - u \|_{L^2(B(0,R))} \right] \\
&=: I_1 + I_2.
\end{aligned}
$$

Let us first estimate the term $I_2$. Since we have chosen $\eta$ sufficiently small, 
Proposition~\ref{prop:config} entails that:
$$
\mathbb P(\mathbf Z^N \in \mathcal O^N_{\lambda_N,M_N} \cup \mathcal O^N_{\alpha}) \lesssim N^{-(3\alpha-2)} \to 0 \quad \text{ when $N \to \infty$.}
$$
Consequently, with Corollary~\ref{cor_unifL1} we obtain that:
\begin{eqnarray*}
\mathbb E \left[ \mathbf 1_{ \mathbf Z^N \in \mathcal O^N_{\lambda_N,M_N} \cup \mathcal O^N_{\alpha}} \,  \|\nabla U[ \mathbf Z^N]\|_{L^2(\mathbb R^3)}   \right] 
&\leq& \dfrac{K}{N^{3/2}} +  \mathbb P( \mathbf Z^N \in \mathcal O^N_{\lambda_N,M_N} \cup \mathcal O^N_{\alpha})^{\frac 12} \,  \mathbb E \left[  \frac1N\sum_{i=1}^N |V_i^N|^2   \right]^{\frac 12}  \\
&\lesssim& \frac{1}{N^{\frac{3 \alpha-2}{2}}} .  
\end{eqnarray*}
Finally we get  
$$
\begin{aligned}
I_2 
&\lesssim
 \mathbb E \left[ \mathbf 1_{ \mathcal O^N_{\lambda_N, M_N} \cup \mathcal O^N_\alpha} (\mathbf Z^N) \,\| U_N [\mathbf Z^N] \|_{D(\R^3)} \right]
+ \mathbb E \left[ \mathbf 1_{ \mathcal O^N_{\lambda_N, M_N} \cup \mathcal O^N_\alpha} (\mathbf Z^N) \,\|  u \|_{L^2(B(0,R))} \right] \\
&\lesssim
\frac{1}{N^{\frac{3 \alpha-2}{2}}}
+ \mathbb P \left[ \mathbf Z^N \in \mathcal O^N_{\lambda_N, M_N} \cup \mathcal O^N_\alpha  \right] 
\lesssim
\frac{1}{N^{\frac{3 \alpha-2}{2}}},
\end{aligned}
$$

We now turn to the term $I_1$. For $N$ sufficiently large, noting that $\|\rho^N[\mathbf Z^N] - \rho \|_{[C^{0,1/2}_b(\mathbb R^3)]^{*}}\leq 2$, we can apply Theorem~\ref{theo_bonneconfig} choosing 
\[
\delta= \left[ \dfrac{1+ \|\rho\|_{L^2(\Omega_0)}}{\left( \frac{1}{N^{\frac{1-\alpha}{5}}}  + \|\rho^N[\mathbf Z^N] - \rho \|_{[C^{0,1/2}_b (\mathbb R^3)]^{*}} \right)}\right]^{\frac 3{19}}.
\]  
This yields that, for arbitrary $\mathbf Z^N \in \mathcal O^N \setminus (\mathcal O^N_{\lambda_N, M_N} \cup \mathcal O^N_\alpha)$, we have:
\begin{multline*}
\| U_N [\mathbf Z^N] - u \|_{L^2(B(0,R))}
\lesssim \left(1+\frac1N\sum_{i=1}^N |V_i^N|^2 \right)^{\frac54}
 \left( 1 + \|\rho\|_{L^2(\Omega_0)} \right)^{\frac{18}{19}} 
\left( \frac{1}{N^{\tfrac{1-\alpha}{5}}}  + \|\rho^N[\mathbf Z^N] - \rho \|_{[C^{0,1/2}_b (\R^3)]^{*}}  \right)^{\frac 1{19}}
\\
+ \|j^N[\mathbf Z^N] - j\|_{[C^{0,1/2}_b (\R^3)]^{*}}.
\end{multline*}
Taking expectation and using the hypotheses of the theorem, this yields
\begin{equation}\label{eq:I1}
\begin{aligned} 
I_1
&\lesssim 
\dfrac{(1 + \|\rho\|_{L^2(\Omega_0)})^{\frac{18}{19}} }{N^{\frac{1-\alpha}{95}}} \quad \mathbb    E \left[ \left(1+\frac1N\sum_{i=1}^N |V_i^N|^2 \right)^{\frac54} \right]  \\
& \quad 
+ \mathbb    E \left[ \left(1+\frac1N\sum_{i=1}^N |V_i^N|^2  \right)^{\frac54} \|\rho[\mathbf Z^N] - \rho\|^{\frac 1{19}}_{[C^{0,1/2}_b(\mathbb R^3)]^{*}}  \right] 
+ 
 \mathbb    E \left[ \|j[\mathbf Z] - j\|_{[C^{0,1/2}_b(\mathbb R^3)]^{*}}  \right]
\\
&\lesssim [M_5(F^N_1)]^{\frac 12} 
\left( \dfrac{(1+\|\rho\|_{L^2(\Omega_0)})^{\frac{18}{19}}}{N^{\frac{1-\alpha}{95}}} +  \mathbb E \left[ \|\rho^N[\mathbf Z^N] - \rho\|_{[C^{0,1/2}_b (\mathbb R^3)]^{*}}  \right]^{\frac 1{19}} \right)
+\mathbb E \left[  \| j^N [\mathbf Z^N] - j \|_{[C^{0,1/2}_b (\mathbb R^3)]^*}  \right]  \\
& \lesssim   \mathbb E \left[ \|\rho^N[\mathbf Z^N] - \rho\|_{[C^{0,1/2}_b(\mathbb R^3)]^{*}}  \right]^{\frac 1{19}} 
+\mathbb E \left[  \| j^N [\mathbf Z^N] - j \|_{[C^{0,1/2}_b (\R^3)]^*}  \right]
+ N^{-\frac{(1-\alpha)}{95}} \\
& \lesssim   \mathbb E \left[ W_1 ( \rho^N[\mathbf Z^N] , \rho ) \right]^{\frac{1}{57}} 
+\mathbb E \left[  \| j^N [\mathbf Z^N] - j \|_{[C^{0,1}_b (\R^3)]^*}  \right]^{\frac{1}{3}}
+ N^{-\frac{(1-\alpha)}{95}},
\end{aligned}
\end{equation}
where we have used Lemma~\ref{lem:equiv-distances} in last line.

We complete the proof of \eqref{eq:theo-erreur-0} by gathering previous estimates, and the last part of the theorem immediately follows from it. \qed

\subsection{Proof of the Corollary~\ref{cor:main}}
Let $f$ satisfy the hypotheses of Corollary~\ref{cor:main}. We shall construct here a sequence $(F^N)_{N \in \N^*}$ of symmetric probability measures on $\mathcal O^N$ that satisfy Assumption~\ref{assump:A1} and that is $f$-chaotic with quantitative estimates (in the sense of Definition~\ref{def:chaos}), hence also satisfies Assumption~\ref{assump:A2} thanks to Lemma~\ref{lem:chaos-A2}.

A classical way in statistical physics to construct chaotic probability measures in the phase space of a $N$-particle system is to take the $N$-tensor product of a probability measure on the phase space of one particle that we condition to the energy surface of the system. More precisely, given a probability measure $f$ on $\Omega_0 \times \R^3$ we define a probability measure $\Pi^N[f]$ on $\mathcal O^N$ by
\begin{equation}\label{eq:cor-fN}
\Pi^N[f] (\d \mathbf z^N) :=  \mathcal W_N^{-1}(f) \, \mathbf 1_{\mathbf z^N \in \mathcal O^{N}} \, f^{\otimes N}(\d \mathbf z^N),
\end{equation}
where $\mathcal W_N (f)$ is the partition function
$$
 \mathcal W_N (f) := \int_{(\Omega_0 \times \mathbb R^3)^N} \mathbf 1_{\mathbf z^N \in \mathcal O^{N}} \, f^{\otimes N}( \mathrm d \mathbf z^N).
$$
We now verify that the sequence $(\Pi^N[f])_{N \in \N^*}$ satisfies Assumption~\ref{assump:A1}. We start with a technical remark:

\begin{lemma}\label{WN}
For any $1 \le m \le N$ and $N$ large enough there holds
$$
1 \le \mathcal W_N^{-1}(f) \mathcal W_{N-m}(f) \le (1 - 8c_0 \,  N^{-2} \| \rho \|_{L^\infty(\mathbb R^3)} )^{-m} \le e^{16 c_0 m N^{-2} \| \rho \|_{L^{\infty}(\mathbb R^3)} },
$$
where $c_0=|B_{\mathbb R^3}|$ is the volume of the unit ball in $\mathbb R^3$.
\end{lemma}
\begin{proof}
We have
$$
\begin{aligned}	
\mathcal W_{m+1}(f) 
&= \int_{(\mathbb R^3 \times \mathbb R^3)^{m+1}} \mathbf 1_{ (z_1, \dots, z_m) \in \mathcal O^m[\frac{1}{N}]} \, \left(  \prod_{i=1}^{m} \mathbf 1_{|x_i - x_{m+1}| > \frac{2}{N} } \right) f^{\otimes (m+1)} (z_1, \dots, z_m, z_{m+1}) \, \mathrm d z_{1} \dots \mathrm d z_{m+1} \\
&= \int_{(\mathbb R^3 \times \mathbb R^3)^{m}} \left\{  \int_{\mathbb R^3 \times \mathbb R^3} \prod_{i=1}^{m} \left( 1 - \mathbf 1_{|x_i - x_{m+1}| \le \frac2N} \right) f(z_{m+1}) \, \mathrm dz_{m+1} \right\} \mathbf 1_{(z_1, \dots, z_m) \in \mathcal O^m[\frac{1}{N}]} \, f^{\otimes m} (z_1, \dots, z_m) \, \mathrm d z_{1} \dots \mathrm d z_{m} \\
&\ge \int_{(\mathbb R^3 \times \mathbb R^3)^{m}} (1 - 8 m c_0 N^{-3} \| \rho \|_{L^{\infty}(\mathbb R^3)} ) \mathbf 1_{(z_1, \dots, z_m) \in \mathcal O^m[\frac{1}{N}]} \, f^{\otimes m} (z_1, \dots, z_m) \, \mathrm d z_{1} \dots \mathrm d z_{m},
\end{aligned}
$$
We note here that, to pass from the second to the last line, we only remark that the indicator functions deletes at most $m$ balls of radius $2/N$ in $\mathbb R^3.$  
From the last inequality, we deduce $\mathcal W_{m+1}(f) \ge \mathcal W_{m}(f) (1 -8 m c_0 N^{-3} \| \rho \|_{L^{\infty}(\mathbb R^3)} ) $. We conclude the proof of the first claimed inequality by induction. 

For the second inequality, observe that $x \mapsto 2x + \log(1-x)$ is nonnegative for $0\le x \le 1/2$, therefore for $N$ large enough (so that $16 c_0 m N^{-2} \| \rho \|_{L^{\infty}(\mathbb R^3)} \le 1$) we get 
$$
(1 - 8 c_0 \,  N^{-2} \| \rho \|_{L^{\infty}(\mathbb R^3)} )^{-m} \le e^{16 c_0 m N^{-2} \| \rho \|_{L^{\infty}(\mathbb R^3)}}.
$$
\end{proof}

As a consequence we obtain the following bounds on $(\Pi^N[f])_{N \in \N^*}$:
\begin{lemma}
Given $N$ sufficiently large, for any $1\leq m  \leq N$ there holds:
\begin{align*} 
& \| \Pi^N_m[f] \|_{L^\infty_x L^1_v \left(\mathcal O^m[\frac{1}{N}] \right)}
\le e^{16c_0 m N^{-2} \| \rho \|_{L^\infty(\R^3)}} \, \| \rho \|_{L^\infty(\R^3)}^m,\\
&  \| |z_1|^{k_0} \Pi^N_1[f] \|_{L^1_x L^1_v (\mathbb R^3 \times \mathbb R^3)}\le e^{16c_0  N^{-2} \| \rho \|_{L^\infty(\R^3)}} \int_{\mathbb R^3 \times \mathbb R^3} 
|z_1|^{k_0} f(z_1){\rm d}z_1  ,\\
& \| |v_1|^{k_0} \Pi^N_2[f] \|_{L^{\infty}_x L^1_v \left(\mathcal O^2[\frac{1}{N}]\right)}  \le e^{32 c_0  N^{-2} \| \rho \|_{L^\infty(\R^3)}}  \|\rho\|_{L^{\infty}} \sup_{x_1 \in \mathbb R^3} \int_{\mathbb R^3} 
|v_1|^{k_0} f(x_1,v_1){\rm d}v_1  ,
\end{align*}
where $\Pi^N_m[f]$ denotes the $m$-marginal of $\Pi^N[f]$.

\end{lemma}

\begin{proof}
We write
$$
\begin{aligned}
f^N_{m}(z_1,\dots, z_m) 
&\le \mathcal W_N^{-1}(f) \mathbf 1_{(z_1,\dots, z_m) \in \mathcal O^m[\frac{1}{N}]} \, f^{\otimes m} (z_1,\dots, z_m)
\int_{(\mathbb R^3 \times \mathbb R^3)^{N-m}} \prod_{m+1 \le i < j \le N} \mathbf 1_{|x_i - x_{j}| > \frac2N} \, \prod_{j=m+1}^N f(z_j) \, \mathrm d z_{j}  \\
&\le \mathcal W_N^{-1}(f) \mathcal W_{N-m}(f) \mathbf 1_{(z_1,\dots, z_m) \in \mathcal O^m[\frac{1}{N}]} \, f^{\otimes m} (z_1,\dots, z_m).
\end{aligned}
$$
Each estimate then follows easily by using the bound of Lemma~\ref{WN}.
\end{proof}

This lemma shows that $(\Pi^N[f])_{N \in \N^*}$ satisfies Assumption~\ref{assump:A1}. We shall prove now that $(\Pi^N[f])_{N \in \N^*}$ is $f$-chaotic with quantitative estimates, which hence implies that it satisfies Assumption~\ref{assump:A2}.
To this end, we recall that we denote $(\mathbf{Z}^N)_{N \in \mathbb N}$ a sequence of random variables on $\mathcal O^N$ with corresponding laws $(\Pi^N[f])_{N \in \N^*}$ and that proving that $(\Pi^N[f])_{N \in \N^*}$ is $f$-chaotic reduces to measuring the expectation of the Wasserstein $W_1$-distance between the empirical measure $\mu^N[\mathbf Z^N]$
and $f.$ This is the content of the following lemma, from which Corollary~\ref{cor:main} follows straightforwardly.

\begin{lemma}\label{lem:W1}
Consider the framework of Corollary~\ref{cor:main}.
Let $(\mathbf Z^{N})_{N \in \N^*}$ be a sequence of random variables on $\mathcal O^N$ with laws $(\Pi^N[f])_{N \in \N^*}$ defined by \eqref{eq:cor-fN}. There holds
\[
\mathbb E[W_1(\rho^N[\mathbf Z^N],\rho)] \lesssim \dfrac{1}{N^{1/3}}
\quad
and
\quad
\mathbb E[W_1(\mu^N[\mathbf Z^N],f)] \lesssim \dfrac{1}{N^{1/6}}.
\]
\end{lemma}

%

\begin{proof}
We shall only prove the second estimate, the first one being similar arguing with the random variable $\mathbf X^N$ on $\mathcal O^N_x$ (coming from $\mathbf Z^N = (\mathbf X^N,  \mathbf V^N)$).

Let $(\mathbf W^N)_{N \in \N^*}$ be a i.i.d.\ sequence of random variables on $(\R^3 \times \R^3)^N$ with common law $f$, and $\mu^N[\mathbf W^N]$ be the associated empirical measure. We split
$$
W_1(\mu^N[\mathbf Z^N] , f) 
\le W_1(\mu^N[\mathbf W^N] , f) 
+\mathbf 1_{\mathbf W^N \in \mathcal O^N} \, W_1 (\mu^N[\mathbf Z^N] , \mu^N[\mathbf W^N])
+ \mathbf 1_{\mathbf W^N \not\in \mathcal O^N} \, W_1 (\mu^N[\mathbf Z^N] , \mu^N[\mathbf W^N]),
$$
which implies
$$
\begin{aligned}
\mathbb E \left[ W_1(\mu^N[\mathbf Z^N] , f)  \right]
&\le \mathbb E \left[ W_1(\mu^N[\mathbf W^N] , f) \right]
+ \mathbb E \left[  \mathbf 1_{\mathbf W^N \in \mathcal O^N} \, W_1 (\mu^N[\mathbf Z^N] , \mu^N[\mathbf W^N])  \right] \\
&\quad
+ \mathbb P \left[ \mathbf W^N \not\in \mathcal O^N \right]^{\frac12} \mathbb E \left[ W_1 (\mu^N[\mathbf Z^N] , \mu^N[\mathbf W^N])   \right]^{\frac12} .
\end{aligned}
$$
The first term on the right-hand side can be controlled by $N^{-1/6}$ thanks to \cite[Theorem 1]{FG}, since $\mathbf W^N$ is a i.i.d.\ sequence of common law $f$ and using the fact that $f$ has support included in $\Omega_0 \times \R^3$ as well as a finite moment of order $5$. The second term is bounded (up to a constant) by the first one, indeed
$$
\begin{aligned}
\mathbb E \left[  \mathbf 1_{\mathbf W^N \in \mathcal O^N} \, W_1 (\mu^N[\mathbf Z^N] , \mu^N[\mathbf W^N])  \right]
&= \int_{\mathcal O^N} \int_{\mathcal O^N} W_1 (\mu^N[\mathbf z^N] , \mu^N[\mathbf w^N]) \, \frac{\mathbf 1_{\mathbf z^N \in \mathcal O^N} f^{\otimes N} (\d \mathbf z^N) }{\mathcal W_N(f)} \,\mathbf 1_{\mathbf w^N \in \mathcal O^N} f^{\otimes N} (\d \mathbf w^N) \\
&\le \mathcal W_N(f)^{-1} \, \mathbb E \left[ W_1 (\mu^N[\widetilde{\mathbf W}^N] , \mu^N[\mathbf W^N])   \right] \\
&\lesssim \mathbb E \left[ W_1 (\mu^N[{\mathbf W}^N] , f)   \right] 
+ \mathbb E \left[ W_1 (\mu^N[\widetilde{\mathbf W}^N] , f)   \right],
\end{aligned}
$$
where $\widetilde{\mathbf W}^N$ is an independent copy of ${\mathbf W}^N$.
Finally the third term is bounded by $N^{-1/2}$ since $\mathbb P \left[ \mathbf W^N \not\in \mathcal O^N   \right]\lesssim N^{-1}$ (thanks to a similar argument as in Lemma~\ref{lem_sp2}) and  
$$
\mathbb E \left[ W_1 (\mu^N[\mathbf Z^N] , \mu^N[\mathbf W^N])   \right] 
\lesssim  \mathbb E \left[ M_2 (\mu^N[\mathbf Z^N]) \right] +  \mathbb E  \left[ M_2(\mu^N[\mathbf W^N])   \right] 
= M_2 ( \Pi^N_1[f] ) +   M_2(f)  ,
$$
which are uniformly bounded.
\end{proof}


\appendix

\section{Construction of $w_i$} \label{app_wi}
This section is devoted to the proof of {\bf Lemma~\ref{lem_wi}} and {\bf Lemma~\ref{lem_geometrie}}. 
We recall first the frame of these results. We assume that $N \in \mathbb N$ is given and strictly positive in the whole section and
we drop the parameter $N$ in most of notations.
We consider $N$ balls $B_i,$ $i=1,\ldots,N,$ of centers $(X_1,\ldots,X_N) \in \mathbb R^{3N}$
and common radii $1/N.$ We assume that $|X_i-X_j| > 2/N$ for $j\neq i$ so that these balls are disjoint. 

We begin with {\bf Lemma~\ref{lem_geometrie}} on the possible intersections of $(B(X_i,\frac{3}{2N}))_{i=1,\ldots,N}$.
We recall the statement of this lemma and give a proof:
\begin{lemma} \label{lem_geometrie_app}
Let $i\in \{1,\ldots,N\}.$ Setting  
$$
\mathcal I_i := \{j \in \{1,\ldots,N\} \text{ s.t. }B(X_i,\tfrac{3}{2N} ) \cap B(X_j,\tfrac{3}{2N}) \neq \emptyset \}, 
$$ 
we have that $\mathcal I_i$ contains at most $64$ distinct indices.
\end{lemma} 
\begin{proof}
The idea of this proof is adapted from \cite{OJ}. 

Let $i \in \{1,\ldots,N\}$ be fixed. Without restriction we may assume that $i=1$ and  $X_1 =0.$ 
For arbitrary $j \in \mathcal I_1$ we have that $B(X_j,\frac{3}{2N} ) \cap B(0, \frac{3}{2N} ) \neq \emptyset.$
This entails that $|X_j| \leq 3/N$ and  $B(X_j,\frac1N) \subset B(0,\frac4N).$
As the $B(X_j,\frac1N)$ are disjoint by assumption, we have then:
$$
\dfrac{4\pi}{3N^{3}}|\mathcal I_1|  = |\bigcup_{j \in \mathcal I_1} B(X_j,\tfrac1N)| \leq |B(0,\tfrac4N)| \leq \dfrac{4\pi}{3N^3} 64   . 
$$
This completes the proof.
\end{proof}

We proceed with {\bf Lemma~\ref{lem_wi}} that we recall with the notations of {\bf Section~\ref{sec:UN-prop}}:
\begin{lemma} \label{lem_wiapp}
Given $i\in \{1,\ldots,N\},$ there exists $w_i \in D(\mathbb R^3)$ satisfying 
\begin{align}
\label{eq_wi1_app}& w_i = V_i \text{ on $B_i$ and $w_i =0$ on $B_j$ for $j \neq i\,,$}\\[4pt]
\label{eq_wi2_app}& \mathrm{Supp}(w_i) \subset B(X_i, \tfrac{3}{2N}), \\[4pt]
&  \label{eq_bornewiapp}
\|\nabla  w_i \|_{L^2(\mathbb R^3)}^2  \leq C \dfrac{|V_i|^2}{N} \left( 1
 + \dfrac 1N   \sum_{j \neq i} \dfrac{\mathbf{1}_{|X_i- X_j| < \frac 5{2N}}}{{|X_i-X_j| - \dfrac{2}{N}}}    \right) \,,
\end{align}
for a universal constant $C.$
\end{lemma}

The remainder of this section is devoted to the proof of this result. Without loss of generality, we assume that 
$i=1$ and $X_1=0.$ We look for $w_1$ of the form:
\begin{equation} \label{eq_scaling}
w_1(x) = \tilde{w}_1(Nx), \quad \forall \, x \in \mathbb R^3.
\end{equation}
To define the constraints to be satisfied by $\tilde{w}_1,$ we introduce notations for the shape of the fluid
domain after dilation. Namely, we denote:
$$
\tilde{X}_i = NX_i, \qquad 
\tilde{B}_i= B(NX_i,1), \quad \forall \, i =1,\ldots,N.
$$
In particular, $\tilde{B}_1= B(0,1).$ We want now to construct $\tilde{w}_1 \in D(\mathbb R^3)$ such that:
\begin{align}
\label{eq_twi1}& \tilde{w}_1 = V_1 \text{ on $\tilde{B}_1$ and $\tilde{w}_1 =0$ on $\tilde{B}_j$ for $j >1\,,$}\\[4pt]
\label{eq_twi2}& \mathrm{Supp}(\tilde{w}_1) \subset B(0, \tfrac32),
\end{align}
A natural candidate for $\tilde{w}_1$ is obtained by focusing on \eqref{eq_twi2}.
Indeed, introducing a truncation function $\chi_0 \in C^{\infty}(\mathbb R)$ which satisfies:
$$
\chi_0(t) = \left\{
\begin{array}{rl}
1 &\text{ if $t < 1,$} \\
0 &\text{ if $t> 1+h_{0},$} 
\end{array}
\right.
$$ 
with $h_{0} \in (0,1/2)$ to be fixed later on, we may set:
$$
\tilde{w}_{1,0}= \nabla \times \left[ \dfrac{V_1 \times x}{2} \chi_0(|x|) \right].
$$
This candidate satisfies indeed $\tilde{w}_{1,0} \in \mathcal D(\mathbb R^3)$ with 
$$
 \tilde{w}_{1,0} = V_1 \text{ on $\tilde{B}_1$}, \qquad
\mathrm{Supp}(\tilde{w}_{1,0}) \subset B(0,1+h_{0}) \subset B(0,\tfrac32),
$$
However, it does not take into account the balls that are too close to $\tilde{B}_1.$ To match the further condition on these balls,
we modify our candidate.   

For this, let fix $j \in \{1,\ldots,N\}.$ To describe the geometry between $\tilde{B}_1$ and $\tilde{B}_j$
we introduce a system of coordinates $(x_1,x_2,x_3)$ such that $x_3$ corresponds to the coordinates directed along $e_3 = \tilde{X}_j/|\tilde{X}_j|.$ The associated cylindrical coordinates read:
$$
r= \sqrt{x_1^2+x_2^2},  \quad e_{r} = \dfrac{1}{\sqrt{x_1^2+x_2^2}} (x_1,x_2,0), \quad \forall \, (x_1,x_2,x_3) \in \mathbb R^3 \setminus \{x_3=0\}.
$$
We remark that, in these coordinates, close to $(0,0,1)$ the boundary $\partial \tilde{B}_1$ 
is the graph of the function $(x_1,x_2) \mapsto \gamma_b(\sqrt{x_1^2+x_2^2})$ where:
$$
\gamma_b(r) = \sqrt{1-r^2}\,, \quad \forall \, r \in (0,1)  .
$$
Furthermore, denoting by $h_j = {\rm dist}(\tilde{B}_1,\tilde{B}_j),$ we have also that close to $(0,0,1+h_j),$
the boundary $\partial \tilde{B}_j$ is the graph of the function $(x_1,x_2) \mapsto \gamma_t(\sqrt{x_1^2+x_2^2})$ where:
$$
\gamma_t(r) = 2+h_j-\sqrt{1-r^2}\,, \quad \forall\, r \in (0,1) .
$$
Given $\delta >0$ we set, in these cylindrical coordinate:
\begin{align*}
&\mathfrak{C}_{j}[\delta] := \{(x_1,x_2,x_3) \in \mathbb R^3 \; \text{ s.t. } \; r \in (0,\delta) \text{ and } x_3 \in (\gamma_b(r),\gamma_t(r))\},\\
&\mathfrak{A}_{j}[\delta] := \{(x_1,x_2,x_3) \in \mathbb R^3 \; \text{ s.t. } \; r \in (\delta/2,\delta) \text{ and } x_3 \in (\gamma_b(r),\gamma_b(\delta/2))\}.
 \end{align*}
 These notations are illustrated by Figure~\ref{fig_notations}.
 
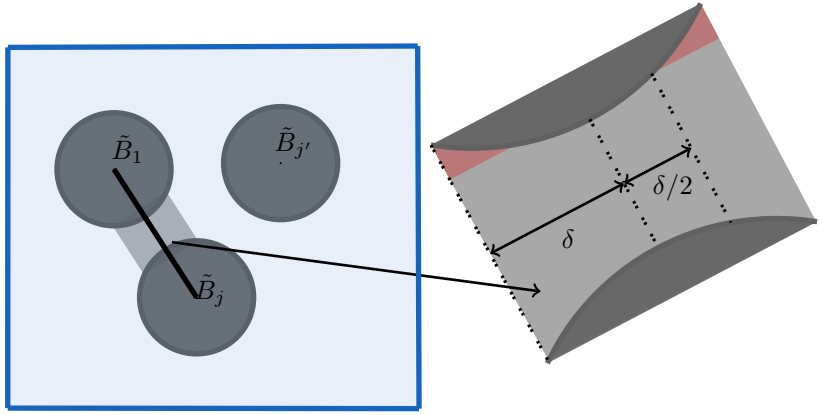
\begin{figure}[h] 
\begin{center}
\begin{tikzpicture}[x=1cm,y=1cm,scale=.3]\clip(-21.039281008231345,-10.218504191485415) rectangle (40.34141744287439,10.778196142565328);
\draw [line width=2pt,color=wrwrwr,fill=wrwrwr,fill opacity=0.9] (-4.595824107573774,0.8640117555213342) circle (2.5cm);
\draw [line width=2pt,color=wrwrwr,fill=wrwrwr,fill opacity=0.9] (-0.9745887201721779,-4.796741801584462) circle (2.5cm);
\fill[line width=0pt,color=wrwrwr,fill=wrwrwr,fill opacity=0.45] (-4.661961086269871,-1.6351132713728687) -- (-2.2991776536107604,-0.12361928160696241) -- (-0.9084517414760809,-2.2976167746902587) -- (-3.2712351741351915,-3.8091107644561655) -- cycle;
\fill[line width=2pt,color=wrwrwr,fill=wrwrwr,fill opacity=0.55] (9.298106723114913,1.9355612843359054) -- (14.458918076020574,-7.5927283369227485) -- (26.18677586866712,-1.447158774228276) -- (21.121386485239107,8.19274052900534) -- cycle;
\fill[line width=0pt,color=dtsfsf,fill=dtsfsf,fill opacity=0.4] (18.969627161303347,5.091948618936739) -- (21.121386485239107,8.19274052900534) -- (21.9326215186126,6.648886060766328) -- cycle;
\fill[line width=0pt,color=dtsfsf,fill=dtsfsf,fill opacity=0.4] (13.07219355775615,1.9708773186411401) -- (9.298106723114913,1.9355612843359058) -- (10.11842224205613,0.3965124391531907) -- cycle;
\draw [line width=2pt,color=wrwrwr,fill=wrwrwr,fill opacity=0.9] (2.708923226939263,1.1359703526133895) circle (2.5cm);
\fill[line width=2pt,color=rvwvcq,fill=rvwvcq,fill opacity=0.10000000149011612] (-9.248554012919069,6.25114916789872) -- (8.715696570118672,6.311634860097501) -- (8.776182262317452,-9.717073572579345) -- (-9.255481878114473,-9.693433714517724) -- cycle;
\draw [line width=2pt] (-4.595824107573774,0.8640117555213342)-- (-0.9745887201721779,-4.796741801584462);
\draw [->,line width=1pt] (-2.055628403349913,-2.3577021654999744) -- (14.103868922620816,-4.533859381648242);
\draw [shift={(11.083073766897053,12.861723188699337)},line width=2pt,color=wrwrwr,fill=wrwrwr,fill opacity=0.9]  plot[domain=4.55045314332096:5.8478320320272825,variable=\t]({1*11.071003626941463*cos(\t r)+0*11.071003626941463*sin(\t r)},{0*11.071003626941463*cos(\t r)+1*11.071003626941463*sin(\t r)});\draw [shift={(24.44149003551045,-12.379729316801221)},line width=2pt,color=wrwrwr,fill=wrwrwr,fill opacity=0.9]  plot[domain=4.55045314332096:5.8478320320272825,variable=\t]({0.5624030091464867*11.071003626941463*cos(\t r)+0.8268632627605225*11.071003626941463*sin(\t r)},{0.8268632627605225*11.071003626941463*cos(\t r)+-0.5624030091464867*11.071003626941463*sin(\t r)});
\draw [<->,line width=1pt] (17.762281901203753,0.24099693594905608) -- (20.71099870297735,1.8015325860968543);
\draw [<->,line width=1pt] (17.762281901203753,0.24099693594905608) -- (11.897144012753698,-2.862982653563806);
\draw [line width=1.2pt,dash pattern=on 1pt off 1pt on 1pt off 4pt] (18.969627161303347,5.091948618936739)-- (22.452370244651355,-1.4888834467430256);
\draw [line width=1.2pt,dash pattern=on 1pt off 1pt on 1pt off 4pt] (9.298106723114913,1.9355612843359058)-- (14.446012400884083,-7.619734761821659);
\draw [line width=1.2pt,dash pattern=on 1pt off 1pt on 1pt off 4pt] (16.26163551036644,3.0765497320442954)-- (19.262928292041064,-2.5945558601461847);
\draw [line width=2pt,color=rvwvcq] (-9.248554012919069,6.25114916789872)-- (8.715696570118672,6.311634860097501);
\draw [line width=2pt,color=rvwvcq] (8.715696570118672,6.311634860097501)-- (8.776182262317452,-9.717073572579345);
\draw [line width=2pt,color=rvwvcq] (8.776182262317452,-9.717073572579345)-- (-9.255481878114473,-9.693433714517724);
\draw [line width=2pt,color=rvwvcq] (-9.255481878114473,-9.693433714517724)-- (-9.248554012919069,6.25114916789872);
\draw [fill=black] (-4.595824107573774,0.8640117555213342) circle (0.5pt);
\draw[color=black] (-3.995080032353128,1.7566455989510392) node {$\tilde{B}_1$};
\draw [fill=black] (-0.9745887201721779,-4.796741801584462) circle (0.5pt);
\draw[color=black] (-0.38127792670806276,-4.5067756113882806) node {$\tilde{B}_j$};
\draw [fill=wrwrwr] (17.762281901203753,0.24099693594905608) circle (2.5pt);
\draw[color=black] (19.8991637407926,0.055460115766171) node {$\delta/2$};
\draw[color=black] (15.314489427660801,-2.209908368369557) node {$\delta$};
\draw [fill=black] (2.708923226939263,1.1359703526133895) circle (0.5pt);
\draw[color=black] (3.2864615237973767,2.0263323232529116) node {$\tilde{B}_{j'}$};
\end{tikzpicture}
\caption{Notations $\mathfrak{A}_j[\delta]$ and $\mathfrak{C}_j[\delta]$}
\label{fig_notations}
\vskip 2pt
\begin{minipage}{.8\textwidth}
On the left a typical configuration is presented (in 2D). The gray zone corresponds to the set $\mathfrak{C}_{j}[\delta]$. On the right is a zoom on $\mathfrak{C}_j[\delta]$
where the subset $\mathfrak{A}_j[\delta]$ appears in the red color. We emphasize that the 3D geometry is obtained by revolution around the axis of the figure so that $\mathfrak{A}_j[\delta]$ is indeed connected.
\end{minipage}
\end{center}
\end{figure} 
 
We note that, whatever the value of $\delta \in (0,1)$ we have that $\mathfrak{C}_{j}[\delta]$ and $\mathfrak{A}_{j}[\delta]$ are Lipschitz,
and that $\mathfrak{A}_{j}[\delta] \subset  \mathfrak{C}_{j}[\delta].$  We have also the following technical property:
\begin{proposition} \label{prop_geometrie}
There exists  $h_{max} \in (0,1/2)$ and $\delta_0 \in (0,1/2)$ such that, if $h_j < h_{max}$ the following holds true:
\begin{enumerate}
\item[$i)$] $\mathfrak C_j[\delta_0] \subset B(0,\tfrac32),$\\
\item[$ii)$] $\overline{\mathfrak{C}_{j}[\delta_0]} \subset \mathbb R^3 \setminus \overline{\bigcup_{i \neq 1,j}^N \tilde{B}_i},$\\
\item[$iii)$] for arbitrary $j' \neq j$ such that $h_{j'} < h_{max},$ there holds $\mathfrak C_j[\delta_0] \cap \mathfrak C_{j'}[\delta_0] = \emptyset.$
\end{enumerate}   
\end{proposition}
\begin{proof}

We compute restrictions on the values for $\delta_0$ and $h_{max}$
in order to fulfill the three conditions $i),$ $ii)$ and $iii).$ This will yield an open set of admissible values for $\delta_0$ and $h_{max}.$ 

For the proof, we only give two draws which explain where the restrictions come from. 
Let $j \in \{1,\ldots,N\}$  such that ${\rm dist}(\tilde{B}_1,\tilde{B}_j)=: h_j < h_{max}.$ In Figure
\ref{fig_voisinage}, we illustrate that there exists a ball $\mathcal V_j$  centered in $X_{j1}$ (the unique point in the closure of $\tilde{B}_j$ realizing the distance between $\tilde{B}_1$ and $\tilde{B}_j$) such that $\mathfrak{C}_j[\delta]$ (in blue on the figure) is contained in  $\mathcal{V}_j$ (empty circle on the figure). The radius $r_0$ of this neighborhood is controlled by $h_{max}$ and $\delta.$ In particular, for $h_{max}$ and $\delta_0$ sufficiently small we have $B(X_{j1},r_0) \subset B(0,1+h_{max} + r_0) \subset B(0,3/2)$ and $i)$ is realized.

\begin{figure}[h] 
\begin{center}
\begin{tikzpicture}[line cap=round,line join=round,x=1cm,y=1cm,scale=.65]
\clip(-8.881592901922332,-6.963745343136491) rectangle (26.672284700605292,8.757431759240362);
\draw [line width=2pt,color=wrwrwr,fill=wrwrwr,fill opacity=0.7] (-0.9415117661267175,3.5118087138602228) circle (4cm);
\draw [line width=2pt,color=wrwrwr,fill=wrwrwr,fill opacity=0.7] (7.36314204508927,-2.748700223320726) circle (4cm);
\fill[line width=2pt,color=rvwvcq,fill=rvwvcq,fill opacity=0.10000000149011612] (1.1226319993463245,0.0855409900156936) -- (2.92058322521458,2.470548766557696) -- (5.298998279616228,0.677567500523804) -- (3.501047053747974,-1.7074402760181997) -- cycle;
\draw [line width=1.2pt,dash pattern=on 1pt off 1pt on 1pt off 4pt,color=cqcqcq] (2.92058322521458,2.470548766557696)-- (5.298998279616228,0.677567500523804);
\draw [line width=2pt] (4.169062839084269,-0.3408260518228783) circle (3.0761225131116956cm);
\draw [<->,line width=1pt] (2.2525674398782853,1.103934542362374) -- (4.169062839084269,-0.3408260518228783);
\draw [<->,line width=1pt] (3.210815139481276,0.38155424526974846) -- (4.109790752415404,1.57405813354075);
\draw [<->,line width=1pt] (4.169062839084269,-0.3408260518228783) -- (1.821067398499832,-2.328147648008888);
\draw [fill=black] (-0.9415117661267175,3.5118087138602228) circle (0.5pt);
\draw[color=black] (-0.49961728836554065,4.1569686296904464) node {$\tilde{B}_1$};
\draw [fill=rvwvcq] (7.36314204508927,-2.748700223320726) circle (0.5pt);
\draw[color=black] (7.802908319375547,-2.0798568268422866) node {$\tilde{B}_j$};
\draw [fill=wrwrwr] (4.169062839084269,-0.3408260518228783) circle (1.5pt);
\draw[color=black] (4.982433112918048,-0.331956698896807) node {$X_{j1}$};
\draw [fill=wrwrwr] (2.2525674398782853,1.103934542362374) circle (1.5pt);
\draw[color=black] (3.234532984972556,0.08309335344452544) node {$h_{j}$};
\draw[color=black] (4.0701330314981894,0.6392433941544507) node {$\delta$};
\draw[color=black] (2.678382944262627,-2.223318006716202) node {$r_0$};
\end{tikzpicture}
\caption{Construction of a neighborhood of $X_{j1}$ containing $\mathfrak C_{j}.$ }
\label{fig_voisinage}
\vskip 2pt
\begin{minipage}{.8\textwidth}
\end{minipage}
\end{center}
\end{figure}
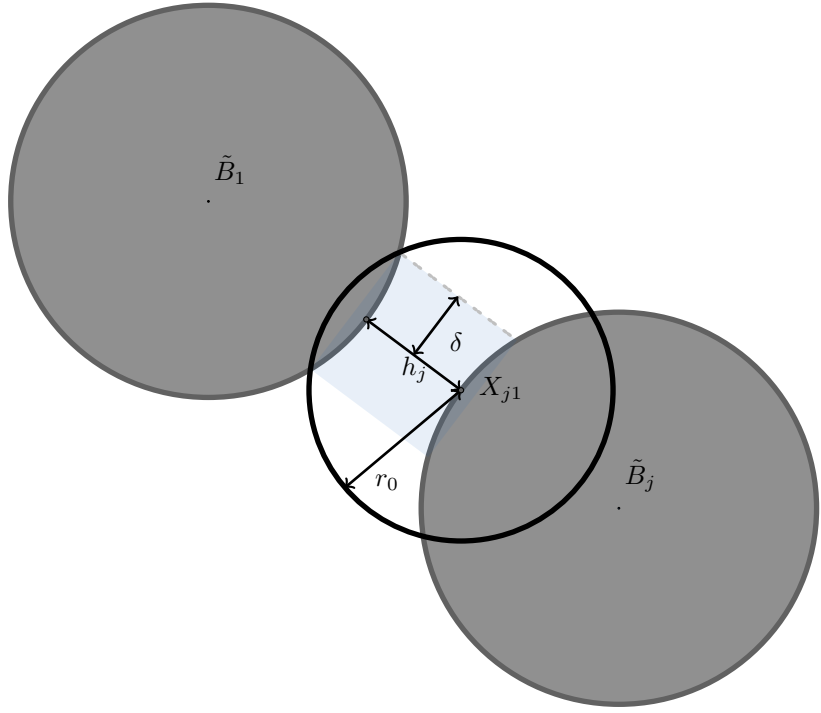 
 
Second, we illustrate with Figure~\ref{fig_optimal}, that given another particle $\tilde{B}_{j'},$ the distance between $\tilde{B}_{j'}$ and the segment $[\tilde{X}_1,\tilde{X}_j]$ joining the centers of $\tilde{B}_1$ and $\tilde{B}_j$ is minimal when $\tilde{B}_{j'}$ is in simultaneous contact with $\tilde{B}_1$ and $\tilde{B}_j$ (several configurations are provided in red, the optimal one is the most opaque one). The minimal distance $r^{(j)}_{min}$ between $\tilde{B}_{j'}$ and $[\tilde{X}_1,\tilde{X}_j]$ 
is then a decreasing function of $h_{j}$ vanishing when $h_{j} = 2(\sqrt{3}-1).$ 
The minimal distance $r^{(j,j')}$ between the point $X_{j'1}$ (the point in the closure of $\tilde{B}_{j'}$ realizing the distance with $\tilde{B}_1$) and $X_{j1}$ is also realized with this configuration. It is then a continuous function of $h_{j}$ which converges to $1$ when $h_j \to 0.$ So, with the notations of the proof, for $h_{max}$ and $\delta_0$ small we have that $r_0 < r^{(j)}_{min}$  and $2 r_0 < r^{(j,j')}$ so that $ii)$ and $iii)$ hold true.
\begin{figure}[h] 
\begin{center}
\begin{tikzpicture}[x=1cm,y=1cm,scale=.45]
\clip(-13.905821720640294,-10.241032905446307) rectangle (33.22709411904434,13.833937096891017);
\draw [line width=2pt,color=wrwrwr,fill=wrwrwr,fill opacity=0.8] (2,2.81) circle (4cm);
\draw [line width=2pt,color=wrwrwr,fill=wrwrwr,fill opacity=0.8] (2,-6) circle (4cm);
\draw [line width=2pt,color=dtsfsf,fill=dtsfsf,fill opacity=0.25] (-6.4101894543320315,5.5844521533833875) circle (4cm);
\draw [line width=2pt,color=ffqqqq,fill=ffqqqq,fill opacity=0.25] (10,-4) circle (4cm);
\draw [line width=2pt,color=dtsfsf,fill=dtsfsf,fill opacity=0.8] (8.647484661864283,-1.6373069352127023) circle (4cm);
\draw [line width=2pt,color=ffqqqq,fill=ffqqqq,fill opacity=0.25] (10.275919426418579,6.0781240137606245) circle (4cm);
\draw [line width=2pt] (2,2.81)-- (2,-6);
\draw [<->,line width=1.5pt,dash pattern=on 1pt off 1pt] (2,-1.595) -- (4.647708402717666,-1.595);
\draw [<->,line width=1.5pt,dash pattern=on 1pt off 1pt] (2,-2) -- (5.273690741171852,0.5115333521746095);
\draw [fill=black] (2,2.81) circle (0.5pt);
\draw[color=black] (1.313633785134967,3.433098126681305) node {$\tilde{B}_1$};
\draw [fill=black] (2,-6) circle (0.5pt);
\draw[color=black] (1.7349335915232096,-6.2881600896585375) node {$\tilde{B}_j$};
\draw [fill=ttqqqq] (-6.4101894543320315,5.5844521533833875) circle (0.5pt);
\draw [fill=black] (10,-4) circle (1pt);
\draw [fill=ttqqqq] (5.273690741171852,0.5115333521746095) circle (1.5pt);
\draw[color=ttqqqq] (5.658875078685793,0.7265990947400979) node {$X_{j'1}$};
\draw [fill=wrwrwr] (8.647484661864283,-1.6373069352127023) circle (2.5pt);
\draw [fill=black] (10.275919426418579,6.0781240137606245) circle (1.5pt);
\draw [fill=black] (2,-2) circle (1pt);
\draw[color=black] (1.3881831074938157,-2.501786783761425) node {$X_{j1}$};
\draw[color=black] (3.4201328170761798,-0.6224995499775923) node {$r^{(j,j')}$};
\draw[color=black] (3.368057381449725,-2.221912509805519) node {$r^{(j)}_{min}$};
\end{tikzpicture}
\caption{Minimizing configuration}
\label{fig_optimal}
\vskip 2pt
\begin{minipage}{.8\textwidth}
\end{minipage}
\end{center}
\end{figure}
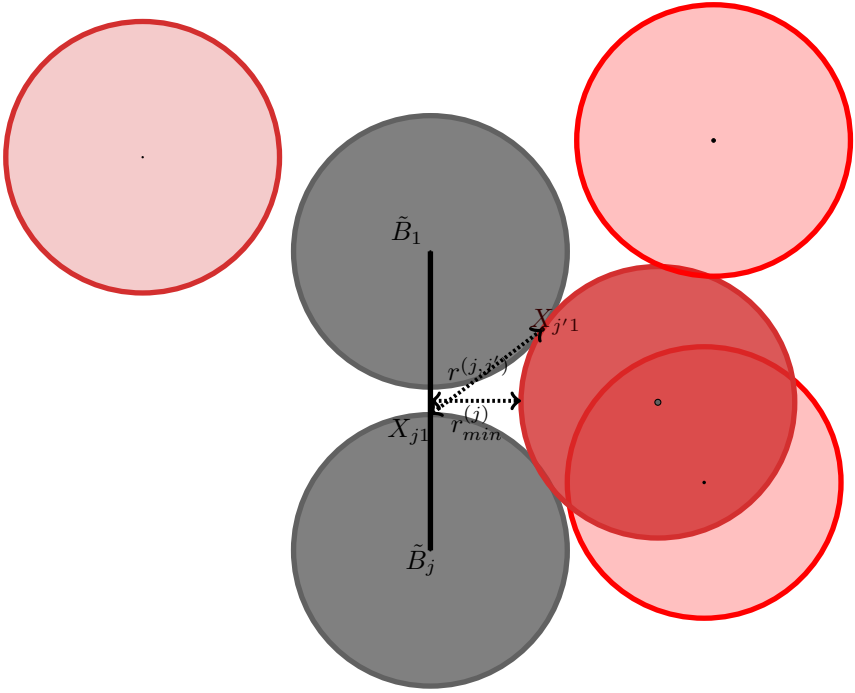 
\end{proof}

With the proposition above, we can now fix $h_{max},\delta_0$ sufficiently small so that the conclusion of the proposition above holds true. Associated with $\delta_0$ we set:
\[
h_0 = \sqrt{\dfrac{\delta_0^2}{4} + \left(2 - \sqrt{1-\left|\frac{\delta_0}{2}\right|^2}\right)^2} - 1. 
\]
If necessary, we restrict the size of $\delta_0$ so that $h_0 < \min(1/2,h_{max}).$ Associated with $h_{max}$ we introduce:
$$
\mathcal J := \left\{ j \in \{2,\ldots,N\} \text{ s.t.\ } {\rm dist}(\tilde{B}_1,\tilde{B}_j) < h_{max} \right\}\,,
$$
We note that, by construction, we do have $h_0 >0$ and that:
\begin{itemize}
\item since $h_0 < h_{max},$ $\tilde{w}_{1,0}$ vanishes on $\tilde{B}_i$ for $i \notin \mathcal J.$\\[-8pt]
\item for $j \in \mathcal J,$ $\chi_0$ vanishes on $\partial \mathfrak C_j \cap \tilde{B}_j$ at a distance larger than $\delta_0/2$ from the axis $\mathbb Re_3.$ 
\end{itemize}

Furthermore, the $(\mathfrak{C}_j)_{j \in \mathcal J}$ are disjoint and do not intersect the $(\tilde{B}_i)_{i=1,\ldots,N}.$ So,  in what follows, we construct $\tilde{w}_1$ on 
the $(\mathfrak{C}_j)_{j\in \mathcal J}.$ We shall then extend 
$\tilde{w}_{1}$ by $\tilde{w}_{1,0}$ on the remaining fluid domain
and by the expected values on the $(\tilde{B}_i)_{i=1,\ldots,N}.$

Let $j \in \mathcal J$ and make precise ${w}_j = (\tilde{w}_{1})_{|_{\mathfrak C_j}}.$ We decompose $w_{j} = w_{j}^{(1)} - w_{j}^{(2)}.$ 
For $w_{j}^{(1)},$ we set:
$$
w^{(1)}_j(x) = \nabla \times \left[  \dfrac{V_1}{2} \times \left(x-e_3 \right) \zeta_0(r) P\left( \dfrac{\gamma_t(r) - z }{\gamma_t(r) - \gamma_b(r)}\right)  + (1- \zeta_0(r))\chi_0(|x|) \dfrac{V_1}{2} \times x  \right] 
$$
where  $P(t) = (3t^2- 2t^3)$ for $t \in \mathbb R$ and $\zeta_0 \in C^{\infty}(\mathbb R)$ is a truncation function such that:
$$
\zeta_0(t) = 
\left\{
\begin{array}{rl}
1 & \text{ if $t < \delta_0/2,$} \\[4pt]
0 & \text{ if $t > 3\delta_0/4.$}
\end{array}
\right.
$$

Clearly, we have that $w_{j}^{(1)} \in C^{\infty}(\overline{\mathfrak C_j})$ is divergence-free. Expanding the curl operator, we obtain:
\begin{equation} \label{eq_boundarywj1}
w_{j}^{(1)}(x) =
\left\{
\begin{array}{cl}
0 & \text{ if $x \in \partial \mathfrak{C}_j \cap \partial \tilde{B}_j$ ({\em i.e.}$z = \gamma_t(r)$),} \\
V_1 - \dfrac{\zeta_0'(r)}{2} (V_1 \times e_{3}) \times e_r   & \text{ if $x \in \partial \mathfrak{C}_j \cap \partial \tilde{B}_1$ ({\em i.e.} $z= \gamma_b(r)$),}\\
w_{1,0}(x) & \text{ if $x \in \partial \mathfrak C_j \setminus \left(\partial \tilde{B}_1 \cup \partial \tilde{B}_j \right) $ ({\em i.e.} $r=\delta$)}.
\end{array}
\right.
\end{equation}
All these identities derive from the choices for $\chi_0,\zeta_0$ and $P.$ To obtain the first of these identities, it is worth noting that, with our choice for $h_0,\delta_0$ the function $x \mapsto (1-\zeta_0(r))\chi_0(r)$ vanishes on $\partial \tilde{B}_j \cap \partial \mathfrak{C}_j.$
 
Finally, we obtain that there exists a constant $C_{max}$ depending only on $(h_{max},\delta_0)$ such that:
\begin{equation} \label{eq_boundwj1}
\|\nabla w_{j}^{(1)} \|_{L^2(\mathfrak C_j)}^2 \leq \dfrac{C_{max}|V_1|^2}{h_j} .
\end{equation}
Indeed, away from the axis ({\em i.e.} on $\mathfrak C_j \cap \{r > \delta_0/2\}$), $w_j^{(1)}$ depends smoothly on the parameter $h_j.$ Hence, the contribution to $\|\nabla w_{j}^{(1)}\|_{L^2}$ is bounded by $C |V_1|^2$ where $C$ is independent of $h_j$ and depends only on $\delta_0,h_{max}.$ When $r<\delta_0/2,$ we have: 
$$
w_{j}^{(1)}(x) =  \nabla \times \left[  \dfrac{V_1}{2} \times \left(x- e_3 \right) P\left( \dfrac{z-\gamma_b(r)}{\gamma_t(r) - \gamma_b(r)}\right) \right]
$$
Explicit computations show that, the worst term in $|\nabla w_{j}^{(1)}|$ corresponds to two differentiations of the $P$-term w.r.t.\ $z,$ which we may bound by
\begin{align*}
|\partial_z w_j^{(1)}| \leq  \dfrac{|V_1|}{2} \left|x-e_3  \right|  \left| \partial_{zz}  P\left( \dfrac{z-\gamma_b(r)}{\gamma_t(r) - \gamma_b(r)}\right) \right|
\leq C |V_1| (r + |z - \gamma_b(0)|) \dfrac{1}{(\gamma_t(r) - \gamma_b(r))^2}.
\end{align*}
Remarking that $|z-\gamma_b(0)| \leq C|h_j + r|$ on $\mathfrak C_j$, we derive
$$
\int_{\mathfrak C_j \cap \{r<\delta_0/2\}} |\nabla w_{j}^{(1)}(x)|^2 {\rm d}x \leq  C |V_1|^2 \int_{0}^{\delta_0}   \dfrac{|h_j+r|^2 r {\rm d}r}{(\gamma_t(r) - \gamma_b(r))^3}.
$$
Combining then that $\gamma_t(r) - \gamma_b(r) \geq h_j + cr^{2}$ on $(0,\delta_0)$ for some $c>0$ (since $\delta_0 < 1/2$) 
and a change of variable $r= \sqrt{h}_j s$ in the integral, we obtain
\eqref{eq_boundwj1}. More details on these computations can be found in  \cite{HillairetTakfarinas}.

In order that $w_{j}$ fits the right boundary condition on $\partial \tilde{B}_1$, we add a corrector $w_{j}^{(2)}$ that compensate the error term that appears on the second line of 
\eqref{eq_boundarywj1}, namely:
$$
w_{j}^*(x) = \dfrac{\zeta_0'(r)}{2}\left[{V_1} \times e_3 \right]  \times e_r =\dfrac{\zeta_0'(r)}{2}({V_1} \cdot e_r )  e_3 , 
$$
To construct $w_j^{(2)},$ we note that $w_{j}^{*}$
is smooth and has compact support in $\partial \mathfrak A_j \cap \partial \tilde{B}_1.$ Hence, we may extend $w_j^*$ by 
$0$ on  $\partial \mathfrak A_j \setminus\partial \tilde{B}_1 .$ We obtain then a vector field $w_j^* \in C^{\infty}(\partial \mathfrak A_j)$ such that, by symmetry: 
\[
\int_{\partial \mathfrak A_j} w_{j}^* \cdot n {\rm d}\sigma =\int_{\partial \mathfrak A_j \cap \partial \tilde{B}_1} w_j^*(x) \cdot n {\rm d}\sigma = 0 
\]
Since, there exists a Bogovskii operator on the Lipschitz domain $\mathfrak A_j,$ we construct $w_{j}^{(2)} \in H^1(\mathfrak A_j)$ such that:
\begin{equation} \label{eq_boundarywj2}
{\rm div} \, w_{j}^{(2)} = 0 \text{ in $\mathfrak A_j$} \quad w_{j}^{(2)} = w_{j}^* \text{ on $\partial \mathfrak A_j$}.
\end{equation}
and such that: 
$$
\|w_{j}^{(2)} \|_{H^1(\mathfrak A_j)} \leq C \|w_j^* \|_{H^{1/2}(\partial \mathfrak A_j)}. 
$$
We note here that all the $\mathfrak A_j$ are isometric so that  this last constant $C$ is fixed by the values of $\delta_0$ only and does not depend on $j$. Hence, there exists $C_{max}$ depending only on $\delta_0$ for which:
\begin{equation} \label{eq_boundwj2}
\|w_{j}^{(2)} \|_{H^1(\mathfrak A_j)} \leq C_{max} |V_1|. 
\end{equation}
We note also that, on $\partial \mathfrak A_j,$ $w_j^*$ vanishes outside $\partial \mathfrak A_j \cap \partial \tilde{B}_1$ so that we may extend it by $0$ on $\mathfrak C_j \setminus \mathfrak A_j.$
We keep the same notations for simplicity. 
This yields a divergence-free vector-field $w_{j}^{(2)} \in H^1(\mathfrak C_j)$ defined on $\mathfrak C_j.$

By combination, it is then straightforward that $w_j = w_{j}^{(1)}- w_{j}^{(2)} \in H^1(\mathfrak C_j)$ satisfies:
\begin{itemize}
\item[$i)$] ${\rm div} \, w_{j} = 0$  on $\mathfrak C_j$ 
\item[$ii)$] the following boundary conditions on $\partial \mathfrak C_j$: 
$$
w_{j}(x) =
\left\{
\begin{array}{cl}
0 & \text{ if $x \in \partial \mathfrak{C}_j \cap \partial \tilde{B}_j$} \\
V_1 & \text{ if $x \in \partial \mathfrak{C}_j \cap \partial \tilde{B}_1$} \\
w_{1,0}(x) & \text{if $x \in \partial \mathfrak{C}_j \setminus ( \partial \tilde{B}_1 \cup \partial \tilde{B}_j)$}
\end{array}
\right.
$$
\item[$iii)$] the bounds (with a constant $C_{max}$ depending only on $\delta_0,h_{max}$):
\[
\|\nabla w_{j}^{(1)} \|_{L^2(\mathfrak C_j)}^2 \leq C_{max}|V_1|^2 \left[1 + \dfrac{1}{h_j}\right] .
\]
\end{itemize}

In particular, the above construction of $\tilde{w}_{1}$ on $\mathfrak{C}_j$
for fixed $j \in \mathcal J,$ satisfies the right boundary conditions in order to extend
it by $\tilde{w}_{1,0}$ on the remaining fluid domain. So, we set:
\begin{equation} \label{eq_tw1formula}
\tilde{w}_1 (x) =
\left\{
\begin{array}{rl}
V_1 & \text{ if $x \in \tilde{B}_1$} \\[4pt]
w_j(x) & \text{ if $x \in \mathfrak C_j,$ $j \in \mathcal J $}\\[4pt]
0 	 & \text{ if $x \in \tilde{B}_j,$ $j \neq 1$} \\[4pt]
w_{1,0}(x) & \text{ else}.
\end{array}
\right. 
\end{equation}
Combining \eqref{eq_boundarywj1}-\eqref{eq_boundarywj2} we obtain that 
$\tilde{w}_{1} \in H^1(\mathbb R^3)$ is divergence-free and satisfies the required
conditions on the obstacles $(\tilde{B}_i)_{i=1,\ldots,N}$. Furthermore, combining
\eqref{eq_boundarywj1}-\eqref{eq_boundwj2}, we obtain a constant $C_{max}$ depending only on $\delta_0,h_{max}$ such that:
\[
\|\nabla \tilde{w}_1\|^2_{L^2(\mathbb R^3)} \leq C_{max} |V_1|^2  \left( 1 + \sum_{j \in \mathcal J} \dfrac{1}{h_j} \right) \leq C_{max} |V_1|^2 \left( 1 + \sum_{j=2}^N \dfrac{\mathbf{1}_{|\tilde{X}_j| < 5/2}}{|\tilde{X}_j| -2} \right) 
\]
The associated vector-field $w_1$ ({\em via} the scaling \eqref{eq_scaling}) satisfies then all the requirements of Lemma~\ref{lem_wiapp}.

\section{Analysis of the cell problem} \label{app_stokes}
In this appendix, we fix $(N,M,\lambda) \in (\mathbb N \setminus \{0\})^2 \times (0,\infty),$ and a divergence-free $w \in C^{\infty}_c(\mathbb R^3)$.
We denote $T$ an open cube of width $\lambda$  and $B_i = B(X_i,\tfrac1N) \subset T$ for  $i=1,\ldots,M.$  We assume further that there exists $d_{m}$ satisfying 
\begin{equation} \label{eq_hypdm}
\min_{i=1,\ldots,M} \left\{ \text{dist}(X_{i}, \partial T), \min_{j \neq i} \left( |X_i - X_j |\right)\right\} \geq  d_{m} > \dfrac{4}{N}\,.
\end{equation}
We consider the Stokes problem: 
\begin{equation} \label{eq_stokesledeux}
\left\{
\begin{array}{rcl}
- \Delta u + \nabla p &=& 0\,, \, \\
{\rm div}\,  u &= & 0 \,,
\end{array}
\right.
\quad \text{ in $\mathcal F = T \setminus \bigcup_{i=1}^M \overline{B_i}$}\,,
\end{equation}
completed with boundary conditions 
\begin{equation} \label{cab_stokesledeux}
\left\{
\begin{array}{rcll}
u(x) &=& w(x)\,, &\text{ in $B_i\,,$ } \forall \, i =1,\ldots, M\,,\\[4pt]
u(x) &=& 0\,, & \text{ on $\partial T\,.$}
\end{array}
\right.
\end{equation}
Assumption \eqref{eq_hypdm} entails that the $B_i$ do not intersect and do not meet the boundary $\partial T.$ So, the set
$T \setminus \bigcup_{i=1}^M \overline{B_i}$ has a Lipschitz boundary that one can decompose in $M+1$ connected components
corresponding to $\partial T$ and $\partial B_i$ for $i=1,\ldots,M.$ 

For any $i=1,\ldots,M,$ direct computations show that:
$$
\int_{\partial B_i} w \cdot n {\rm d}\sigma = \int_{B_i} {\rm div} \, w =0.
$$
Hence, the problem \eqref{eq_stokesledeux}-\eqref{cab_stokesledeux} is solved by applying \cite[Theorem 3]{Hil17}
and it admits a unique generalized solution $u \in H^1(\mathcal F).$  We want to compare this solution 
with:
$$
u_s(x) = \sum_{i=1}^M G^N[w(X_i)](x-X_i) ,
$$
where, for arbitrary $v \in \mathbb R^3,$ $G^N[v]$ is the unique vector-field that can be associated to a pressure $P^N[v]$
in order to form a pair solution to the Stokes problem outside  $B(0,1/N)$. Explicit formulas for these solutions can be found in standard textbooks:
\begin{eqnarray}
G^N[v](x) &:=& \dfrac{1}{4N} \left( \dfrac{3}{|x|} +  \dfrac{1}{N^2|x|^3} \right) v + \dfrac{3}{4N}  \left( \dfrac{1}{|x|} - \dfrac{1}{N^2|x|^3}  \right)  \dfrac{v \cdot x}{|x|^2} x \,,  \label{eq_Stokeslet}\\
 P^N[v](x) &:=& \dfrac{3}{2N} \dfrac{v \cdot x}{|x|^3} \,. \label{eq_presslet}
\end{eqnarray}
The main result of this appendix section reads:
\begin{proposition} \label{prop_truncationprocess} 
There exists a constant $K$ independent of $(N,M,d_m,w,\lambda)$ for which:
\[
\|(u-u_s)\|_{L^6(\mathcal F)} +
\|\nabla (u-u_s)\|_{L^2(\mathcal F)} \leq K \Big[ \|w\|_{C^{0,1/2}(\overline{T})} + \|\nabla w\|_{L^6(T)} \Big]\sqrt{\dfrac{M}{N}} \left( \dfrac{1}{\sqrt{N}} +   \sqrt{\dfrac{M}{N{d_m}}}\right).   
\]
\end{proposition}
\begin{proof}
The proof is an adaptation to our notations and assumptions of \cite[Proposition 7]{Hil17}. 
We split the error term into two pieces. First, we reduce the boundary conditions of the Stokes problem \eqref{eq_stokesledeux}-\eqref{cab_stokesledeux} to constant
boundary conditions. Then, we compare the solution to the Stokes problem with constant boundary conditions to the combination of Stokeslets 
$u_s.$ In the whole proof, the symbol $\lesssim$ is used when the implicit constant in the written inequality does not depend on $N,M,d_m,w$ and $\lambda.$

\medskip

So, we introduce  ${u}_c$ the unique generalized solution to the Stokes problem on $\mathcal F$ with boundary conditions:
\begin{equation} \label{cab_Stokesletrois}
\left\{
\begin{array}{rcll}
u_c &=& w(X_i)\,, & \text{ in $B_i\,,$ } \forall \, i =1,\ldots, M\,,\\[4pt]
u_c &=& 0\,, & \text{ on $\partial T\,.$}
\end{array}
\right.
\end{equation}
Again, existence and uniqueness of this velocity-field holds by applying \cite[Theorem 3]{Hil17}.
We split then:
\begin{eqnarray*}
\| (u-u_s)\|_{L^6(\mathcal F)} &\leq& \| (u - u_c)\|_{L^6(\mathcal F)} + \| (u_c-u_s)\|_{L^6(\mathcal F)}\,, \\
\|\nabla (u-u_s)\|_{L^2(\mathcal F)} &\leq& \|\nabla (u - u_c)\|_{L^2(\mathcal F)} + \|\nabla (u_c-u_s)\|_{L^2(\mathcal F)}.
\end{eqnarray*}
To control the first term on the right-hand sides, we note that $(u-u_c)$ is the unique generalized solution to the Stokes problem on $\mathcal F$
with boundary conditions:
$$
\left\{
\begin{array}{rcll}
(u - u_c) (x) &=& w(x) - w(X_i)\,, & \text{ in $B_i\,,$ } \forall \, i =1,\ldots, M\,,\\[4pt]
(u - u_c)(x) &=& 0\,, & \text{ on $\partial T\,.$}
\end{array}
\right.
$$
Hence, by the variational characterization of \cite[Lemma 4]{Hil17}, 
$\|\nabla (u - u_c)\|_{L^2(\mathcal F)}$ realizes the minimum of $\|\nabla \tilde{w}\|_{L^2(\mathcal F)}$ amongst  
$$
\left\{\tilde{w} \in H^1(\mathcal F) \text{ s.t. }  {\rm div}\, \tilde{w} = 0\,, \;  \tilde{w}_{|_{\partial T}} = 0\,, \; \tilde{w}_{|_{\partial B_i}} = w(\cdot) - w(X_i)\,,  \; \forall \, i =1,\ldots,M \right\}\,.
$$
We construct thus a suitable $\tilde{w}$ in this space. We set:
$$
\tilde{w} = \sum_{i=1}^M  \tilde{w}_i
$$
with, for $i=1,\ldots,M:$ 
$$
\tilde{w}_i =  \left( \chi^{N}(\cdot-X_i) (w(\cdot) - w(X_i)) - \mathfrak B_{X_i,\frac1N,\frac2N} \left[ x \mapsto (w(x) - w(X_i)) \cdot \nabla \chi^N(x -X_i)\right] \right)\,. 
$$
In this definition $\chi^{N}$ is again chosen truncation function that between $B(0,\frac1N)$ and $B(0,\frac2N).$  We assume further that $\chi^N$ is obtained from $\chi^1$ by dilation. The operator $\mathfrak B_{X_i,\tfrac1N,\tfrac2N}$ denotes the Bogovskii operator on the annulus 
$$
A(X_i,\frac1N,\frac2N) = B(0,\tfrac2N) \setminus \overline{B(0,\tfrac1N)}.
$$ 
The properties of this operator are analyzed in \cite[Appendix A]{Hil17} (though these results are nowadays classical and can also be found in \cite{Allaire} for instance).
It is straightforward to verify  that the mean of
$x \mapsto (w(x) - w(X_i)) \cdot \nabla \chi^N(x -X_i)$ vanishes so that the above vector-field $\tilde{w}_i$ is well-defined. We note that $\tilde{w}_i$ has support in $B(X_i,\frac2N)$ so that, as $d_{m} > 4/N,$  the $\tilde{w}_i$
have disjoint supports inside $T.$ This yields that $\tilde{w}$ is indeed divergence-free and fits the required boundary conditions. 
Furthermore, there holds:
$$
\|\nabla \tilde{w}\|_{L^2(\mathcal F)} \leq \left[ \sum_{i=1}^M \|\nabla \tilde{w}_i\|^2_{L^2(B(X_i,\frac2N))}\right]^{\frac 12}\,.
$$
For $i\in \{1,\ldots,M\}$ we have by direct computations:
\begin{eqnarray*}
\|\nabla \chi^N (\cdot- X_i) (w(\cdot) - w(X_i))\|^2_{L^2(B_{\infty}(X_i^N,\frac2N))}  &\lesssim & \dfrac{\|w\|^2_{C^{0,1/2}}}{N^2}\,, \\
\| \chi^N (\cdot- X_i) \nabla (w(\cdot) - w(X_i))\|^2_{L^2(B_{\infty}(X_i,\frac2N))} & \lesssim & \dfrac{\|w\|^2_{W^{1,6}(T)}}{N^2} \,, \\
\end{eqnarray*}
and, by applying  \cite[Lemma 20]{Hil17}:
\begin{multline*} 
\|\nabla \mathfrak B_{X_i^N,\frac1N,\frac2N} \left[ x \mapsto (w(x) - w(X_i)) \cdot \nabla \chi^N(x -X_i)\right]\|^2_{L^2(B_{\infty}(X_i,\frac2N))}  \\[6pt]
\begin{array}{rcl}&\lesssim& \|x \mapsto (w(x) - w(X_i)) \cdot \nabla \chi^N(x -X_i)\|^2_{L^2(B(X_i,\frac2N))}  \\
& \lesssim & \dfrac{\|w\|^2_{C^{0,1/2}(\overline{T})}}{N^2}.
\end{array}
\end{multline*}
Gathering all these inequalities in the computation of $\tilde w$  yields finally:
$$
\|\nabla \tilde{w}\|_{L^2(\mathcal F)} \lesssim \sqrt{M}  \dfrac{\|w\|_{C^{0,1/2(\overline{T})}} +\|w\|_{W^{1,6}(T)} }{N}\,.
$$
The variational characterization of generalized solutions to Stokes problems entails that we have the same bound for $(u-u_c).$ 
At this point, we argue that the straightforward extension of $u$ and $u_c$ (by $w$ and $w(X_i)$ on the $B_i$ respectively) satisfy $(u-u_c) \in H^1_0(T) \subset L^6(T)$ so that 
\begin{eqnarray*}
\|u - u_c\|_{L^6(\mathcal F)} &\leq & \|u - u_c\|_{L^6(T)} \lesssim   \|\nabla (u-u_c)\|_{L^2(T)} \\
			              &\lesssim&   \left(  \|\nabla (u-u_c)\|^2_{L^2(\mathcal F)} + {M} \dfrac{\|w\|^2_{W^{1,6}(T)}}{N^{2}} \right)^{\frac 12} \\
			              & \lesssim & \sqrt{M}  \dfrac{\|w\|_{C^{0,1/2}(\overline{T})} +\|w\|_{W^{1,6}(T)}}{N}.
\end{eqnarray*}
We emphasize that, by a scaling argument, the constant deriving from the embedding $H^1_0(T) \subset L^6(T)$ does not
depend on $\lambda$ so that it is not significant to our problem.

\medskip

We turn to estimating $u_c - u_s.$ Due to the linearity of the Stokes equations, we split 
$$
u_c  = \sum_{i=1}^M u_{c,i},
$$ 
where $u_{c,i}$ is the generalized solution to the Stokes problem on $\mathcal F$ with boundary conditions:
$$
\left\{
\begin{array}{rcll}
u_{c,i}  &=& w(X_i)\,, & \text{ on $\partial B_i$\,,}\\[4pt]
u_{c,i}  &=& 0\,, & \text{ on $\partial T \cup \bigcup_{j\neq i} \partial B_{j}\,.$}
\end{array}
\right.
$$
We have then 
\begin{equation} \label{eq_bounderrorlepremier}
\|\nabla (u_c - u_s)\|_{L^2(\mathcal F)} \leq \sum_{i=1}^M \|\nabla (u_{c,i} - G^{N}[w(X_i)](\cdot-X_i))\|_{L^2(\mathcal F)}.
\end{equation}
Similarly, we expand :
$$
u_s = \sum_{i=1}^M G_i \,,  \text{ where } G_i(x) = G^N[w(X_i)](x-X_i)\,,   \quad \forall \, x \in \mathbb R^3.
$$
For $i \in \{1,\ldots,M\}$ we extend $u_{c,i}$ by $0$ on $\mathbb R^3 \setminus T$ and  ${B_j}$ for $j \neq i.$ 
The extension we still denote by $u_{c,i}$ satisfies  $u_{c,i} \in H^1(\mathbb R^3 \setminus \overline{B_i}),$ it is divergence-free and constant on $\partial B_i.$ 
In particular, we have
$u_{c,i} \in D(\mathbb R^3 \setminus \overline{B_i})$.
Consequently,  $u_{c,i} - G_i \in D(\mathbb R^3 \setminus \overline{B_i})$ and:
\begin{eqnarray*}
\|\nabla (u_{c,i} - G_i)\|^2_{L^2(\mathcal F)}  & \leq&   \displaystyle \int_{\mathbb R^3 \setminus \overline{B_i}} |\nabla u_{c,i} - \nabla G_i|^2  \\
	&\leq & \displaystyle \int_{\mathbb R^3 \setminus  \overline{B_i}} |\nabla u_{c,i}|^2 - 2 \int_{\mathbb R^3 \setminus  \overline{B_i}} \nabla u_{c,i} : \nabla G_i + \int_{\mathbb R^3 \setminus  \overline{B_i}} |\nabla G_i|^2 \,.					
\end{eqnarray*}
To compute the product term, we apply that $u_{c,i}$ and $G_i = G^N[w(X_i)](\cdot -X_i)$ have the same trace on $\partial B_i$
and that $U_i$ is a generalized solution to the Stokes problem on $\mathbb R^3 \setminus  \overline{B_i}.$ So, integrals of the form $\int_{\mathbb R^3 \setminus  \overline{B_i}} \nabla G_i : \nabla w$
(for $w \in D(\mathbb R^3 \setminus \overline{B_i})$)  depend  only on the trace of $w$ on $\partial B_i.$ This entails that:
$$
\int_{\mathbb R^3 \setminus  \overline{B_i}}  \nabla u_{c,i} : \nabla G_i = \int_{\mathbb R^3 \setminus  \overline{B_i}} |\nabla G_i|^2\,,
$$
and we have:
\begin{equation} \label{eq_bounderror2}
\|\nabla (u_{c,i} - G_i)\|^2_{L^2(\mathcal F)} \leq  \int_{\mathbb R^3 \setminus  \overline{B_i}} |\nabla u_{c,i}|^2 - \int_{\mathbb R^3 \setminus  \overline{B_i}} |\nabla G_i|^2\,. 
\end{equation}
To conclude, we find a bound from above for 
$$
\int_{\mathbb R^3 \setminus  \overline{B_i}} |\nabla u_{c,i}(x)|^2 {\rm d}x = \int_{\mathcal F} |\nabla u_{c,i}(x)|^2{\rm d}x.
$$ 
As $u_{c,i}$ is a generalized solution to a Stokes problem on $\mathcal F,$ this can be done by constructing  a divergence-free $\bar{w}_i$ satisfying the same boundary condition as $u_{c,i}$. We define:
$$
\bar{w}_i =  \chi_{d_{m}/4}(\cdot-X_i) G_i - \mathfrak B_{X_i,\frac{d_{m}}{4},\frac{d_{m}}{2}} \left[ x \mapsto G_i(x) \cdot \nabla \chi_{d_{m}/4}(x -X_i)\right]
$$
where $\chi_{d_m/4}$ truncates between $B(0,d_{m}/4)$ and $B(0,d_m/2)$. As previously, we have here a divergence-free function which satisfies the right boundary conditions because  $\chi_{d_{m}/4}(\cdot-X_i)=1$ on $B_i$
(since $d_m/4 > 1/N$) and vanishes on all the other boundaries of $\partial \mathcal F$ (since the distance between one hole center and the other holes or $\partial T$ is larger than $d_m - 1/N > d_m/2$). 
Again, similarly as in the computation of $\tilde{w}_i$  we apply the properties of the Bogovskii operator $ \mathfrak B_{X_i,\frac{d_{m}}{4},\frac{d_{m}}{2}}$ and there exists an
absolute constant $K$ for which:
\begin{multline*}
\|\nabla \bar{w}_i\|^2_{L^2(\mathcal F)} \leq  \int_{\mathbb R^3 \setminus  \overline{B_i}} |\chi_{d_{m}/4}(\cdot-X_i)  \nabla G_i |^2 \\
			+ K \left(\int_{A(X_i,\frac{d_{m}}{4},\frac{d_m}{2})} |\nabla G_i(x) |^2 + |\nabla \chi_{d_{m}/4}(x-X_i) \otimes G_i(x)|^2  {\rm d}x  \right).
\end{multline*}
As we have the same bound for $u_{c,i},$ we plug the right-hand side above into \eqref{eq_bounderror2} and get:
$$
\|\nabla (u_{c,i} - G_i)\|^2_{L^2(\mathcal F)} \lesssim \int_{\mathbb R^3 \setminus B(X_i,\frac{d_{m}}{4})}  |\nabla G_i(x)|^2 {\rm d}x 
+ \int_{A(X_i,\frac{d_{m}}{4},\frac{d_m}{2})} |\nabla \chi_{d_{m}/4}(x-X_i)\otimes G_i(x)|^2  {\rm d}x  \,.
$$
With the explicit decay properties for $G_i$ (see \eqref{eq_Stokeslet}) and $\nabla \chi_{d_m/4}$ we derive:
$$
\int_{\mathbb R^3 \setminus B(X_i,\frac{d_{m}}{4})} |\nabla G_i(x)|^2 {\rm d}x  + \int_{A(X_i,\frac{d_{m}}{4}, \frac{d_m}{2})}  |\nabla \chi_{d_{m}/4}(x-X_i) \otimes G_i(x)|^2  {\rm d}x \lesssim \dfrac{\|w\|^2_{L^{\infty}(T)}}{N^{2}d_{m}}\,. \\
$$
Combining these bounds for $i=1,\ldots,M$ in \eqref{eq_bounderrorlepremier} we get: 
$$
\|\nabla (u_{c}-  u_s) \|_{L^2(\mathcal F)} \leq  \dfrac{ M \|w\|_{L^{\infty}(T)}}{N\sqrt{d_{m}}}. 
$$
By similar arguments, we also have:
$$
\|u_{c}-  u_s\|_{L^6(\mathcal F)} =  \|u_{c}-  u_s\|_{L^6(T)}  \leq \sum_{i=1}^M \|u_{c,i}-  G_i\|_{L^6(\mathbb R^3 \setminus  \overline{B_i})}. 
$$
As $u_{c,i},G_i \in D(\mathbb R^3 \setminus \overline{B_i})$ and $u_{c,i},G_i$ share the same value
on $\partial B_i,$ there holds $u_{c,i}- G_i \in D_0(\mathbb R^3 \setminus  \overline{B_i})$ and we may use the classical inequality (see \cite[(II.6.9)]{Galdi}):
$$
\|u_{c,i}-  G_i\|_{L^6(\mathbb R^3\setminus  \overline{B_i})} \lesssim \|\nabla u_{c,i}- \nabla G_i\|_{L^2(\mathbb R^3\setminus  \overline{B_i})} \,, \quad \forall \, i =1,\ldots,M\,,
$$
(again the constant arising from this embedding does not depend on $N$ by a standard scaling argument).
This yields again the bound:
$$
\| (u_{c}-  u_s) \|_{L^6(\mathcal F)} \leq  \dfrac{ M \|w\|_{L^{\infty}(T)}}{N\sqrt{d_{m}}}.
$$

Finally, combining the error terms between $u_c$ and $u_s$ and between $u$ and $u_c$ we obtain
$$
\|(u-u_s)\|_{L^6(\mathcal F)} +
\|\nabla (u-u_s)\|_{L^2(\mathcal F)} \leq K \sqrt{\dfrac{M}{N}}\left( \dfrac{1}{\sqrt{N}}  + \sqrt{\dfrac{M}{Nd_m}}\right) \Big[ \|w\|_{C^{0,1/2}(\overline{T})} + \|\nabla w\|_{L^6(T)} \Big].   
$$
This ends the proof.
\end{proof}

We note that, when we apply Proposition~\ref{prop_truncationprocess} in this article, we will choose $M\geq 1$ and $d_m$ that has to be small. 
In that case we have that
$$
\dfrac{1}{\sqrt{N}}  \leq 2  \sqrt{\dfrac{M}{Nd_m}} ,
$$
and the result of Proposition~\ref{prop_truncationprocess} reads:
\[
\|(u-u_s)\|_{L^6(\mathcal F)} +
\|\nabla (u-u_s)\|_{L^2(\mathcal F)} \leq K \Big[ \|w\|_{C^{0,1/2}(\overline{T})} + \|\nabla w\|_{L^6(T)} \Big]\dfrac{M}{N\sqrt{d_m}}.   
\]

\section{Analysis of some constants} \label{sec_constants}
In this section, we consider the problem of finding constants for the Poincar\'e-Wirtinger inequality and
the Bogovskii operator on a cubic annulus $A(0,1-1/\delta,1) :=]-1,1[^3 \setminus [-(1-1/\delta),1-1/\delta]^3.$ In both proofs, we proceed by change of variables
(since only the asymptotics of the constant when $\delta \to \infty$ is needed). For this, we fix $\delta >2.$ 
We introduce  a odd strictly increasing application $\chi_{\delta} \in C^{2}([-1,1])$ such that 
\[
\chi_{\delta}([0,1/2]) = [0,1-1/\delta]  , \quad \chi_{\delta}(1) = 1.
\]
For this, we introduce an even $\zeta \in C^{\infty}(\mathbb R)$ such that:
\[
\mathbf 1_{[-1/4,1/4]} \leq \zeta \leq \mathbf 1_{[-1/2,1/2]}.
\]
We  fix a constant $k$ to be chosen later on and we define $\chi_{\delta}'$ 
as the interpolation between $2(1-1/\delta)$ on [0,1/2] and $k$ on $[1/2+1/\delta,1]$ that we integrate
from $t=0.$ This means:
$$
\chi_{\delta}(x) = \int_0^{x} 2(1-{1}/{\delta})\zeta(\delta(s-1/2)_{+}) +  k(1-\zeta(\delta(s-1/2)_{+}){\rm d}s.  
$$ 
With this choice, we fix $k$ so that $\chi_{\delta}(1)= 1$ yielding:
$$
k =  \dfrac{1- 2(1-1/\delta)\int_0^{1} \zeta(\delta(s-1/2)_+){\rm d}s}{\int_0^{1} (1-\zeta(\delta(s-1/2)_+){\rm d}s} =  \dfrac{1- 2(1-1/\delta)(1/2+ \int_0^{1/\delta} \zeta(\delta s){\rm d}s )}{\int_0^{1/2} (1-\zeta(\delta s)){\rm d}s} = O\left(\frac{1}{\delta}\right).
$$
We emphasize that, due to our choice for $\zeta,$ we have $\int_0^{1} \zeta(s){\rm d}s < 1/2.$ This entails that we have also $k >0$ and $\chi_{\delta}$ is indeed strictly increasing.

Consequently, we have that:
\begin{itemize}
\item $\chi_{\delta}$ realizes a $C^2$-diffeomorphism from $[-1,1]$ to $[-1,1]$ such that
$\chi_{\delta}([-1/2,1/2]) = [-(1-1/\delta),1-1/\delta],$
\item  ${1}/{\delta} \lesssim \chi_{\delta}'(y) \leq 2$ and $|\chi''_{\delta}(y)| \lesssim \delta$ for any $y \in [-1,1].$
\end{itemize}
We introduce $\sigma_{\delta}$ its converse mapping. It satisfies:
\begin{itemize}
\item  $\sigma_{\delta}([-(1-1/\delta),1-1/\delta]) = [-1/2,1/2],$
\item  $1/2 \leq \sigma_{\delta}'(x) \leq \delta$ and $|\sigma''_{\delta}(x)| \lesssim \delta^4$ for any $x \in [-1,1].$
\end{itemize}
Finally, we denote $X_{\delta}$ and $Y_{\delta}$ the corresponding $C^{2}$-diffeomorphisms between $A(0,1/2,1)$
and $A(0,1-1/\delta,1):$
\[
\begin{array}{rrcl}
X_{\delta} : & A(0,1/2,1) & \longrightarrow & A(0,1-1/\delta,1) \\
		& (y_1,y_2,y_3) &  \longmapsto & (\chi_\delta(y_1),\chi_{\delta}(y_2),\chi_{\delta}(y_3)))  
\end{array}
\begin{array}{rrcl}
Y_{\delta} : & A(0,1-1/\delta,1) & \longrightarrow & A(0,1/2,1) \\
		& (x_1,x_2,x_3) &  \longmapsto & (\sigma_\delta(x_1),\sigma_{\delta}(x_2),\sigma_{\delta}(x_3)))  
\end{array}
\]

We start with the Poincar\'e-Wirtinger inequality. Our main result reads:
\begin{proposition}
There holds $C_{PW}[\delta] \lesssim \delta.$ Namely, given $f \in L^2_0(A(0,1-1/\delta,1)) \cap H^1(A(0,1-1/\delta,1)),$
we have:
\begin{equation} \label{eq_PW}
\int_{A(0,1-1/\delta,1)} |f(x)|^2{\rm d}x \lesssim \delta^2 \int_{A(0,1-1/\delta,1)} |\nabla f(x)|^2 {\rm d}x.
\end{equation}
\end{proposition}
\begin{proof}
We fix $f \in L^2_0(A(0,1-1/\delta,1)) \cap H^1(A(0,1-1/\delta,1))$ and, with
the previous notations, let us consider:
\[
\tilde{f}(y) = f(X_{\delta}(y)) - \fint \tilde{f}\,,  \qquad \forall \, y \in A(0,1/2,1),
\]
with
\[
 \fint \tilde{f} :=  \int_{A(0,1/2,1)} f(X_{\delta}(y)){\rm d}y.
\]
Standard computations show that $\tilde{f} \in L^2_0(A(0,1/2,1)) \cap H^1(A(0,1/2,1))$ so that, by the Poincar\'e-Wirtinger inequality
we have:
\[
\int_{A(0,1/2,1)} |\tilde{f}(y)|^2{\rm d}y \lesssim \int_{A(0,1/2,1)} |\nabla \tilde{f}(y)|{\rm d}y.
\]
Conversely, there holds:
\[
f(x) = \tilde{f}(Y_{\delta}(x)) + \fint \tilde{f} \,,  \qquad \forall \, x \in A(0,1-1/\delta,1).
\]
Hence, because $f$ is mean-free on $A(0,1-1/\delta,1),$ there holds: 
\begin{align*}
\int_{A(0,1-1/\delta,1)} |f(x)|^2 {\rm d}x& \leq  \int_{A(0,1-1/\delta,1)} |f(x)|^2 +  |A(0,1/2,1)| \left[\fint \tilde{f} \right]^2 \\
							& \leq    \int_{A(0,1-1/\delta,1)} \left|f(x) - \fint \tilde{f} \right|^2{\rm d}x \\
							& \leq  \int_{A(0,1-1/\delta,1)} |\tilde{f}(Y_{\delta}(x))|^2{\rm d}x
\end{align*}
We can then transform the geometry to go back in the $A(0,1/2,1)$  and apply the previous inequalities on $\sigma'_{\delta}:$
\begin{align*}
\int_{A(0,1-1/\delta,1)} |f(x)|^2 {\rm d}x& \leq  \left(\prod_{i=1}^{3}\max_{x_i \in[0,1]} \frac{1}{\sigma'_{\delta}(x_i)} \right) \int_{A(0,1-1/\delta,1)} |\tilde{f}(Y_{\delta}(x))|^2 \prod_{i=1}^{3} \sigma_{\delta}'(x_i){\rm d}x_i \\
& \lesssim    \int_{A(0,1/2,1)} |\tilde{f}(y)|^2 {\rm d}y   \\
& \lesssim  \int_{A(0,1/2,1)} |\nabla \tilde{f}(y)|^2 {\rm d}y.
\end{align*}
At this point, we compute $\nabla \tilde{f}$ with respect to  $\nabla f$ and apply the previous inequalities on $\chi_{\delta}'$:
\begin{align*}
\int_{A(0,1/2,1)} |\nabla \tilde{f}(y)|^2 {\rm d}y & \lesssim \int_{A(0,1/2,1)} \sum_{i=1}^3 \chi_{\delta}'(y_i)^2| \partial_i f(X_{\delta}(y))|^2 {\rm d}y \\
						& \lesssim \int_{A(0,1/2,1)}  \sum_{i=1}^3 \dfrac{\chi_{\delta}'(y_i)}{\prod_{j\neq i}\chi_{\delta}'(y_j)}  |\partial_i f(X_{\delta}(y))|^2 \prod_{j=1}^3 \chi_{\delta}'(y_j){\rm d}y_j   \\
						&\lesssim \delta^2 \int_{A(0,1-1/\delta,1)} |\nabla f(x)|{\rm d}x.
\end{align*}
This ends the proof.
\end{proof}

Finally, we consider the Bogovskii operator on the annulus:
\begin{proposition}
There holds $C_{\mathfrak B}[\delta] \lesssim \delta^{9/2}$. Namely, given $f \in L^2_0(A(0,1-1/\delta,1))$ there exists $u \in H^1_0(A(0,1-1/\delta,1))$
such that
\begin{align*}
& {\rm div} \, u = f  \text{ on $A(0,1-1/\delta,1)$} \\
& \|\nabla u\|_{L^2(A(0,1-1/\delta,1))} \lesssim \delta^{9/2} \|f\|_{L^2(A(0,1-1/\delta,1))}
\end{align*}
\end{proposition}
\begin{proof}
We provide a proof by change of variable as for the previous proposition. Given $f \in L^2_0(A(0,1-1/\delta,1))$ we define
\[
\hat{f}(y) = \prod_{i=1}^{3} \chi_{\delta}'(y_i) f(X_{\delta}(y))\,, \quad \forall \, y \in A(0,1/2,1).
\]
Straightforward computations show that $\hat{f} \in L^2_0(A(0,1/2,1)).$ Consequently, there exists $\hat{u} \in H^1_0(A(0,1/2,1))$
such that:
\begin{align*}
& {\rm div} \, \hat{u} = \hat{f}  \text{ on $A(0,1/2,1)$} \\
& \|\nabla \hat{u}\|_{L^2(A(0,1/2,1))} \lesssim\|\hat{f}\|_{L^2(A(0,1/2,1))}.
\end{align*}
We set then:
\[
u(x) = \left( \prod_{\ell \neq i} \sigma_{\delta}'(x_\ell) \hat{u}_i(Y_{\delta}(x))\right)_{i=1,2,3} \quad \forall \, x \in A(0,1-1/\delta,1).
\]
Since $\sigma_{\delta}'(x_\ell) \chi_{\delta}'(\sigma_{\delta}(x_\ell)) = 1,$ we may expand the divergence to prove:
\[
{\rm div} u(x) = f(x)\,, \quad \forall \, x \in A(0,1-1/\delta,1).
\]
It is straightforward that $u=0$ on the boundaries of $A(0,1-1/\delta,1)$, and we are left with computing the size of its gradient.
We note that (introducing ${\rm Kron}$ the Kronecker symbol)
\[
\partial_{j} u_{i}(x) =  \sigma_{\delta}'(x_{j}) \left[\prod_{\ell \neq i} \sigma_{\delta}'(x_\ell)  \right] \partial_j \hat{u}_i(Y_{\delta}(x)) +  (1 - {\rm Kron}[j,i]) \sigma_{\delta}^{"}(x_{j}) \left[ \prod_{\ell \neq i,j}   \sigma_{\delta}'(x_\ell)\right]  \hat{u}_i(Y_{\delta}(x)).
\]
Consequently:
\begin{align*}
\int_{A(0,1-1/\delta,1)} |\partial_{j}u_i(x)|^2 & \lesssim \int_{A(0,1-1/\delta,1)} \left( \delta^{4}  |\partial_j \hat{u}_i(Y_{\delta}(x))|^2 +  \delta^9 |\hat{u}_i(Y_{\delta}(x))|^2 \right) \prod_{\ell=1}^3 \sigma'_{\delta}(x_\ell){\rm d}x_{\ell} \\
& \lesssim  \delta^{9}  \int_{A(0,1/2,1)} [ |\partial_j \hat{u}(y)|^2 + |\hat{u}(y)|^2] {\rm d}y.
\end{align*}
Here we apply the classical Poincar\'e inequality in $H^{1}_0(A(0,1/2,1))$ and the definition of $\hat{u}$, which yields
\[
\int_{A(0,1-1/\delta,1)} |\partial_{j}u_i(x)|^2  \lesssim \int_{A(0,1/2,1)} |\hat{f}(y)|^{2}{\rm d}y.
\]
We end up by dominating the right-hand side w.r.t. $f$ recalling the bound above for $\chi_{\delta}':$
\begin{align*}
 \int_{A(0,1/2,1)} |\hat{f}(y)|^{2}{\rm d}y & =  \int_{A(0,1/2,1)} \prod_{i=1}^{3} \chi_{\delta}'(y_i)|f(X_{\delta}(y_i))|^2 \prod_{i=1}^{3} \chi_{\delta}'(x_i){\rm d}x_i\,, \\
&   \lesssim \int_{A(0,1-1/\delta,1)} |f(x)|^2{\rm d}x.
\end{align*}
\end{proof}


\end{document}